\documentclass{article}
\usepackage{graphicx} 
\usepackage{mathrsfs}
\usepackage{tikz-cd}
\usepackage{hyperref}
\usepackage{amsmath}
\usepackage{quiver}
\usepackage{amsthm}
\usepackage{amssymb}
\usepackage{comment}
\usepackage[skip=10pt]{parskip}
\usepackage[left=2cm,right=2cm,top=2cm,bottom=2cm]{geometry}
\usepackage{ytableau}
\usepackage{multicol}
\renewcommand\labelenumi{\upshape(\roman{enumi})}
\renewcommand\theenumi\labelenumi

\newcommand{\N}{\mathbb{N}}

\newcommand{\T}{\text}

\usepackage{circledsteps}
\pgfkeys{/csteps/inner ysep=10pt}
\pgfkeys{/csteps/inner xsep=10pt}
\usepackage{tikz}
\usepackage{xcolor}
\usepackage{bm}
\usepackage{dashbox}

\usepackage{tikz}
\usepackage[dvipsnames]{xcolor}
\usepackage{bm}
\usepackage{mathdots}

\newcommand{\red}{\color{red}}
\newcommand{\blue}{\color{blue}}
\newcommand{\green}{\color{ForestGreen}}
\newcommand{\purple}{\color{Plum}}
\newcommand{\yellow}{\color{BurntOrange}}

\newtheorem{theorem}{Theorem}[section]
\newtheorem{lemma}[theorem]{Lemma}
\newtheorem{corollary}[theorem]{Corollary}
\newtheorem{proposition}[theorem]{Proposition}
\newtheorem{conjecture}[theorem]{Conjecture}
\newtheorem{claim}[theorem]{Claim}

\theoremstyle{definition} 
\newtheorem{example}[theorem]{Example}

\newtheorem{remark}[theorem]{Remark}
\newtheorem{definition}[theorem]{Definition}
\newtheorem{problem}[theorem]{Problem}

\newtheorem{fact}[theorem]{Fact}
\newcommand{\Dora}[1]{{\color{red} \sf $\clubsuit$ Dora: [#1]}}
\newcommand{\Rosa}[1]{{\color{blue} \sf $\clubsuit$ Rosa: [#1]}}

\title{Temperley-Lieb Immanants, Key Positivity, and Demazure Crystals}
\author{Rosa Paten and Dora Woodruff}
\date{}

\begin{document}

\maketitle

\begin{abstract}
    The main goal of this paper is to extend three important Schur positivity results to key positivity, replacing all Schur polynomials in relevant expressions with \textit{flagged} Schur polynomials. Namely, we first show that the Temperley-Lieb immanants of (many) flagged Jacobi-Trudi matrices are key positive. Using this result, we give a combinatorial rule for the key expansion of (most) products of flagged skew Schur polynomials, and also give a log concavity result inspired by that of Lam-Postnikov-Pylyavskyy. 

    The main tools in our proofs are Demazure crystals, and the recently defined shuffle tableaux of Nguyen and Pylyavskyy. In order to prove our main results, we must develop a new characterization of Demazure crystals, which builds off of prior work of Assaf and Gonzalez. This characterization may be useful in other contexts. 
\end{abstract}

\section{Introduction}\label{sec:intro}

\textit{Schur polynomials} $s_{\lambda}(x_1 \dots x_n)$ form a ubiquitous family of symmetric polynomials. Generalizing Schur polynomials are the \textit{flagged Schur polynomials}, denoted $s_{\lambda}^{\vec{b}}(x_1 \dots x_n)$, certain truncations of Schur polynomials which are no longer necessarily symmetric. Flagged Schur polynomials are themselves contained within \textit{key polynomials}. Key polynomials form a basis for $\mathbb{Z}[x_1, x_2, \dots]$ and arise as characters of certain submodules generated by extremal weight spaces under the Borel action of a Lie algebra \cite{demazure1974nouvelle}. 

It is therefore interesting to ask which important properties of Schur polynomials are also enjoyed by flagged Schur polynomials, and more generally by key polynomials. For example, the celebrated \textit{Littlewood-Richardson rule} \cite{littlewood1934group} tells us that the product of two Schur polynomials expands as a positive sum of Schur polynomials. This property does not extend to key polynomials (that is, the product of two key polynomials need not be positive in the key basis. See \cite{monksrulekey} for example). But, as we will see, this positivity result does extend to a large class of \textit{skew flagged Schur polynomials} (although not all). 

Schur polynomials can be expressed via the \textit{Jacobi-Trudi identities} \cite{gessel1989determinants}; Wachs \cite{wachs1985flagged} showed that flagged Schur polynomials satisfy similar such identities. A second important Schur positivity result is that the \textit{Temperley-Lieb immanants} of generalized Jacobi-Trudi matrices are always Schur positive \cite{rhoades2005temperley} \cite{rhoades2006kazhdan}. This fact has been used to prove several intriguing deep Schur positivity results (for example, see \cite{lam2007schur} and \cite{okounkov1997log}). Inspired by these results, we will also show that the Temperley-Lieb immanants of (many) \textit{flagged} Jacobi-Trudi matrices also expand nonnegatively in the key basis. 

The overall goal of this paper is threefold: first, to extend several Schur positivity results to key positivity, second, to push further the characterizations of Demazure crystals recently developed by Assaf and Gonzalez \cite{extremalidealprincipal}, and third, to demonstrate the utility of Nguyen and Pylyavskyy's shuffle tableaux \cite{nguyen2025temperley} as a model for Littlewood-Richardson coefficients. 

\subsection{Main results and organization}

Our first result is that the Temperley-Lieb immanants of \textit{most} flagged Jacobi-Trudi matrices are key positive. In what follows, a partition $\lambda$ is called \textit{strict} if $\lambda_i > \lambda_{i+1}$ for all $i \leq n$. 
\vspace{1em}

\begin{theorem}\label{thm:TL}
    Let $\mu$ be a partition, $\lambda$ be a \textit{strict} partition, and $\vec{b}$ be a nondecreasing flag. Then, all Temperley-Lieb immanants of the flagged Jacobi-Trudi matrix $A_{\lambda, \mu}^{\vec{b}}$ are key nonnegative. 
\end{theorem}

We note that the strictness of the partition $\lambda$ in Theorem \ref{thm:TL} is a necessary assumption; for $\lambda$ nonstrict, it is not hard to find counterexamples. (This is a consequence of the fact that \textit{Demazure subsets} of crystals are quite subtle; see subsection \ref{subsection:demazurecrystals} for some further discussion). 

\vspace{1em}
\begin{example}[Counterexample for nonstrict $\lambda$]
    Let us take $\lambda = (7,7,7,5)$, $\mu = (6,4,3,2)$, and $\vec{b} = (1,2,2,3)$. Here, $\lambda$ is not strict, so Theorem \ref{thm:TL} should not apply. Indeed, we can compute that the Temperley-Lieb immanant corresponding to the matching depicted below is 
    \[k_{3,6,2} + k_{3,7,1} + k_{3,8} + k_{4,5,2} + k_{4,6,1} + k_{4,7} - k_{5,4,2} - k_{6,3,2} - k_{6,4,1} - k_{7,3,1} - k_{7, 4} - k_{8,3}\]
    which is far from being key nonnegative. However, when we change $\lambda$ to $(8,7,6,5)$ so that it is now strict, the immanant indexed by the same diagram becomes
    \[k_{5,4,2} + k_{5,5,1} + k_{5,6} + k_{6,3,2} + k_{6,4,1} + k_{7,3,1} + k_{7,4} + k_{8,3}\]
    Magically, this expression is now key positive!
    (This example was computed with the aid of SageMath \cite{sagemath}). 
    
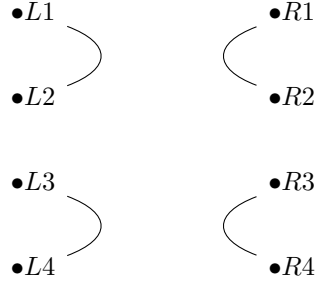
    \begin{center}
        \begin{tikzcd}
	{\bullet L1} &&& {\bullet R1} \\
	{\bullet L2} &&& {\bullet R2} \\
	{\bullet L3} &&& {\bullet R3} \\
	{\bullet L4} &&& {\bullet R4}
	\arrow[curve={height=-30pt}, no head, from=1-1, to=2-1]
	\arrow[curve={height=30pt}, no head, from=1-4, to=2-4]
	\arrow[curve={height=-30pt}, no head, from=3-1, to=4-1]
	\arrow[curve={height=30pt}, no head, from=3-4, to=4-4]
\end{tikzcd}
    \end{center}
\end{example}

Theorem \ref{thm:TL} has several nice applications. We use it to prove that (most) products of flagged skew Schur polynomials are key positive and give a combinatorial rule for their expansion, building on a theorem of Reiner-Shimozono \cite{reiner1995key}. 
\vspace{1em}

\begin{theorem}\label{thm:flaggedLR}
    Let $\lambda/\mu$ and $\nu/\rho$ be skew shapes satisfying the same two properties:
    \begin{enumerate}
        \item For all $i$, $\lambda_i \neq \nu_i$
        \item $\lambda$ and $\nu$ are \textit{interlacing}: that is, $\lambda_i \geq \nu_{i+1}$ and $\nu_i \geq \lambda_{i+1}$. 
    \end{enumerate}

    Then, $s_{\lambda/\mu}^{\vec{b}} s_{\nu/\rho}^{\vec{b}}$ is key positive. Furthermore, there is a relatively simple combinatorial algorithm for its key expansion. 
\end{theorem}

Interestingly, for arbitrary $\lambda$ and $\nu$, this product also need not be key positive. Without these hypotheses, we have found several counterexamples. 

The starting point of this work was to find an analogue of Lam-Postnikov-Pylyavskyy's \cite{lam2007schur} Schur log concavity conjecture for key positivity; an important case of this conjecture follows directly from Theorem \ref{thm:TL}.
\vspace{1em}

\begin{corollary}\label{cor:logconcave}
    Let $\lambda/\mu$ and $\nu/\rho$ be skew shapes satisfying the same two properties as in Theorem \ref{thm:flaggedLR}.

    Then, the difference 

    \[s^{\vec{b}}_{\lambda/\mu \vee \nu/\rho} s^{\vec{b}}_{\lambda/\mu \wedge \nu/\rho} - s^{\vec{b}}_{\lambda/\mu}s^{\vec{b}}_{\nu/\rho}\]

    is key nonnegative. 

\end{corollary}

In our attempt to prove Theorem \ref{thm:TL}, we also further develop Assaf and Gonzalez's recent characterization of Demazure crystals, providing a more local and user-friendly set of axioms.

The paper is organized as follows. In Section \ref{sec:prelims}, we describe important background, focusing particularly on crystals, shuffle tableaux, and Demazure subsets. In Section \ref{sec:characterization}, we state our characterization of Demazure crystals and demonstrate our axioms' usability by applying them to flagged skew tableaux. (We note that a couple of technical lemmas that apply to both usual skew tableaux and to shuffle tableaux are delayed to section \ref{sec:principality}). In Sections \ref{sec:extremal+ideal} and \ref{sec:principality}, we prove that our Demazure crystal axioms hold for flagged shuffle tableaux (subject to the strictness condition in Theorem \ref{thm:TL}), completing the proof of Theorem \ref{thm:TL}. In Section \ref{sec:applications}, we prove our applications of Theorem \ref{thm:TL}. Finally, we give a full proof of our characterization of Demazure crystals (based on an approach of Assaf and Gonzalez \cite{extremalidealprincipal}) in an appendix \ref{sec:appendix}. 

\subsection{Acknowledgements}
    This project was started through the MIT SPUR program and continued through an MIT UROP. We  thank the organizers of SPUR, Roman Bezrukavnikov and Jonathan Bloom, for their encouragement and support. We thank Son Nguyen, Pasha Pylyavskyy, and Brendan Rhoades for interesting conversations about their work with Temperley-Lieb immanants. Finally, we especially thank Sami Assaf and Nicolle Gonzalez for their helpful correspondences on Demazure crystals.

\section{Preliminaries}\label{sec:prelims}

\subsection{Skew Schur polynomials and tableau conventions}

A partition $\lambda$, represented pictorially by a \textit{Young diagram}, is just a nondecreasing sequence of positive integers $(\lambda_1, \lambda_2 \dots \lambda_n)$. A \textit{skew diagram} is indexed by a pair of partitions $\lambda/\mu$, where $\lambda_i \geq \mu_i$ for all $i$. The diagram is obtained by removing the boxes of $\mu$ from the boxes of $\lambda$. 
\vspace{1em}

\begin{example}
Below is the skew diagram $(4, 2, 2)/ (2, 1, 0)$: 
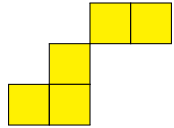
\begin{center}

    \begin{ytableau}
        \none & \none & *(yellow) & *(yellow) \\
        \none & *(yellow) \\
        *(yellow) & *(yellow)
    \end{ytableau}
    \end{center}
\end{example}

A \textit{semistandard tableau} on the diagram $\lambda/ \mu$ is a filling of the boxes of the diagram with positive integers, such that labels weakly increase left to right and strictly increase from top to bottom. The \textit{weight} $\text{wt}(T)$ of a tableau $T$ is the vector $(w_1, w_2 \dots)$, where $w_i$ is the numbers of entries $i$ in $T$. 

\begin{example}
Below is a skew tableau of shape $\lambda/\mu$ with weight $(1, 3, 1)$. 

    
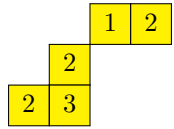
\begin{center}

    \begin{ytableau}
        \none & \none & *(yellow) $1$ & *(yellow) $2$ \\
        \none & *(yellow) $2$ \\
        *(yellow) $2$ & *(yellow) $3$
    \end{ytableau}
    \end{center}
\end{example}

The set of semistandard tableaux of shape $\lambda/\mu$ is denoted $\text{SSYT}(\lambda/\mu)$ (or $\text{SSYT}(\lambda)$ if $\mu$ is empty). The skew Schur polynomial $s_{\lambda/\mu}$ is defined as 

\[s_{\lambda/\mu}(x_1 \dots x_n) = \sum_{T \in \text{SSYT}(\lambda/\mu)} \bf{x}^{\text{wt}(T)}\]

\subsection{Flagged Schur polynomials}

A \textit{flag} $\vec{b}$ is a vector $(b_1 \dots b_n)$ of nondecreasing positive integers. We say that a tableau $T$ of shape $\lambda$ is \textit{flagged by $\vec{b}$} if every entry in row $i$ of $T$ is bounded above by $b_i$.  
\vspace{1em}

\begin{definition}
    For a flag $\vec{b}$, the \textit{flagged Schur polynomial} $s_{\lambda}^{\vec{b}}$ is given by 
    \[\sum_{T \in \text{SSYT}(\lambda) \text{ flagged by }\vec{b}} \text{wt}(T)\]
\end{definition}

All the same definitions apply to a skew shape $\lambda/\mu$ rather than a straight shape. Flagged Schur polynomials generalize Schur polynomials; to recover $s_{\lambda}$, just let $\vec{b} = (n, n, n\dots)$. While Schur polynomials $s_{\lambda}$ are always symmetric in $x_1 \dots x_{|\lambda|}$, flagged Schur polynomials need not be.

\subsection{Key polynomials}

Key polynomials are indexed by \textit{weak compositions} $\alpha = (\alpha_1, \alpha_2 \dots)$. Alternatively, they are indexed by a pair $(\lambda, w)$, where $\lambda$ is a (nonstrict) partition and $w$ is a permutation (although there may be multiple possible choices for $w$). When $w$ is the longest permutation $w_0$, then $\kappa_{\lambda, w_0}$ is the Schur polynomial $s_{\lambda}$. When $w$ is the identity on the other hand, $\kappa_{\lambda, w}$ is the single monomial $x^{\lambda}$. In general, key polynomials can be regarded as certain truncations of Schur polynomials. Every flagged Schur polynomial is also a key polynomial. 

Importantly, as $\alpha$ ranges over all compositions, $\{\kappa_{\alpha}\}$ forms a basis for the polynomial ring $\mathbb{Z}[x_1, x_2 \dots]$. Since flagged Schur polynomials are not symmetric polynomials, it does not usually make sense to ask about the Schur positivity of expressions involving flagged Schur polynomials. \textit{Key positivity} turns out to be the right notion for us instead. 

There are many combinatorial models for key polynomials (in terms of Kohnert diagrams \cite{kohnert},  tableaux \cite{keytableaux}, and more). However, we will always be interpreting them as \textit{characters of Demazure crystals}; see subsection \ref{subsection:demazurecrystals}.   

\subsection{Jacobi-Trudi matrices}

A \textit{generalized Jacobi-Trudi matrix} $A_{\lambda, \mu}$ is an $n \times n$ matrix of the form 

\[(h_{\lambda_i-\mu_j})^n_{i, j = 1}\]

for partitions $\lambda, \mu$. Here, $h_{\lambda_i- \mu_j}$ is a homogeneous symmetric function.

The \textit{Jacobi-Trudi identity} expresses skew Schur polynomials as minors of generalized Jacobi-Trudi matrices. Specifically, if $\Delta_{I,J}(M)$ denotes the minor with row set $I$ and column set $J$, we have 
\[s_{\lambda/\mu} = \Delta_{I, J}(A_{\lambda, \mu})\]
Here, the associated row and column sets are given by $I = \{\mu_k + 1 < \mu_{k-1} + 2 < \dots < \mu_1 + k\}$ and $J = \{\lambda_k+1 < \lambda_{k-1}+2 \dots < \lambda_1+k\}$. 

Wachs showed that we can similarly express flagged Schur polynomials $s_{\lambda/\mu}^{\vec{b}}$ as minors of a \textit{flagged} Jacobi-Trudi matrix $A^{\vec{b}}_{\lambda, \mu}$. The flagged Jacobi-Trudi matrix $A^{\vec{b}}_{\lambda, \mu}$ is obtained by replacing each $h_{\lambda_i - \mu_j}$ in row $i$ with a homogeneous symmetric polynomial \textit{in $b_i$ variables} \cite{wachs1985flagged}.

\subsection{Temperley-Lieb Immanants}

As mentioned above, the determinant of any (generalized) Jacobi-Trudi matrix is a skew Schur polynomial, hence Schur positive. Determinants are examples of elements of Lusztig's \textit{dual canonical bases} in type $A$ \cite{lusztig1990canonical}. Rhoades-Skandera \cite{rhoades2006kazhdan} showed that \textit{all} elements of the dual canonical bases evaluate to Schur nonnegative expressions on generalized Jacobi-Trudi matrices. Their proof relies on deep earlier work of Haiman \cite{haiman1993hecke}. 

In general, there is no known combinatorial proof of this fact. But for a subset of dual canonical bases known as \textit{Temperley-Lieb immanants}, introduced by Rhoades-Skandera \cite{rhoades2005temperley}, there is a combinatorial proof by Nguyen-Pylyavskyy \cite{nguyen2025temperley}. These Temperley-Lieb immanants have been previously used to prove intriguing Schur positivity facts, such as a log concavity inequality of Lam-Postnikov-Pylyavskyy \cite{lam2007schur}. 

We do not define Temperley-Lieb immanants here, although we will briefly describe the combinatorial rule given by Nguyen-Pylyavskyy in the next section. We will only recall that an immanant is generally a function $I$ evaluated on matrices of the form 

\[I(M) = \sum_{w \in S_n}  f(w) \prod_{i=1}^n m_{i, w(i)}\]

for some function $f: S_n \to \mathbb{C}$. For the determinant of a matrix, $f$ is simply the sign function. Meanwhile, the permanent of a matrix is given by taking $f$ to be constant at $1$. A certain class of functions $f: S_n \to \mathbb{C}$ defined via the \textit{Temperley-Lieb algebra} give us Temperley-Lieb immanants.

\subsection{Shuffle Tableaux}

Given two skew diagrams $\lambda/\mu$ and $\nu/\rho$, their \textit{shuffle diagram} $(\lambda/\mu)\circledast(\nu/\rho)$ is the diagram we obtain by `interlacing the rows of the two shapes' (informally. For more details, see \cite{nguyen2025temperley}).  
\vspace{1em}
\begin{example}
    Let $\lambda = (3,2)$ and $\mu = (2,1)$. The shuffle diagram is 
    
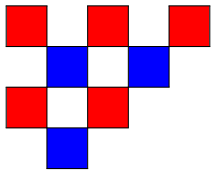
\begin{center}
   \begin{ytableau}
       *(red) & \none & *(red) & \none & *(red) \\
  \none  & *(blue) & \none & *(blue)\\
  *(red) & \none & *(red)\\
  \none & *(blue)

\end{ytableau}
\end{center}

  The shaded boxes are boxes of the tableaux; the unshaded area is all `empty.' We only label the shaded squares when writing down tableaux of shape $(\lambda/\mu)\circledast(\nu/\rho)$.  

\end{example}

Shuffle tableaux were first defined by Nguyen-Pylyavskyy \cite{nguyen2025temperley} in order to study Temperley-Lieb immanants. A shuffle tableau is just a labeling of the boxes of $(\lambda/\mu)\circledast(\nu/\rho)$ satisfying the usual rules - labels increase weakly along rows and strictly down columns. So, a shuffle tableau simply corresponds to a pair of tableaux, one of shape $\lambda/\mu$, one of shape $\nu/\rho$. 

For Nguyen and Pylyavskyy, the magic of shuffle tableaux was that they can coherently be assigned `Temperley-Lieb types.' Using this idea, they give a combinatorial rule for Schur expansions of Temperley-Lieb immanants. For us, the magic of shuffle tableaux is also that they `behave well under flagging.' 

\subsection{Permutation conventions}

Throughout this paper, we will use $s_i$ to denote the simple transposition $(i, i+1)$. A \textit{reduced word} for a permutation $w$ is a factorization $w = s_{i_1}s_{i_2} \dots s_{i_k}$ of minimal length. If $\vec{v}$ is a vector of size $n$, then $s_i(\vec{v})$ is obtained by swapping the $i$th and $i+1$st entries of $\vec{v}$, and $w(\vec{v})$ is defined similarly for an arbitrary permutation $w$. 
\vspace{1em}

\begin{definition}\label{def:bruhatorder}
    Let $u, w$ be permutations. Then, $u \leq w$ in the \textit{Bruhat order} if and only if there exist reduced words $u = s_{i_1} s_{i_2} \dots s_{i_k}$ and $w = s_{j_1} s_{j_2} \dots s_{j_m}$ such that the word for $u$ is a subexpression of the word for $w$. 
\end{definition}
\vspace{1em}
\begin{remark}
    In fact, $u\leq w$ in the Bruhat order if and only if, \textit{for any} reduced word for $w$, there exists a reduced word for $u$ which is a subexpression of the reduced word for $w$.
\end{remark}

When we consider permutations $w(\vec{v})$ of vectors, the Bruhat order translates to the \textit{dominance order}. Recall that, if $\vec{v}$ and $\vec{w}$ are vectors, we say $\vec{v} \leq \vec{w}$ in the dominance order if, for all $i$, $v_1 + v_2 \dots + v_i \leq w_1 + w_2 \dots w_i$. That is:
\vspace{1em}

\begin{fact}
    Let $\vec{v}$ be a partition vector (that is, $v_i \geq v_{i+1}$ for all $i$) and $u, w$ be permutations. Then, $u \leq w$ in the Bruhat order if and only if $u(\vec{v}) \leq w(\vec{v})$ in the dominance order. 
\end{fact}

Finally, the symmetric group $S_n$ is generated by simple transpositions, with relations given by the ubiquitous \textit{braid relations}:
\vspace{1em}

\begin{definition}[Braid relations]\label{def:braidrelations}
    \item For all $i$, $s_i^2 = \text{id}$
    \item For $|i-j| > 1$, $s_is_j = s_js_i$
    \item For $|i-j|=1$, $s_is_js_i = s_js_is_j$
\end{definition}

Many of our proofs will rely on lengthy manipulations with these braid relations.

\subsection{Crystal conventions}

Kashiwara's theory of crystal bases \cite{crystals!} is a powerful tool for proving Schur positivity. In this section, we give a very brief overview of the basic definitions. 

Let $e_1, e_2 \dots e_n$ denote the standard basis for $\mathbb{R}^n$. A \textit{$\mathfrak{gl}_n$-crystal of dimension $n$} consists of the following data:

\begin{enumerate}
    \item A set $\mathcal{B}$ of objects not containing $0$ 
    \item Crystal operators $e_i, f_i: \mathcal{B} \to \mathcal{B} \cup \{0\}$ for $i = 1, 2 \dots n-1$ and
    \item A \textit{weight map} $\text{wt}: \mathcal{B} \to \mathbb{Z}^n$ satisfying $\text{wt}(e_i(b)) = \text{wt}(b) + (e_i-e_{i+1})$
\end{enumerate}

The operators $e_i, f_i$ are called \textit{raising and lowering operators} respectively, and they are always partial inverses: that is, for $b, b' \in \mathcal{B}$, $e_i(b) = b'$ if and only if $f_i(b') = b$. (In this paper, all $\mathfrak{gl}_n$-crystals considered will be finite and normal). 

We will often abuse notation, referring to the set $\mathcal{B}$ as `the crystal' when the underlying crystal operators and weight function are clear from context. 

The set of weight vectors in $\mathbb{Z}^n$ is partially ordered under the \textit{dominance order}. Our crystals will always contain a unique \textit{highest weight element}, often denoted $b_{\lambda}$. This highest weight element satisfies $e_i(b_{\lambda}) = 0$ for all $i = 1, 2 \dots n-1$. 

The \textit{crystal graph} of $\mathcal{B}$ is obtained by taking objects in $\mathcal{B}$ to be vertices, with directed edges given by the lowering operators $f_i$ (pointing from $x \in \mathcal{B}$ to $f_i(x)$). Edges to $0$ are not drawn. If the resulting graph is connected (as an undirected graph), we call the crystal \textit{connected}. Any subset $X \subseteq \mathcal{B}$ comes equipped with the same raising and lowering operators and weight function as $\mathcal{B}$. 
\vspace{1em}
\begin{definition}
    The \textit{character} of a crystal $\mathcal{B}$ is the polynomial 
    \[\sum_{b \in \mathcal{B}} x_1^{\text{wt}(b)_1}x_2^{\text{wt}(b)_2} \dots x_n^{\text{wt}(b)_n}\]

    (The same definition applies for a subset of a crystal). 
\end{definition}

When $\mathcal{B}$ is connected and normal, its character is always a Schur polynomial $s_{\lambda}$. More generally, the character of a (normal) crystal $\mathcal{B}$ is given by 
\[\sum_{b \in \mathcal{B} \text{ highest weight} s_{\text{wt}(b)}}(x_1 \dots x_n)\]

Interpreting a polynomial as a character of a crystal is thus a useful way to prove its Schur positivity. 

An \textit{$i$-string} in a crystal graph is a list of elements $b_1, b_2 \dots b_k \in \mathcal{B}$ such that $b_j = f_i(b_{j-1})$, $f_i(b_k) =0$, and $e_i(b_1) = 0$. An element $b \in \mathcal{B}$ is called \textit{extremal} if its weight is a permutation of $\text{wt}(b_{\lambda})$ (with $b_{\lambda}$ the highest weight element). Extremal weight elements sit at the endpoints of all of their $i$-strings. 

For any partition $\lambda$, there is a corresponding crystal $\mathcal{B}(\lambda)$, with underlying set given by $\text{SSYT}(\lambda)$ (semistandard Young tableaux of shape $\lambda$ with entries in $[n]$). Kashiwara, Nakashima \cite{crystalgraphs1} and Littlemann \cite{crystalgraphs2} gave an explicit combinatorial construction of the crystal graph on tableaux: 

\vspace{1em}

\begin{definition}
    The operator $e_i$ acts on a Young tableau $x$ as follows. First, record the \textit{reading word} of $x$ (by reading off its entries from left to right, bottom to top). 
    
    Next, write a closing parenthesis for every entry $i$ and an opening parenthesis for every entry $i+1$. Match these parentheses in the usual way. 

    Then, $e_i$ changes the leftmost umatched $i+1$ to an $i$. If there is no such $i+1$, then $e_i(x) = 0$. 
\end{definition}
The operator $f_i$ is defined similarly; the only difference is the last step, where $f_i$ changes the rightmost unmatched $i$ to an $i+1$. 

Finally, we will use the following useful notation throughout this paper:

\begin{enumerate}
    \item The expression $f_i^*(x)$ means that we apply $f_i$ to $x$ as many times as possible, until another application would yield $0$. The operator $e_i^*$ is defined similarly. 
    \item If $i > a$, the expression $F_{i, a}(x)$
    denotes 
    \[f_i^*f_{i-1}^* \dots f_{a}^*(x)\]
    The \textit{length} of the expression $F_{i, a}$ is $|i-a+1|$ (the number of $f_j^*$s represented by $F_{i, a}$).  
    \item As a convention, $F_{i, 0}$ for any $i$ is always the \textit{empty expression}
    \item Whenever we mention an `unpaired entry $i$' in a tableau, unless otherwise stated, we always mean that the $i$ is not paired with any $i+1$, \textit{not} that the $i$ is not paired with any $i-1$
\end{enumerate}

It is natural to consider $f_i^*$ because it acts on the content of a tableau $x$ by transposing entries $i$ and $i+1$; thus, the operators $\{f_i^*\}_i$ satisfy the braid relations of Definition \ref{def:braidrelations}. That is:
\vspace{1em}
\begin{definition}[Braid relations for operators]
Let $x\in X$ be extremal.
\begin{enumerate}
    \item For $|i-j|>1$, $f_i^*f_j^*(x) = f_j^*f_i^*(x)$. We refer to throughout to this relation as the \textit{first} braid relation.
    \item For $|i-j|=1$, $f_i^*f_j^*f_i^*(x) = f_j^*f_i^*f_j^*(x)$. We refer to this as the \textit{second} braid relation.
\end{enumerate}

As a convention, we also say that applying $f_i^*$ twice in a row to $x$ yields $0$. 

\end{definition}

To see why operators of the form $F_{i,a_i}$ are natural to consider, see Proposition \ref{rmk:algorithm}. 

\subsection{Crystal operators and shuffle tableaux}

In order to prove their refined Littlewood-Richardson rule, Nguyen and Pylyavskyy showed that shuffle tableaux assemble themselves into crystals:
\vspace{1em}
\begin{theorem}[Nguyen-Pylyavskyy \cite{nguyen2025temperley}]
    There exist natural raising and lowering operators $e_i, f_i$ such that the set of shuffle tableaux of shape $(\lambda/\mu)\circledast(\nu/\rho)$ assemble into type $A$ crystals. Furthermore, these crystal operators preserve the Temperley-Lieb type of the tableaux.  
\end{theorem}

We will briefly describe their raising and lowering operators here. If an $i$ appears directly underneath an $i+1$ in a shuffle tableau, the $i$ and the $i+1$ are called \textit{column paired}. The \textit{reading word} of a shuffle tableau is, as usual, the word obtained by reading entries from left to right, and bottom to top. 

In order to apply $e_i$ to a shuffle tableau $x$, perform the following:

\begin{enumerate}
    \item Find all instances of column paired $i$s and $i+1$s
    \item Remove these pairs from the reading word of $x$
    \item Match $i$s and $i+1$s in the remaining word in the usual way (by sending $i$s to closing parentheses and $i+1$s to opening parentheses) 
    \item $e_i$ changes the leftmost unmatched $i+1$ to an $i$. If there is no unmatched $i+1$, then $e_i(x) = 0$. 
\end{enumerate}

The lowering operator $f_{i}$ is defined similarly; the only difference is in the fourth step, where $f_{i}$ changes the rightmost unmatched $i$ to an $i+1$. 
\vspace{1em}
\begin{example}
    Let's apply $f_1$ to the following shuffle tableau: 
        \begin{center}
   \begin{ytableau}
        *(pink)$1$ & \none & *(pink)$1$ & \none & $1$ \\
  \none  & $1$ & \none & $2$\\
  *(pink)$2$ & \none & *(pink)$2$\\
  \none & $3$
  
\end{ytableau}
\end{center}

The pink entries are column paired. The reading word of this tableau is $32212111$, but after removing the column paired $1$s and $2$s, it becomes $3121$. There is only one unpaired $1$, and it raises to a $2$. The result is: 

    \begin{center}
   \begin{ytableau}
        $1$ & \none & $1$ & \none & $1$ \\
  \none  & $2$ & \none & $2$\\
  $2$ & \none & $2$\\
  \none & $3$
  
\end{ytableau}
\end{center}

Let's apply $e_1$ to this resulting tableau. The same $1$s and $2$s are still column paired, so the reading word with those entries removed is $3221$. The leftmost $2$ is unpaired in this word, and gets lowered back to a $1$. The result is the tableau we started with, which is what we expect.

\end{example}

In a nutshell, crystal operators for shuffle tableaux are very similar to crystal operators for usual tableaux, with the difference that they `prioritize' column pairing. 
\vspace{1em}
\begin{remark}
Recall from \cite{nguyen2025temperley} that we can define the Temperley-Lieb type of a shuffle tableau as equal to the Temperley-Lieb type of the corresponding wiring diagram (rows in the tableau correspond to paths in the diagram). For flagged shuffle tableaux, we define the corresponding wiring diagram slightly differently: rather than extending the paths infinitely vertically, we 'flag' the paths so that the path corresponding to row $j$ only extends to $b(j)$. For a given flagged shuffle tableau, we can recover the non-flagged TL type by extending each path vertically. 

As long as $\lambda$ is strict, extending the diagram will not change the TL type (all crossings are preserved and no crossings are added). Thus for $\lambda$ strict, we obtain as a direct corollary from \cite{nguyen2025temperley} that the crystal operators preserve TL type for flagged tableaux. If $\lambda$ is not strict however, the crystal operators do not in general preserve TL type. We refer the reader to section $3$ of \cite{nguyen2025temperley} for more details on the Temperley-Lieb type of shuffle tableaux. 
\end{remark}
\vspace{1em}
We use the following simple lemma throughout our proof that flagged shuffle tableaux form Demazure crystals. 
\vspace{1em}
\begin{lemma}\label{lemma:subsetlemma}
    Let $x\in X$. Suppose that $f_j(x)\not=0$. Then, for all $i\not = j$, the set of unpaired $i$s in $x$ is a subset of the set of unpaired $i$s in $f_j(x)$
\end{lemma}
\begin{proof}
    This follows since applying $f_i$ only creates $i+1$s, which do not pair with $j$s since $i\not = j$. 
\end{proof}

\subsection{Demazure crystals}\label{subsection:demazurecrystals}

This section briefly summarizes Demazure's formula for the characters of Demazure modules \cite{demazure1974nouvelle}.  
\vspace{1em}

\begin{definition} 
Let $\mathcal{B}$ be a highest weight crystal with highest weight element $b_{\lambda}$, and let $w \in S_n$ be a Weyl element. Pick \textit{any} reduced expression $s_{i_1}s_{i_2} \dots s_{i_n}$ for $w$. Then, the \textit{Demazure crystal} $\mathcal{B}_w$ is 
\[\bigcup_{m_i \in \mathbb{N}} f_{i_1}^{m_1}f_{i_2}^{m_2} \dots f_{i_n}^{m_n}(b_{\lambda})\]
\end{definition}

A subset $X \subseteq \mathcal{B}$ is called a Demazure crystal, or a Demazure subset, if $X = \mathcal{B}_w$ for some $w \in S_n$. 
\vspace{1em}

\begin{example}
    Let $w = \text{id}$. Then, the empty word is the only reduced expression for $w$, and so $\mathcal{B}_w = \{b_{\lambda}\}$. 

    On the other hand, if $w$ is the longest element of $S_n$, we have $\mathcal{B}_w = \mathcal{B}$. 
\end{example}
\vspace{1em}

\begin{theorem}[\cite{demazure1974nouvelle}]\label{thm:demazure}
    The Demazure crystal $\mathcal{B}_w$ is well-defined (that is, it does not depend on the choice of reduced expression for $w$). Its character is the character of the Demazure module $V_w(\lambda)$, which is the key polynomial $\kappa_{w(\lambda)}$. 
\end{theorem}

Demazure crystals are deeply connected to the Bruhat order. In particular, two extremal weight vectors $u\lambda$ and $w\lambda$ satisfy $u\lambda \leq w\lambda$ (in the \textit{dominance order}) if and only if $u \leq w$ in the Bruhat order. With Definition \ref{def:bruhatorder} in mind, the extremal elements of $\mathcal{B}_w(\lambda)$ all have weight vectors $u\lambda$, where $u \leq w$. 

In light of Theorem \ref{thm:demazure}, our goal will be to show that flagged shuffle tableaux (given our strictness condition) assemble themselves into disjoint Demazure crystals. 
\vspace{1em}
\begin{example}\label{example:demazure}

Below is the crystal graph on $\text{SSYT}(2,1)$ with entries in $\{1,2,3\}$ and edges labeled by lowering operators. The yellow highlighted tableaux give the Demazure subset defined by the flag $(2,3)$. The corresponding permutation is $s_1s_2$, so our Demazure subset is $\mathcal{B}(2,1)_{s_1s_2}$. Summing the contents of all the tableaux in $\text{SSYT}(2,1)$ gives us $s_{(2,1)}$. Summing the contents of the yellow subset gives us $s_{(2,1)}^{(2,3)}$, which is also the key polynomial $\kappa_{(2,1), s_1s_2}$.

\begin{center}
        \begin{tikzpicture}
    \node(A) at (8,8) {\begin{ytableau}
        *(yellow)1&*(yellow)1\\
        *(yellow)2
    \end{ytableau}};
    
    \node(B) at (6,6.5) {\begin{ytableau}
        *(yellow)1&*(yellow)2\\
        *(yellow)2
    \end{ytableau}};

    \node(B2) at (6,4.5) {\begin{ytableau}
        1&3\\
        2
    \end{ytableau}};

    \node(B3) at (6,2.5) {\begin{ytableau}
        1&3\\
        3
    \end{ytableau}};

    \node(C) at (10,6.5) {\begin{ytableau}
        *(yellow)1&*(yellow)1\\
        *(yellow)3
    \end{ytableau}};

    \node(C2) at (10,4.5) {\begin{ytableau}
        *(yellow)1&*(yellow)2\\
        *(yellow)3
    \end{ytableau}};

    \node(C3) at (10,2.5) {\begin{ytableau}
        *(yellow)2&*(yellow)2\\
        *(yellow)3
    \end{ytableau}};

    \node(D) at (8,1.5) {\begin{ytableau}
        2&3\\
        3
    \end{ytableau}};
    
    \draw[->] (A) -- (B) node[midway, above] {$f_1$};
    \draw[->] (B) -- (B2) node[midway, left] {$f_2$};
    \draw[->] (B2) -- (B3) node[midway, left] {$f_2$};
    \draw[->] (A) -- (C) node[midway, above] {$f_2$};
    \draw[->] (C) -- (C2) node[midway, right] {$f_1$};
    \draw[->] (C2) -- (C3) node[midway, right] {$f_1$};
    \draw[->] (B3) -- (D) node[midway, below] {$f_1$};
    \draw[->] (C3) -- (D) node[midway, below] {$f_2$};
\end{tikzpicture}
\end{center}
\end{example}

It is clear from the definition that a Demazure crystal must include the highest weight element $b_{\lambda}$ and be closed under applications of $e_i$; it is also clear that the subset of shuffle tableaux flagged by $\vec{b}$, if nonempty, satisfies these two properties. However, these two conditions are wildly short of sufficient to show that flagged shuffle tableaux form Demazure crystals. 

\section{Characterizing Demazure Crystals}\label{sec:characterization}

Our characterization of Demazure crystals is very similar to the characterization recently given by Assaf and Gonzalez \cite{extremalidealprincipal}. The authors describe three properties - \textit{extremal, ideal, principal} - which are necessary and sufficient for $X \subseteq \mathcal{B}$ to be a Demazure subset. In practice, showing that $X \subseteq \mathcal{B}$ is extremal and ideal seems quite doable, whereas it is often difficult to show that $X$ is principal. Our main contribution in this section is a further breakdown of this principality condition. 

\subsection{Statement of our characterization}\label{subsec:characterizationstatement}

First, we will restate the characterization given by Assaf and Gonzalez:
\vspace{1em}

\begin{theorem}[Assaf-Gonzalez \cite{extremalidealprincipal}]\label{thm:characterization}
    A subset $X$ of a crystal $\mathcal{B}$ is a Demazure subset if and only if $X$ satisfies the following three properties:

    \begin{enumerate}
        \item $X$ is extremal: for any $i$-string $S$ in $\mathcal{B}$, either $S \cap X = \emptyset, S$, or $\{b\}$, where $e_i(b) = 0$
        \item $X$ is ideal: if $x, y \in X$ are extremal elements and 
        \[x = e_j^* e_{i_1}^* \dots e_{i_m}^*(y)\]
        then 
        \[f_{i_m}^* \dots f_{i_1}^*(x) \in X \sqcup \{0\}\]
        \item $X$ is principal: for every connected component $X'$ of $X$, there exists an extremal $x \in X'$ which is \textit{lowest weight}. That is, if $y \in X'$ is extremal, then $\text{wt}(y) \leq \text{wt}(x)$ (in dominance order). 
    \end{enumerate}
\end{theorem}

If $X$ is extremal and ideal, then $X$ is a \textit{union} of Demazure crystals. However, to show that $X$ is a single Demazure crystal, we must show that it is principal. Principality is often difficult to prove because we are asked to locate such a lowest weight $x$, and because the condition is global. The following, more local characterization decomposes principality into two somewhat more tractable properties:
\vspace{1em}

\begin{theorem}\label{thm:principalmoreusable}
        Let $X$ be a connected union of Demazure crystals. Then $X$ is a Demazure crystal iff $X$ satisfies the following two properties, which we call \textit{extension} and \textit{gluing}: 
        Let $n,a,b\in \N$.
    \begin{enumerate}
        \item (Extension) Let $x\in X$ be extremal such that $e_i(x)=0$ for all $i\leq n+1$. Suppose that $F_{n,a}F_{n+1,b}x,F_{n+1,b-1}(x)\in X$ with $a \geq b$. Then $F_{n,a}F_{n+1,b-1}(x)\in X$. 
        \item (Gluing) Let $k,m,a$ be integers, and let $x\in X$ be extremal so that:
        \begin{enumerate}
            \item $e_i(x)=0$ for all $a-m\leq i\leq a+k$
            \item $f_{a+k}^*\dots f_{a}^*(x)\in X$ and
            \item $f_{a-m}^*\dots f_{a-1}^*(x)\in X$
        \end{enumerate} Then, either $f_{a+k}^*\dots f_a^*f_{a-1}^*(x)\in X$ or $f_{a-m}^*\dots f_{a-1}^*f_a^*(x)\in X$. 
\end{enumerate}
\end{theorem}
\vspace{1em}

The statement of Theorem \ref{thm:principalmoreusable} may seem strange. It is motivated by the following fact (which we believe to be well-known; for example, see \cite{assafgonzalezmacdonald}). 
\vspace{1em}

\begin{proposition}\label{rmk:algorithm}
    Suppose $X$ is a connected Demazure crystal, with highest weight element $b_{1}$. Then, the unique lowest weight element of $X$ can be constructed via a recursive algorithm as follows:
    \begin{enumerate}
        \item Find the smallest integer $a_n$ such that $b_2 = f_n^*f_{n-1}^*\dots f_{a_n}^*(b_{1})$ is nonzero and still in $X$. If there is no such $a_n$, we set $a_n$ to $0$ and set $b_2 = b_1$.  
        \item Next, find the smallest $a_{n-1}$ such that $b_3 = f_{n-1}^* f_{n-2}^* \dots f_{a_{n-1}}^*(b_2)$ is nonzero and still in $X$. If there is no such $a_{n-1}$, set $a_{n-1} = 0$ and $b_3 = b_2$. 
    \end{enumerate}  Continue like this for each $i \geq 1$, until we have chosen an element $b_n$. Then, $b_n$ is the unique lowest weight element of $X$. 
\end{proposition}

    Essentially, this algorithm greedily chooses the `longest' operator $F_{n,a_n}$ we can apply while staying within $X$, and then the longest $F_{n-1, a_{n-1}}$, and so on. In our notation, this `greedy algorithm' tells us to apply $F_{1,a_1}F_{2,a_2} \dots F_{n,a_n}$, where each $a_i$ is chosen successively minimal, to the highest weight element in order to locate the lowest weight element.  
    
    Given this fact, the idea of Theorem \ref{thm:principalmoreusable} is that we want to find a lowest weight element in $X$ in order to prove that $X$ is a Demazure crystal; Proposition \ref{rmk:algorithm} gives us a candidate for this lowest weight element, and the conditions in Theorem \ref{thm:principalmoreusable} help us check that our candidate is actually lowest weight. See section \ref{sec:appendix} for the details of this proof. 

    \vspace{1em}

    \begin{example}

Let's consider a Demazure crystal given by a set of flagged straight tableaux, with $\lambda =  (2, 1, 1)$ and flag $\vec{b} = (2, 2, 5)$. We will use Proposition \ref{rmk:algorithm} to locate the lowest weight element. (Only the tableaux satisfying the flagging condition, not all the tableaux in $\mathcal{B}(\lambda)$, are depicted below). 

First, we compute $a_4$. We must find the minimal integer such that $f_4^*f_3^* \dots f_{a_4}^*(b_{\lambda})$ is nonzero and inside the Demazure crystal. This integer can be computed to be $3$, so we move to the element $x = f_4^*f_3^*(b_{\lambda})$ (which has weight vector $(2, 1, 0, 0, 1)$). Now, we look for $a_3$: $f_3(x) = 0$, and $f_3^*f_2^*(x)$ does not obey the flagging condition. So, we skip to $a_2$ without adding any raising operators to our expression. Here, again because $f_2^*(x)$ disobeys the flagging condition, we set $a_2 = 0$ and skip to $a_1$. Finally, $f_1^*(x)$ is nonzero and in our Demazure crystal, so $a_1 = 1$. 

The result of our algorithm gives us the lowest weight element $f_1^*f_4^*f_3^*(b_{\lambda})$ (in this small case, $f_i$ always happens to be equal to $f_i^*$). This element is, indeed, the bottom-most tableau in the diagram below, and has the unique lowest weight in the dominance order. 

There are many different paths in the crystal graph we could have taken to go from $b_{\lambda}$ to $f_1^*f_4^*f_3^*(b_{\lambda})$. Each path corresponds to a different reduced word for the permutation $s_1s_4s_3$; our algorithm just gives us one preferred path downwards. 
    
        
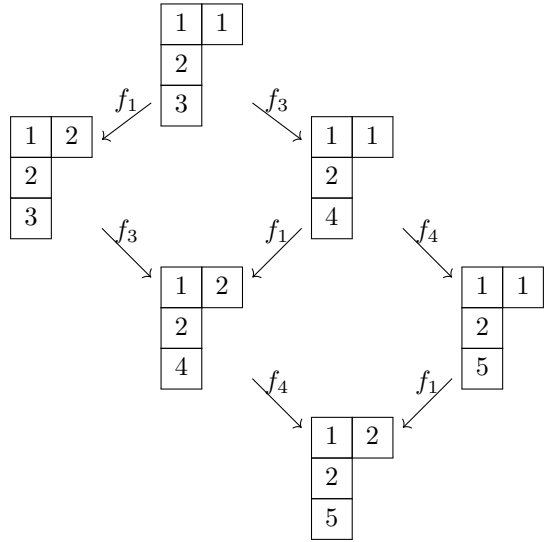
\begin{center}
        \begin{tikzpicture}
    \node(A) at (8,8) {\begin{ytableau}
        1&1\\
        2\\
        3
    \end{ytableau}};
    
    \node(B) at (6,6.5) {\begin{ytableau}
         1&2\\
        2\\
        3
    \end{ytableau}};

    \node(B2) at (10,6.5) {\begin{ytableau}
         1&1\\
        2\\
        4
    \end{ytableau}};

    \node(C) at (8,4.5) {\begin{ytableau}
         1&2\\
        2\\
        4
    \end{ytableau}};

    \node(D) at (12,4.5) {\begin{ytableau}
         1&1\\
        2\\
        5
    \end{ytableau}};

    \node(E) at (10,2.5) {\begin{ytableau}
         1&2\\
        2\\
        5
    \end{ytableau}};
    
    \draw[->] (A) -- (B) node[midway, above] {$f_1$};

\draw[->] (A) -- (B2) node[midway, above] {$f_3$};

\draw[->] (B) -- (C) node[midway, above] {$f_3$};

\draw[->] (B2) -- (C) node[midway, above] {$f_1$};

\draw[->] (B2) -- (D) node[midway, above] {$f_4$};

\draw[->] (D) -- (E)
node[midway, above] {$f_1$};

\draw[->] (C) -- (E) node[midway, above] {$f_4$};

\end{tikzpicture}
\end{center}

    \end{example}

\begin{proposition}\label{prop:applyingcharacterization}
    Assume that $\lambda$ and $\nu$ are strict partitions, and $\vec{b}$ is a nondecreasing flag. Then, the set of flagged shuffle tableaux satisfying flag $\vec{b}$ of shape $(\lambda/\mu)\circledast(\nu/\rho)$ is extremal, ideal, and satisfies the conditions in Theorem \ref{thm:principalmoreusable}. Hence, it forms a disjoint union of Demazure crystals. 
\end{proposition}

We will prove Theorem \ref{thm:principalmoreusable} more formally in an appendix (section \ref{sec:appendix}). The proof of Proposition \ref{prop:applyingcharacterization} is still somewhat onerous from this point. However, to demonstrate the utility of Theorems \ref{thm:characterization} and \ref{thm:principalmoreusable}, we will first show how they can be used on (usual) flagged skew tableaux.

\subsection{Example: flagged skew tableaux form Demazure crystals}
In this section, we use our characterization of Demazure crystals in Theorem \ref{thm:principalmoreusable} to prove that flagged skew tableaux form a Demazure crystal. Although this fact has been previously known (for example, see section $3$ of \cite{demazureskewtableaux}), we include this section to provide a simpler example application of Theorem \ref{thm:principalmoreusable}. Moreover, this simpler case might give the reader some insight into the more difficult case of shuffle tableaux.

Throughout this section, let $X$ denote a connected subset of flagged Young tableaux containing the highest weight element $b_{\lambda}$. The proof that $X$ is extremal and ideal is nearly identical as the case for shuffle tableaux, and so we focus on proving that $X$ is principal. 
\vspace{1em}

\begin{remark}
    We start with a few remarks on notation that we use throughout our proof (both for skew and shuffle tableaux).
    \begin{itemize}
        \item We call an entry with value $i$ which is in a row with flag $i$ \textit{bad} (ie because if it raises to an $i+1$, the tableau will violate the flag condition).  
        \item We often wish to refer to various different entries in our tableau, some of which have the same value. In order to avoid confusion, we both color code these entries and use markings such as underlines, boxes, circles, bolded entries, and dashed boxes to differentiate each element.
        \item Whenever we say an entry $a$ is \textit{right} of another entry $b$ we mean that $a$ is right of $b$ in the reading word of our tableau. If we say that $a$ is \textit{weakly} right of $b$, this implies either $a$ and $b$ are the same entry or $a$ is right of $b$. Similarly for (weakly) left. 
    \end{itemize}
\end{remark}

We now prove a couple of useful lemmas.
\vspace{1em}

\begin{lemma}\label{lemma:skewrightunpairs}
    Let $x\in X$. \begin{enumerate}
        \item Suppose that $f_{i-1}(x)$ contains an unpaired $i$ that is either paired or not present in $x$. Then the rightmost unpaired $i$ in $f_{i-1}^*(x)$ is also unpaired in $f_{i-1}(x)$.
        \item Suppose that $f_{i+1}(x)$ contains an unpaired $i$ that is paired in $x$. Then the rightmost unpaired $i$ in $f_{i+1}^*(x)$ is also present in $f_{i+1}(x)$.
    \end{enumerate}
\end{lemma}
\begin{proof}
    We start with (i). Note that the rightmost unpaired $\underline{\red{i-1}}$ in $x$ unpairs a $\boxed{\blue{i}}$ weakly to its right after applying $f_{i-1}$. Suppose that the rightmost unpaired $i$ in $f_{i-1}^*x$ is right of this $\boxed{\blue{i}}$. It is then unpaired by an $i-1\mapsto i$ left of the $\underline{\red{i-1}}$ and so is paired with an $i+1$ left of the $i-1\mapsto i$. Summarizing this, we have the following sequence of elements
    \[\begin{tikzcd}
	{i+1} & {i-1\mapsto i} & {\underline{\red{i-1\mapsto i}}} & {\boxed{\blue{i}}} & i
	\arrow[curve={height=-25pt}, no head, from=1-1, to=1-5]
\end{tikzcd}\]
    Note that we may have that the $\underline{\red{i-1\mapsto i}}=\boxed{\blue{i}}$. In either case though, we get a contradiction as we have an unpaired $\boxed{\blue{i}}$ between the paired $i+1,i$ (the $\boxed{\blue{i}}$ should instead pair with the $i+1$).

    We now do (ii). We again have that the rightmost unpaired $\underline{\red{i+1}}$ in $x$ unpairs an $\boxed{\blue{i}}$ to its right after applying $f_{i+1}$. Suppose that the $\boxed{\blue{i}}$ is not the rightmost unpaired $i$ in $f_{i+1}^*x$. Consider the $i+1$ that the rightmost unpaired $i$ in $f_{i+1}^*x$ is paired with. This $i+1$ is right of the $\boxed{\blue{i}}$ (as it is not paired with the $\boxed{\blue{i}}$). In order for the $i$ to get unpaired, it is unpaired by some \dbox{$\purple{i+1}\mapsto i+2$}, which by assumption is left of the $\underline{\red{i+1}}$. There thus must be some $i+1$ weakly left of the \dbox{$\purple{i+1}\mapsto i+2$} which is paired with an $i$ right of the $i+1$. This though is a contradiction, as we get the following sequence of elements:
    \[\begin{tikzcd}
	{i+1} & {\boxed{\blue{i}}} & i \\
	& {\text{unpaired}}
	\arrow[curve={height=-12pt}, no head, from=1-1, to=1-3]
    \end{tikzcd}\]
    This is impossible as the $\boxed{\blue{i}}$ should pair with the left $i+1$.
\end{proof}

The following useful lemma follows entirely by symbolic manipulations, and so we can also use it in later sections for shuffle tableaux and more general crystals.
\vspace{1em}

\begin{lemma}\label{lemma:symbolicmanipulation}
    Assume that $a \geq b$. Then,
    \[(f_n^* \dots f_a^*)(f_{n+1}^* \dots f_{b-1}^*) = F_{n,a}F_{n+1,b-1} = (f_{n+1}^*f_n^* \dots f_a^*)(f_{n+1}^* \dots f_{a+1}^* f_{a-1}^* f_{a-2}^* \dots f_{b-1}^*).\]
    Reorganizing, this in turn equals $f_{n+1}^*(F_{n,b-1}F_{n+1,a+1})$.
\end{lemma}
\begin{proof}
    The first equality is definitional. We prove the second equality by inducting on $n$, using only braid relations. First, repeatedly using the first braid relation, we push the $f_{n+1}$ to the left: 
\[(f_n^* \dots f_a^*)(\textcolor{red}{{f_{n+1}^*}} \dots f_{b-1}^*) = 
(f_n^* \textcolor{red}{f_{n+1}^*})(f_{n-1}^* \dots f_a^*)(f_{n}^* \dots f_{b-1}^*)
\]

Next, by induction, this expression becomes:

\[(f_n^* f_{n+1}^*) (f_{n}^* \dots f_a^*)(f_n^* \dots f_{a+1}^* f_{a-1}^* \dots f_{b-1}^*)\]

Using the second braid relation, we get:
\[(f_{n+1}^*f_n^* f_{n+1}^*)(f_{n-1}^* \dots f_a^*)(f_n^* \dots f_{a+1}^* f_{a-1}^* \dots f_{b-1}^*)\]

Finally, we push an $f_{n+1}^*$ to the right again, repeatedly using the first braid relation again: 

\[(f_{n+1}^*f_n^* \textcolor{red}{f_{n+1}}^*)(f_{n-1}^* \dots f_a^*)(f_n^* \dots f_{a+1}^* f_{a-1}^* \dots f_{b-1}^*) = (f_{n+1}^*f_n^* f_{n-1}^* \dots f_a^*)(\textcolor{red}{f_{n+1}^*}f_n^* \dots f_{a+1}^* f_{a-1}^* \dots f_{b-1}^*)\]
as claimed. The last remark follows by using the first braid relation to commute $f_{a-1}^* \dots f_{b-1}^*$ left.
\end{proof}

\begin{lemma}
    Flagged Young tableaux satisfy the extension property from \ref{thm:principalmoreusable}.
\end{lemma}
\begin{proof}
    As this proof is somewhat involved, we start by sketching the structure of the argument. In the first part of the proof, we seek to reduce our claim, which is currently about a long sequence of operators, to a simpler statement involving only three operators $f_n^*,f_{n+1}^*,f_{n-1}^*$. To do this, we define an auxiliary tableau $y$ and modified crystal operator acting on $y$ so that our original claim becomes equivalent to a much simpler result about $y$. We end the proof by proving this result about $y$ and then finally relate all of our work back to our original tableaux.

    First, recall the extension property from \ref{thm:principalmoreusable}:
    
    \textbf{Extension Property}: Let $x\in X$ be extremal such that $e_i(x)=0$ for all $i\leq n+1$. Suppose that $F_{n,a}F_{n+1,b}x,F_{n+1,b-1}(x)\in X$ with $b\leq a$. Then $F_{n,a}F_{n+1,b-1}(x)\in X$. 

    By Lemma \ref{lemma:symbolicmanipulation}, we have
    \begin{equation}\label{equation:F_na}
    F_{n,a}F_{n+1,b-1}(x) = f_{n+1}^*(F_{n,a}F_{n+1,a+1}f_{a-1}^*(f_{a-2}^*\dots f_{b-1}^*(x))).
    \end{equation}
    By assumption $F_{n+1,b-1}(x)\in X$, so applying the leftmost $f_{n+1}$ on the right side of Equation \ref{equation:F_na} cannot break the flagging conditions. Therefore, it suffices to show that 
\[F_{n,a}F_{n+1,a+1}f_{a-1}^*(f_{a-2}^*\dots f_{b-1}^*x)\in X\]
    
    Now define the following auxiliary shuffle tableaux:
    \begin{enumerate}
        \item $z:=f_{a-2}^*\dots f_{b-1}^*(x)$.
        \item $z':=f_{n-2}^*\dots f_a^*\mathbf{f_{a-1}}^*f_{n-1}^*\dots f_{a+1}^*(z)$.
        \item $z'':=f_{n-2}\dots f_a^*f_{n-1}^*\dots f_{a+1}^*(z)$ (the same as $z'$, only without the bolded $f_{a-1}^*$).
    \end{enumerate}
    To motivate these definitions, note that
    \[
    F_{n,a}F_{n+1,a+1}f_{a-1}^*(f_{a-2}^*\dots f_{b-1}^*)x = f_n^*f_{n-1}^*f_{n+1}f_n^*z'
    \]
    and
    \[
    F_{n,a}F_{n+1,a+1}(f_{a-2}^*\dots f_{b-1}^*)x = f_n^*f_{n-1}^*f_{n+1}^*f_n^*z''.
    \]
    We'll later show that $f_n^*f_{n-1}^*f_{n+1}f_n^*(z'')\in X$ (this claim makes intuitive sense as removing the $f_{a-1}^*$ from $z'$ stops the extension by $f_{b-1}^*$ from 'reaching' the leftmost $f_n^*$). Our final goal is to show that $f_n^*f_{n-1}^*f_{n+1}f_n^*z'\in X$.
    

    By exactly the same argument as in Lemma \ref{lemma:unpairedsubset}, the set of unpaired $n-1$s in $f_n^*(z'')$ forms a subset of the unpaired $n-1$s in $f_n^*(z')$. Thus, define a final auxiliary skew tableau $y$ and a `modified' crystal operator $\tilde{f}_{n-1}$ as follows:
    \begin{enumerate}
        \item Start with $f_n^*(z')$. Raise each unpaired $n-1\mapsto n$ in $f_n^*(z')$ iff the corresponding $n-1$ is also unpaired in $f_n^*(z'')$.
        \item The operator $\tilde{f}_{n-1}$ acts on $y$ as follows. After forming the reading word of $y$ as usual, remove each of the $n$s created in the last step. Then, apply $f_{n-1}^*$ to the reading word as usual.
    \end{enumerate}

    We now show that $y$ is a valid skew tableau and that $e_n(y)=0$. 
    \begin{claim}
        $y$ is a valid skew tableau.
    \end{claim}
    We show that the rows of $y$ are weakly increasing while the columns are strictly increasing. Consider some $\underline{\red{n-1}}$ which we raise to an $n$ in (i) while constructing $y$. First, since the $\underline{\red{n-1}}$ is also unpaired in $f_n^*(z'')$, it cannot be directly above an $n$ in $f_n^*(z'')$. It thus also cannot be in the same column as an $n$ in $f_n^*(z')$, and so the columns remain strict. Similarly, for the rows, since the $\underline{\red{n-1}}$ is unpaired in $f_n^*(z'')$, there cannot be any paired $n-1$s to its right in $f_n^*(z'')$. Thus every $n-1$ right of the $\underline{\red{n-1}}$ in $f_n^*(z')$ is unpaired in $f_n^*(z'')$, and so we also raise them to $n$s while constructing $y$. Therefore, rows remain nondecreasing from left to right in $y$. 

    \begin{claim}
        $e_n(y)=0$.
    \end{claim}
    Note that the set of $n$s and $n+1$s in $y$ equals the set of $n$s and $n+1$s in $f_{n-1}^*f_n^*(z'')$ by construction. Then, since $e_n(z'')=e_{n-1}(z'')=0$ by Lemma \ref{e_i=0}, $e_n(y)=0$.

    Note that by our definition of $y$ and $\tilde{f}_{n-1}$, we have that $f_n^*f_{n+1}^*\tilde{f}_{n-1}(y)=f_n^*f_{n+1}^*f_{n-1}^*f_n^*(z')$. The following is the final main claim of our proof.

    \begin{claim}
        If $f_n^*f_{n+1}^*y,f_n^*\tilde{f}_{n-1}^*(y)\in X$ then $f_n^*f_{n+1}^*\tilde{f}_{n-1}^*(y)\in X$.
    \end{claim}
    Suppose for the sake of contradiction that $f_n^*f_{n+1}^*\tilde{f}_{n-1}^*(y)\not\in X$ while both $f_n^*f_{n+1}^*y,f_n^*\tilde{f}_{n-1}^*(y)\in X$. Denote the rightmost unpaired $n$ in $f_{n+1}^*\tilde{f}_{n-1}^*y$ by $\boxed{\blue{n}}$. First suppose that the $\boxed{\blue{n}}$ was not present in $y$. It was then created by an $n-1\mapsto \boxed{\blue{n}}$ but by assumption is paired in $\tilde{f}_{n-1}^*y$. This though is a contradiction, as the $n-1\mapsto\boxed{\blue{n}}$ would unpair an $n$ to its right in $\tilde{f}_{n-1}^*y$.

    We can thus assume that the $\boxed{\blue{n}}$ is present but paired in $y$ with an $\underline{\red{n+1}}$. The $\boxed{\blue{n}}$ is still paired in $f_{n+1}^*y$. We claim in fact that it is still paired with the same $\underline{\red{n+1}}$. If not then there would be an unpaired $n+1$ in $y$ paired with an $n$ which is right of the $\underline{\red{n+1}}$. This though is impossible, as the $n+1\mapsto n+2$ would unpair an $n$ weakly right of the $\boxed{\blue{n}}$. 

    Finally, consider the rightmost unpaired \dbox{\purple{$n-1\mapsto n$}} in $y$. In order for the \dbox{\purple{$n-1\mapsto n$}} to unpair the $\boxed{\blue{n}}$, it is right of the $\underline{\red{n+1}}$. We thus get a final contradiction, as the \dbox{\purple{$n-1\mapsto n$}} would then unpair an $n$ weakly right of the $\boxed{\blue{n}}$. This completes the proof of the claim.

    Recall that by construction, $f_n^*f_{n+1}^*\tilde{f}_{n-1}^*(y)=f_n^*f_{n+1}^*f_{n-1}^*f_n^*(z')$. Thus since $f_n^*f_{n+1}^*f_{n-1}^*f_n^*(z')\not\in X$ by assumption, we either have that $f_n^*f_{n+1}^*(y)\not\in X$ or $f_n^*\tilde{f}_{n-1}^*(y)\not\in X$. We first show that $f_n^*f_{n+1}^*(y)\in X$

    \begin{claim}
        The set of $n$s and $n+1$s in $f_{n+1}^*(y)$ equals the set of $n$s and $n+1$s in $f_{n+1}^*f_{n-1}^*f_n^*(z'')$.
    \end{claim}
    We start with the set of $n$s. The set of $n$s in $f_{n+1}^*y$ equals the set of $n$s in $y$. By construction, the set of $n$s in $y$ equals the union of the set of $n$s in $f_n^*(z')$ and the set of $n$s obtained by raising $n-1\mapsto n$ for every $n-1$ in $f_n^*(z')$ that is also unpaired in $f_n^*(z'')$. Since the set of $n$s in $f_n^*(z')$ equals the set of $n$s in $f_n^*(z'')$, we thus get that the set of $n$s in $f_{n+1}(y)$ equals the set of $n$s in $f_{n+1}^*f_{n-1}^*f_n^*(z'')$. Now consider the set of $n+1$s. The set of $n+1$s in $f_{n+1}^*(y)$ equals the set of $n+1$s in $f_{n+1}^*f_n^*(z')$, which in turn equals the set of $n+1$s in $f_{n+1}^*f_{n-1}^*f_n^*(z'')$. This thus completes the claim.

    Thus $f_n^*f_{n+1}^*(y)\in X$ iff $f_n^*f_{n+1}f_{n-1}^*f_n^*(z'')\in X$. We have,
    \[
    f_n^*f_{n+1}^*f_{n-1}^*f_n^*(z'') = f_n^*f_{n+1}^*f_{n-1}^*f_n^*(f_{n-2}^*\dots f_a^*f_{n-1}^*\dots f_{a+1}^*(z)) = F_{n,a}F_{n+1,a+1}(z) = F_{n,a}f_{a-2}^*\dots f_{b-1}^*F_{n+1,a+1}(x).
    \]
    Since $F_{n+1,b}(x)\in X\implies f_{a-2}^*\dots f_{b-1}^*F_{n+1,a+1}(x)\in X$, it suffices to show that $F_{n,a}F_{n+1,a+1}(x)\in X$.
    We have that $F_{n,a,b}F_{n+1,a+1}(x)\in X$, and so $F_{n,a}F_{n+1,a+1}(x)\in X$ as hoped. Thus $f_n^*f_{n+1}^*f_{n-1}^*f_n^*(z'')\in X$.

    We thus get that $f_n^*\tilde{f}_{n-1}^*(y)\not\in X$. By construction, $f_n^*\tilde{f}_{n-1}^*(y) = f_n^*f_{n-1}^*f_n^*(z')$. We then have,
    \[
    f_n^*f_{n-1}^*f_n^*(z') = f_n^*f_{n-1}^*f_n^*(f_{n-2}^*\dots f_a^*\mathbf{f_{a-1}}f_{n-1}^*\dots f_{a+1}^*(z)) = F_{n,b-1}F_{n,a+1}(x).
    \]
    The last equality follows by the definition of $z$. We can then braid relate $F_{n,a+1}$ leftwards to get
    \[
    F_{n,b-1}F_{n,a+1}x = F_{n-1,a}F_{n,b-1}(x)\not\in X.
    \]
    This though is a contradiction, as we assumed at the beginning that $F_{n,a}F_{n+1,b-1}(x) = f_n^*f_{n+1}^*(F_{n-1,a}F_{n,b-1}(x))$ had a bad $n\mapsto n+1$ and $F_{n+1,b-1}(x) \in X$. This finally completes the proof.
    \end{proof}

    We now prove the gluing property for flagged skew tableaux. This proof is much simpler than that of the extension property.
    \vspace{1em}

    \begin{lemma}
        The gluing property in Theorem \ref{thm:principalmoreusable} holds for flagged skew tableaux.
    \end{lemma}
    \begin{proof}
        We first recall the statement of the gluing property.
    
\textbf{Gluing Property}: Let $k,m,a\in \N$ and $x\in X$ be extremal so that $e_i(x)=0$ for all $a-m\leq i\leq a+k$. Let $k,m,a$ be such that $f_{a+k}^*\dots f_{a}^*(x)\in X$ and $f_{a-m}^*\dots f_{a-1}^*(x)\in X$. Then either $f_{a+k}^*\dots f_{a-1}^*(x)\in X$ or $f_{a-m}^*\dots f_{a-1}^*f_a^*(x)\in X$. 

        Let $x\in X$ be extremal with $e_i(x) = 0$ for all $a-m\leq i\leq a+k$. Suppose for the sake of contradiction that $f_{a+k}^*\dots f_{a}^*x\in X$ and $f_{a-m}^*\dots f_{a-1}^*(x)\in X$ while both $f_{a+k}^*\dots f_{a-1}^*(x), f_{a-m}^*\dots f_{a-1}^*f_a^*(x)\not\in X$.  

        We start with a few observations about the structure of $x$.

        \begin{claim}
            The rightmost unpaired $\boxed{\blue{a-1}}$ in $f_a^*(x)$ is right of the rightmost unpaired $a-1$ in $x$.
        \end{claim}
        Suppose that $f_a^*(x),x$ contain the same rightmost unpaired $a-1$. As $f_a^*(x),x$ contain the same set of $a-1$s and $a-2$s, by lemma \ref{lemma:skewrightunpairs}, $f_{a-1}^*f_a^*(x),f_{a-1}^*(x)$
        contain the same rightmost unpaired $a-2$. Repeating this argument, we obtain that $f_{a-m+1}^*\dots f_{a-1}^*f_a^*(x)$ and $f_{a-m+1}^*\dots f_{a-1}^*(x)$ contain the same rightmost unpaired $a-m$. This though is contradictory, as $f_{a-m}^*\dots f_{a-1}^*(x)\in X$ while $f_{a-m}^*\dots f_{a-1}^*f_a^*(x)\not\in X$.
        
        \begin{claim}
            The rightmost unpaired $a$ in $f_{a-1}^*(x)$ is right of the rightmost unpaired $\underline{\red{a}}$ in $x$.
       \end{claim}
       This follows by the same argument as above.

        \begin{claim}
            The rightmost unpaired $a-1$ in $x$ is right of the rightmost unpaired $\underline{\red{a}}$ in $x$.
        \end{claim}
        By Lemma \ref{lemma:skewrightunpairs}, the rightmost unpaired $a-1$ in $x$ unpairs the the rightmost unpaired $a$ in $f_{a-1}^*(x)$. By the previous claim, this $a$ is right of the rightmost unpaired $\underline{\red{a}}$ in $x$. Thus, in order for the $a-1\mapsto a$ to unpair the $a$ in $f_{a-1}^*(x)$, it must be right of the $\underline{\red{a}}$. 

        We thus get the following sequence of elements:
        \[\begin{tikzcd}
	\underline{\red{a}} & {a-1} & {\boxed{\blue{a-1}}}
    \end{tikzcd}\]
        In $x$, the $\boxed{\blue{a-1}}$ is paired with an $a$. In order for the $\underline{\red{a}}$ to unpair the $\boxed{\blue{a-1}}$, we must have that there is an $a$ (weakly) left of the $\underline{\red{a}}$ paired with an $a-1$ (weakly) right of the $a-1$. We then though have the following
        \[\begin{tikzcd}
	a & {a-1} & {a-1} \\
	& {\text{unpaired}}
	\arrow[curve={height=-12pt}, no head, from=1-1, to=1-3]
    \end{tikzcd}\]
    This though is a contradiction, as the $a-1$ should pair with the $a$.
    \end{proof}

    \begin{corollary}
        Flagged skew tableaux form Demazure crystals.
    \end{corollary}
    \begin{proof}
        By Theorem \ref{thm:principalmoreusable}, flagged skew tableaux are principal. Then, since flagged skew tableaux are also extremal and ideal, they form a Demazure crystal by theorem \ref{thm:characterization}.
    \end{proof}

\section{The ideal and extremal properties for shuffle tableaux}\label{sec:extremal+ideal}

In this section, we will prove that if $\mathcal{B}$ is a crystal of shuffle tableaux, then flagged shuffle tableaux sit inside of $\mathcal{B}$ as an extremal and ideal subset. 
\subsection{The extremal property}

Recall that $X$ is extremal if, for any $i$-string $S$ in $\mathcal{B}$, either $S \cap X = \emptyset, S$, or $\{b\}$, where $e_i(b) = 0$ (so, $b$ is the top endpoint of the $i$-string). 
\vspace{1em}

\begin{lemma}\label{lemma:extremality}
    Let $X$ be the set of shuffle tableaux satisfying flag $\vec{b}$ of skew shape $\lambda\oslash \mu$. Then, $X$ is extremal. 
\end{lemma}

\begin{proof}
        Let $S$ be an $i$-string in $\mathcal{B}$ and let $x \in X \cap S$. 
        First, consider $e_i(x)$. Since $e_i$ only decreases the values in $x$, if $e_i(x)\neq 0$, then $e_i(x)\in X$. Now suppose $e_i(x) \neq 0$ and $f_i(x)\not=0$. Since $f_i(x)\not=0$, there is an unpaired $i$ in $x$; consider the rightmost one in the reading word of $x$. Moreover, since $e_i(x)\neq 0$, there is also an $i+1\in x$ that is not paired with an $i$. If this $i+1$ is left of the rightmost unpaired $i$, they would be paired. Thus, this $i+1$ is right of the rightmost unpaired $i$, and so the $i$ is in a row with flag at least $i+1$ (since the flag $\vec{b}$ is nondecreasing). Therefore, $f_i(x)\in X$. 
        
        Iterating the above argument, if there is any $x \in S \cap X$ with $e_i(x) \neq 0$ and $f_i(x) \neq 0$, then $S \cap X = S$. Therefore, either $S \cap X$ is empty, $S$, or the top endpoint of $S$.
\end{proof}

\subsection{The ideal property}

In section $2$ of the appendix \ref{sec:appendix}, we decompose the ideal property into the following easier to use characterization.
\vspace{1em}
\begin{proposition}\label{prop:idealdecomp}
    Let $X$ be a subset of a crystal. Suppose that $X$ satisfies the following list of ideal properties. 
    \begin{enumerate}
        \item The gluing property from the statement of Theorem \ref{thm:principalmoreusable}.
        \item Let $x\in X$. Suppose that $f_n^*f_{n-1}^*\dots f_{n-k}^*(x)\in X$. Then $f_n^*\dots f_{n-k+1}^*(x)\in X$.
        \item Let $x\in X$. Suppose that $f_n^*f_{n+1}^*\dots f_{n+k}^*(x)\in X$. Then $f_n^* f_{n+1}^* \dots f_{n+k-1}^*(x)\in X$. 
    \end{enumerate}
    Then $X$ is ideal.
\end{proposition}
\vspace{1em}
\begin{proposition}\label{prop:idealshuffle}
    Let $\lambda/\mu$ and $\nu/\rho$ be skew shapes, with $\lambda$ and $\nu$ strict partitions. Then, the set $X$ of shuffle tableaux of shape $(\lambda/\mu)\circledast(\nu/\rho)$ flagged by a nondecreasing $\vec{b}$ is an ideal subset.
\end{proposition}

\begin{proof}
    By Proposition \ref{prop:idealdecomp}, in combination with the gluing property from Theorem \ref{thm:principalmoreusable} proven in Section \ref{section:locprop2}, it suffices to verify ideal properties (ii) and (iii) for flagged shuffle tableaux.
    
    \textbf{Ideal property (ii):}
    Inducting on the length of our expression, we may assume that $f^*_{n-1} \dots f^*_{n-k+1}(x) \in X$. So, it suffices to show that $f^*_{n-1} \dots f^*_{n-k+1}(x)$ does not contain any unpaired `bad' $n$s (that is, entries $n$ in a row of flag $n$). 
    
    We claim that the set of unpaired $n$s in $f_{n-1}^* \dots f_{n-k+1}^*(x)$ is a subset of the unpaired $n$s in $y = f_{n-1}^* \dots f_{n-k}^*(x)$. This claim again follows by inducting on the length of the expression, and by applying Lemma \ref{lemma:subsetlemma}. Since we assumed that $f_n^*f_{n-1}^* \dots f^*_{n-k}(x) \in X$, the tableau $y$ contains no bad $n$s, as desired.

    \textbf{Ideal property (iii):} This property follows very similarly. As before, it suffices to show that $f_{n-1}^* f_{n-2}^* \dots f_{n+k-1}^*(x)$ does not contain any bad unpaired $n$s. Inducting on the length $k$ of the expression and again applying Lemma \ref{lemma:subsetlemma}, we find that the set of unpaired $n$s in $f_{n-1}^* \dots f_{n+k-1}^*(x)$ is a subset of the set of unpaired $n$s in $f_{n-1}^* \dots f_{n+k}^*(x)$. By the assumption that $f_n^*f_{n-1}^* \dots f_{n+k-1}^*(x) \in X$, we are done. 
    
    Thus, since $X$ satisfies the hypotheses in Proposition \ref{prop:idealdecomp}, $X$ is ideal. 
\end{proof}

\section{Principality for shuffle tableaux}\label{sec:principality}

In this section, we will use Theorem \ref{thm:principalmoreusable} to show that the set of shuffle tableaux of shape $(\lambda/\mu)\circledast(\nu/\rho)$ flagged by $\vec{b}$ is a principal subset of its crystal, as long as $\lambda$ and $\nu$ are strict partitions. Throughout this section, $X$ denotes the set of flagged shuffle tableaux of shape $(\lambda/\mu)\circledast(\nu/\rho)$. 


\subsection{Some preliminary lemmas}

The following unsurprising yet useful lemma states essentially that bad $i$s occur at the end of rows. 
\vspace{1em}

\begin{lemma}\label{lemma:easyendofrow}
    Suppose that $x \in X$ but $f_i(x) \notin X$. Then, the $i$ that is raised by $f_i$ is at the end of its row. 
\end{lemma}

\begin{proof}
    Let the $\red{i}$ raised by $f_i$ be in row $j$ of $\lambda$. Then, $\vec{b}(j)$ must be $i$ (otherwise $f_i(x) \in X$). Therefore, if there are any other entries to the right of this $\red{i}$ in row $j$, these entries must all also by $i$s (since $x \in X$, and the rows are nondecreasing left to right). There also though cannot be an $i$ right of the $\red{i}$ in row $j$, as then, since the $\red{i}$ maps to an $i+1$ in $f_ix$, row $j$ of $f_ix$ would violate the non-decreasing condition of $f_ix$. Thus $\red{i}$ is at the end of row $j$.
\end{proof}

\begin{lemma}\label{lemma:ibetweeni,i+1}
    Suppose $x$ is a shuffle tableau containing a sequence of entries 
    \[\begin{tikzcd}
	{i+1} & i & i \\
	& {\text{unpaired}}
	\arrow[curve={height=-12pt}, no head, from=1-1, to=1-3]
    \end{tikzcd}\]
    where the connected elements are paired, while the middle $i$ is unpaired. Then the paired $i+1,i$ are column paired.
\end{lemma}
\begin{proof}
    If the paired $i+1,i$ were not column paired then, since it is further left and unpaired, the middle $i$ would pair with the $i+1$. Thus the paired $i+1,i$ are column paired.
\end{proof}

\begin{lemma}{\label{e_i=0}}
Let $x\in X$ be extremal so that $e_i(x)=0$. Then,
\begin{enumerate}
    \item If $e_{i-1}(x)=0$, $e_i(f_{i-1}^*f_i^*(x))=0$,
    \item If $e_{i+1}(x)=0$, $e_i(f_{i+1}^*f_i^*(x))=0$.
\end{enumerate}
\end{lemma}
\begin{proof}
First suppose that $e_{i-1}(x)=0$; we prove (i). Note first that because $e_{i-1}(x)=0$, every $i$ in $x$ is paired with an $i-1$. There are thus more $i-1$ in $x$ than $i$s. As $f_{i-1}^*,f_i^*$ act by $s_{i-1},s_i$ on the content of $x$, the number of $i+1$s in $f_{i-1}^*f_i^*(x)$ equals the number of $i$s in $x$ while the number of $i$s equals the number of $i-1$s in $x$. There are thus more $i$s in $f_{i-1}^*f_i^*(x)$ than $i+1$s, and so $f_i(f_{i-1}^*f_i^*(x))\not = 0$. Since $f_{i-1}^*f_i^*(x)$ is extremal though, we either have that $e_i(f_{i-1}^*f_i^*(x))=0$ or $f_i(f_{i-1}^*f_i^*(x))=0$. Thus since $f_i(f_{i-1}^*f_i^*(x))\neq 0$, $e_i(f_{i-1}^*f_i^*(x))=0$.

Now suppose that $e_{i+1}(x)=0$. Again, since $e_{i+1}(x)=0$, every $i+2$ in $x$ is paired with an $i+1$, and so there are more $i+2$s in $x$ than $i+1$s. On the other hand, the number of $i+1$s in $f_{i+1}^*f_i^*(x)$ equals the number of $i+2$s in $x$ while the number of $i$s in $x$ equals the number of $i+1$s in $x$. There are hence more $i$s than $i+1$s in $f_{i+1}^*f_i^*(x)$, and so $f_i(f_{i+1}^*f_i^*(x))\not=0$. We thus as before obtain that $e_i(f_{i+1}^*f_i^*(x))=0$.
\end{proof}

\begin{remark}
    We now briefly motivate the following two lemmas, which are the first two results whose proofs require significant work with shuffle tableaux.
    
    When we apply $f_{i-1}^*$ or $f_{i+1}^*$ to a tableau $x$, we will likely create many new unpaired $i$s (which were either paired in $x$ before or did not exist in $x$). Informally, the following two lemmas prove that these $i$s get unpaired in order from right to left (that is, the first $i$ that gets unpaired by $f_{i-1}$, for example, is the rightmost $i$ that gets unpaired by applying $f_{i-1}^*(x)$). This is intuitively reasonable, as $f_{i-1}$ always lowers the rightmost unpaired $i-1$ to an $i$, but takes some effort to prove. 
\end{remark}

\vspace{1em}

\begin{lemma}\label{lemma:rightunpairs1}
Let $x\in X$. If $f_{i-1}(x)$ contains an unpaired $i$ that is either paired or not contained in $x$, the rightmost $i$ that is unpaired in $f_{i-1}^*(x)$ is also unpaired in $f_{i-1}(x)$.
\end{lemma}

\begin{proof}
    We start by fixing some notation. Throughout this proof, let $\underline{\red{i-1\mapsto i}}$ be the rightmost unpaired $i-1$ in $x$ and $\boxed{\blue{i}}$ be the rightmost unpaired $i$ in $f_{i-1}^*(x)$. Assume for the sake of contradiction that $\boxed{\blue{i}}$ is not unpaired in $f_{i-1}(x)$. It is then unpaired after some $\Circled{\green{i-1\mapsto i}}$ unpairs it. By assumption, $\underline{\red{i-1\mapsto i}}$ unpairs some $\mathbf{\yellow{i}}$. We now work through four cases according to whether the rightmost unpaired $\yellow{\mathbf{i}}$ in $f_{i-1}(x)$ and the rightmost unpaired $\boxed{\blue{i}}$ in $f_{i-1}^*x$ are present (but paired) in $x$.
    
    \textbf{Case 1: Both the $\boxed{\blue{i}}$ and $\mathbf{\yellow{i}}$ are not present in $x$}
    
    Then the $\mathbf{\yellow{i}}$ is equal to the $\underline{\red{i-1\mapsto i}}$ and the $\boxed{\blue{i}}$ is created by some other $i-1\mapsto i$. We claim that this situation is impossible. The $\underline{\red{i-1}}\mapsto\mathbf{\yellow{i}}$ is right of the $i-1\mapsto \boxed{\blue{i}}$, since the $\underline{\red{i-1}}$ is the rightmost unpaired $i-1$ in $x$ by definition (and the $i-1 \mapsto \boxed{\blue{i}}$ is also unpaired in $x$).
    But meanwhile, the $\boxed{\blue{i}}$ is right of the $\mathbf{\yellow{i}}$ by assumption. We cannot have both $\boxed{\blue{i}}$ right of $\mathbf{\yellow{i}}$ and $\underline{\red{i-1}}$ right of $i-1 \mapsto \boxed{\blue{i}}$. 

    \textbf{Case 2: The $\mathbf{\yellow{i}}$ is not present in $x$, while the $\boxed{\blue{i}}$ is present in $x$}

    As before, the $\mathbf{\yellow{i}}$ is equal to the $\underline{\red{i-1\mapsto i}}$. On the other hand, the $\boxed{\blue{i}}$ is right of the $\underline{\red{i-1\mapsto i}}$ by assumption. Since the $\Circled{\green{i-1\mapsto i}}$ unpairs the $\boxed{\blue{i}}$, the $i+1$ with which the $\boxed{\blue{i}}$ is initially paired with must be left of the $\Circled{\green{i-1\mapsto i}}$. Finally, by assumption, the $\Circled{\green{i-1\mapsto i}}$ is left of the $\underline{\red{i-1\mapsto i}}$ (since the $\underline{\red{i-1} \mapsto i}$ was the rightmost unpaired $i-1$ in $x$).  We thus have the following sequence of elements:
    
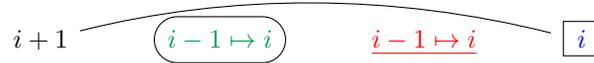
\[\begin{tikzcd}
	{i+1} & {\Circled{\green{i-1\mapsto i}}} & {\underline{\red{i-1\mapsto i}}} & {\boxed{\blue{\hspace{0.2em}i\hspace{0.2em}}}}
	\arrow[curve={height=-18pt}, no head, from=1-1, to=1-4]
\end{tikzcd}\]
    where the connected elements are paired in $x$. Note though that before $\Circled{\green{i-1\mapsto i}}$, the $\boxed{\blue{i}}$ and $i+1$ are paired while the $\red{i}$ is unpaired. Thus the $\boxed{\blue{i}}$ and $i+1$ are column paired. This though is a contradiction because if the $\boxed{\blue{i}}$ is column paired with an $i+1$, it can never become unpaired by only applying $f_{i-1}$s.

    \textbf{Case 3: The $\mathbf{\yellow{i}}$ is present in $x$ while the $\boxed{\blue{i}}$ is not present in $x$}

    We thus have that the $\boxed{\blue{i}}$ is created by some $i-1\mapsto \boxed{\blue{i}}$. 
    
    \textbf{Subcase 3.1:}
    This $i-1\mapsto \boxed{\blue{i}}$ is not the $\underline{\red{i-1\mapsto i}}$. 
    
    By assumption, the $\underline{\red{i-1\mapsto i}}$ is right of the $i-1\mapsto \boxed{\blue{i}}$. Since the $\mathbf{\yellow{i}}$ is present in $x$ and left of the $i-1\mapsto i$, the $\mathbf{\yellow{i}}$ is left of the $\underline{\red{i-1\mapsto i}}$. The $\underline{\red{i-1\mapsto i}}$ pairs with some $i+1$ left of the $\mathbf{\yellow{i}}$ (so that the $\mathbf{\yellow{i}}$ becomes unpaired in $f_{i-1}(x)$). We thus have the following sequence of elements (the sequence above implies that the $i+1$ and $\underline{\red{i-1\mapsto i}}$ must be column paired):
    
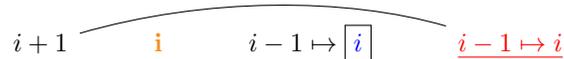
\[\begin{tikzcd}
	{i+1} & {\mathbf{\yellow{i}}} & {i-1\mapsto \boxed{\blue{i}}} & {\underline{\red{i-1\mapsto i}}}
	\arrow[curve={height=18pt}, no head, from=1-4, to=1-1]
    \end{tikzcd}\]
    This sequence determines the following substructure in $x$.

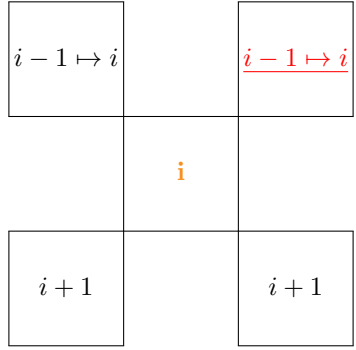
    \begin{center}
    \ytableausetup{boxsize = 4.3em}
\begin{ytableau}
    i-1\mapsto i&\none&\underline{\red{i-1\mapsto i}}\\
    \none&\mathbf{\yellow{i}}\\
    i+1&\none&i+1
\end{ytableau}
\end{center}

    The $i+1$ in the bottom left square is implied by the rest of the diagram (as rows need to weakly increase and columns need to strictly increase). However, this substructure is a contradiction as in the previous case, as it implies the $i-1\mapsto \boxed{\blue{i}}$ is column paired. 

\textbf{Subcase 3.2:} The $i-1\mapsto \boxed{\blue{i}} = \underline{\red{i-1\mapsto i}}$. 

Then, $\underline{\red{i-1\mapsto i}}$ is initially paired in $f_{i-1}(x)$, but later becomes unpaired. We still have  that $\underline{\red{i-1\mapsto i}}$ unpairs an $i$ to its left, and is hence column paired. This though is already contradictory, as then $\underline{\red{i-1\mapsto i}}$ would not be able to get unpaired.

    \textbf{Case 4: Both $\boxed{\blue{i}},\mathbf{\yellow{i}}$ are present in $x$}
    
    \textbf{Subcase 4.1:} The $\underline{\red{i-1\mapsto i}}$ unpairs an $\mathbf{\yellow{i}}$ to its right. 
    
    Observe that both the $\boxed{\blue{i}},\mathbf{\yellow{i}}$ cannot be column paired, as they eventually become unpaired. Thus, in order for the $\boxed{\blue{i}}$ to be right of the $\mathbf{\yellow{i}}$ and paired in $f_{i-1}(x)$, we would have to have an $i+1$ right of the $\mathbf{\yellow{i}}$ that the $\boxed{\blue{i}}$ is paired with. This is a contradiction, as, since the $\Circled{\green{i-1\mapsto i}}$ is left of the $\underline{\red{i-1\mapsto i}}$ and hence left of this $i+1$, it would not be able to unpair the $\boxed{\blue{i}}$.

    \textbf{Subcase 4.2:} The $\mathbf{\yellow{i}}$ is left of the $\underline{\red{i-1\mapsto i}}$. 
    
    This assumption implies that the $\underline{\red{i-1\mapsto i}}$ is column paired with an $i+1$ (as this is the only way the $\underline{\red{i-1\mapsto i}}$ can `steal' the $i+1$ pairing with the $\mathbf{\yellow{i}}$). We still have that the $\boxed{\blue{i}},\mathbf{\yellow{i}}$ both cannot be column paired, and so the $i+1$ that the $\boxed{\blue{i}}$ is paired with before getting unpaired is right of the $\mathbf{\yellow{i}}$. As the $\Circled{\green{i-1\mapsto i}}$ is able to unpair the $\boxed{\blue{i}}$, it is right of this $i+1$. We thus in total get the following sequence of elements
    
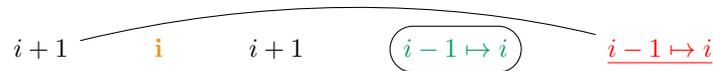
\[\begin{tikzcd}
	{i+1} & {\mathbf{\yellow{i}}} & {i+1} & {\Circled{\green{i-1\mapsto i}}} & {\underline{\red{i-1\mapsto i}}}
	\arrow[curve={height=-20pt}, no head, from=1-1, to=1-5]
\end{tikzcd}\]
    This translates to the following substructure of $x$:
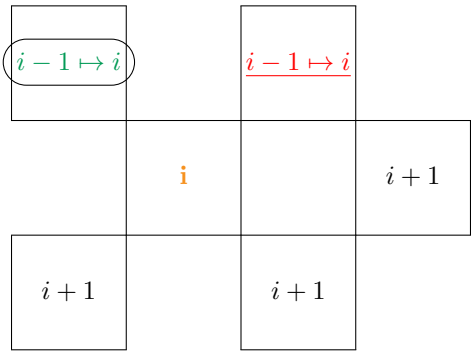
    \begin{center}
    \ytableausetup{boxsize = 4.3em}
\begin{ytableau}
    \Circled{\green{i-1\mapsto i}}&\none&\underline{\red{i-1\mapsto i}}\\
    \none&\mathbf{\yellow{i}}&\none&i+1\\
    i+1&\none&i+1
\end{ytableau}
    \end{center}

    As in Case $3$, the leftmost $i+1$ is implied by the rest of the diagram. This though is a contradiction, as by assumption the $\Circled{\green{i-1\mapsto i}}$ pairs with the right $i+1$. This completes the fourth case.

    As these four cases are exhaustive, this shows that in fact $\boxed{\blue{\hspace{0.2em}i\hspace{0.2em}}}$ is unpaired in $f_{i-1}x$. 
\end{proof}

\begin{lemma}\label{lemma:rightunpairs2}
 If $f_{i+1}x$ contains an unpaired $i$ that is paired in $x$, the rightmost $i$ that is unpaired in $f_{i+1}^*x$ is also unpaired in $f_{i+1}x$.
\end{lemma}
\begin{proof} We will split this proof into two main cases. As before, by assumption $f_{i+1}$ will create a new unpaired $i$. Consider the furthest right such $\boxed{\blue{i}}$ that gets unpaired in $f_{i+1}^*x$. Suppose for the sake of contradiction that the $\yellow{i}$ that gets unpaired by $f_{i+1}$ is not the furthest right $\boxed{\blue{i}}$. First, in order for $\underline{\red{i+1\mapsto i+2}}$ to unpair the $\yellow{i}$, the $\yellow{i}$ must be right of the $\underline{\red{i+1}}$. Consider the $\dbox{\purple{i+1}}$ that the rightmost $\boxed{\blue{i}}$ is paired with. 

\textbf{Case 1:} This $\dbox{\purple{i+1}}$ is left of the furthest right unpaired $\underline{\red{i+1}}$. 

We thus obtain the following ordering of elements:
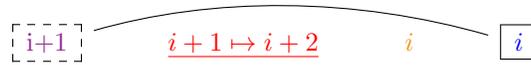
\[\begin{tikzcd}
	{\dbox{\purple{\hspace{0.2em}i+1\hspace{0.2em}}}} & {\underline{\red{i+1\mapsto i+2}}} & {\yellow{i}} & {\boxed{\blue{\hspace{0.2em}i\hspace{0.2em}}}}
	\arrow[curve={height=-18pt}, no head, from=1-1, to=1-4]
\end{tikzcd}\]
Fixing notation, the connected elements are paired. In order for the $\boxed{\blue{i}},\dbox{\purple{i+1}}$ to be paired, they are thus in the same column. Then $\yellow{i}$ must be in the row in between, as if it were in the top row it would be column-paired with an $i+1$. The $\underline{\red{i+1\mapsto i+2}}$ is thus in the same row as the other $\dbox{\purple{i+1}}$. The element above it though is then an $i$, which contradicts the fact that the $\underline{\red{i+1\mapsto i+2}}$ unpaired the $\yellow{i}$. The following diagram shows this:

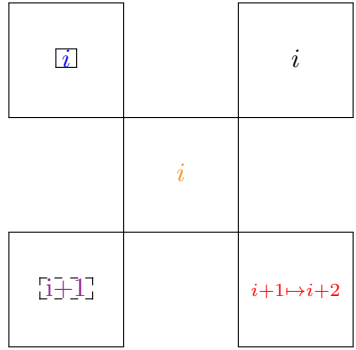
\begin{center}
    \ytableausetup{boxsize = 4.3em}
\begin{ytableau}
    \boxed{\blue{\hspace{0.2em}i\hspace{0.2em}}}&\none&i\\
    \none&\yellow{i}\\
    \dbox{\purple{\hspace{0.2em}i+1\hspace{0.2em}}}&\none&\red{\scriptstyle{i+1\mapsto i+2}}
\end{ytableau}
\end{center}

\textbf{Case 2:} The $\dbox{\purple{i+1}}$ is right of the furthest right unpaired $\underline{\red{i+1}}$. 

The rightmost $\underline{\red{i+1\mapsto i+2}}$ will unpair an $\mathbf{\yellow{i}}$. If the rightmost $\boxed{\blue{i}}$ is paired with an $\dbox{\purple{i+1}}$ left of this $\mathbf{\yellow{i}}$ we would obtain that the $\boxed{\blue{i}},\dbox{\purple{i+1}}$ are column paired. This though is a contradiction as the only way that the $\boxed{\blue{i}}$ can then be unpaired is if $\purple{i+1\mapsto i+2}$ (impossible as this element is right of the furthest right unpaired $\underline{\red{i+1}}$). 

Thus suppose that the $\dbox{\purple{i+1}}$ is right of the $\mathbf{\yellow{i}}$ that becomes unpaired. In order for the rightmost $\boxed{\blue{i}}$ to become unpaired, we must have an $\Circled{\green{i}}$ between the $\dbox{\purple{i+1}},\boxed{\blue{i}}$ which is paired with some $i+1$ left of the $\underline{\red{i+1}}$ (otherwise no unpaired $i+1\mapsto i+2$ would be able to 'reach' the $\boxed{\blue{i}}$ to unpair it). As a sequence, we can summarize this as follows: 
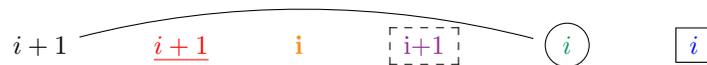
\[\begin{tikzcd}
	{i+1} & {\underline{\red{i+1}}} & {\mathbf{\yellow{i}}} & {\dbox{\purple{\hspace{0.2em}i+1\hspace{0.2em}}}} & {\Circled{\green{i}}} & {\boxed{\blue{\hspace{0.2em}i\hspace{0.2em}}}}
	\arrow[curve={height=-18pt}, no head, from=1-1, to=1-5]
\end{tikzcd}\]
Thus the  $\Circled{\green{i}}$, $i+1$ are column paired, and so we get the following diagram:

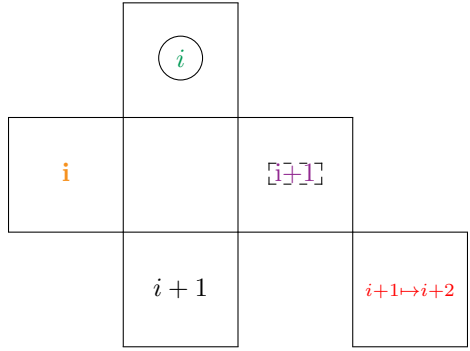
\begin{center}
    \ytableausetup{boxsize = 4.3em}
\begin{ytableau}
    \none&\Circled{\green{i}}&\none&\none\\
    \mathbf{\yellow{i}}&\none&\dbox{\purple{\hspace{0.2em}i+1\hspace{0.2em}}}\\
    \none&i+1&\none&\red{\scriptstyle{i+1\mapsto i+2}}
\end{ytableau}
\end{center}

Over the $\underline{\red{i+1\mapsto i+2}}$ we must have an $i$, but this is a contradiction as the $\underline{\red{i+1\mapsto i+2}}$ by assumption unpairs the middle $\mathbf{\yellow{i}}$. 
This finally completes the proof.
\end{proof}

\subsection{Proof of the extension property}
\begin{proposition}[Extension Property]
Let $a\geq b$ and let $x\in X$ be extremal so that $e_i(x)=0$ for all $i\leq n+1$. Suppose that $F_{n,a}F_{n+1,b}(x), F_{n+1,b-1}(x)\in X$. Then $F_{n,a}F_{n+1,b-1}(x)\in X$.
\end{proposition}
Before proving this, we start with a useful lemma.
\vspace{1em}

\begin{lemma}\label{lemma:unpairedsubset}
    Let $x$ be such that $e_i(x)=0$ for all $i\geq a$. Then the set of unpaired $n$s in $f_{n-1}^*\dots f_{a+1}^*(x)$ is a subset of the unpaired $n$s in $f_{n-1}^*\dots f_a^*(x)$.
\end{lemma}
\begin{proof}
    We induct on $n$. First observe that since $f_{a}$ does not pair any $a+1$s, the unpaired $a+1$s in $x$ are a subset of the unpaired $a+1$s in $f_a^*(x)$. We can thus by induction assume that the unpaired $n-1$s in $f_{n-2}^*\dots f_{a+1}^*(x)$ are a subset of the unpaired $n-1$s in $f_{n-2}^*\dots f_a^*(x)$. Note that $f_{n-2}^*\dots f_{a+1}^*(x)$ and $f_{n-2}^*\dots f_a^*x$ contain the same $n+1$s and $n$s. Thus by the inductive hypothesis, since lowering more $n-1\to n$ can only unpair more $n$s, the unpaired $n$s in $f_{n-1}^*\dots f_{a+1}^*(x)$ are a subset of the unpaired $n$s in $f_{n-1}^*\dots f_a^*(x)$.
\end{proof}

We now proceed to the main proof.

\begin{proof}[Proof of extension property]
As this proof is long, we sketch a brief road map first to make it more tractable. This proof splits into three main parts.\begin{enumerate}
    \item Construct an auxiliary tableau $y$, operator $\tilde{f}_{n-1}$ acting on $y$, and show that $y$ is a valid shuffle tableaux with $e_n(y) = 0$. To do this, we also introduce a couple other useful auxiliary tableaux $z$, $z'$, and $z''$.
    \item Prove a useful result about $y$ (this is the main 'tableaux-chasing' part of the argument), namely Claim \ref{claim:mainclaim_y}.
    \item Relate all of our work back to the proof of the extension property, and show that it follows quickly from Claim \ref{claim:mainclaim_y} about $y$. 
\end{enumerate}

Suppose for the sake of contradiction that $F_{n,a}F_{n+1,b}(x),F_{n+1,b-1}(x)\in X$ while $F_{n,a}F_{n+1,b-1}(x)\not\in X$ where $a,b\in \N$ are such that $a\geq b$. Lemma \ref{lemma:symbolicmanipulation} tells us that
\[
F_{n,a}F_{n+1,b-1}(x) = f_{n+1}^*(f_n^*\dots f_a^*f_{n+1}^*\dots f_{a+1}^*f_{a-1}^*(f_{a-2}^*\dots f_{b-1}^*))(x).
\]

Since by assumption $F_{n+1,b-1}(x)\in X$, $F_{n,a}F_{n+1,b-1}$ does not have any $n+1\mapsto n+2$ in a row with flag less than or equal to $n+1$. Thus, the last $f_{n+1}^*$ in $f_{n+1}^*(f_n^*\dots f_a^*f_{n+1}^*\dots f_{a+1}^*f_{a-1}^*(f_{a-2}^*\dots f_{b-1}^*))(x)$ cannot cause the entry that disobeys the flag condition. We hence deduce that
    \[
(f_n^*\dots f_a^*)(f_{n+1}^*\dots f_{a+1}^*f_{a-1}^*f_{a-2}^*\dots f_b^*f_{b-1}^*)(x)\not\in X.
    \]


We will consider the following auxiliary tableaux for convenience:

\begin{enumerate}
    \item Let $z:=f_{a-2}^*\dots f_{b-1}^*(x)$
    \item Let $z':=f_{n-2}^*\dots f_a^*\mathbf{f_{a-1}^*}f_{n-1}^*\dots f_{a+1}^*(z)$
    \item Let $z'':=f_{n-2}^*\dots f_a^*f_{n-1}^*\dots f_{a+1}^*(z)$ (similar to $z'$, except we delete the bolded $f_{a-1}^*$ from the end of the first descending chain).
\end{enumerate}

Note that by Lemma \ref{lemma:unpairedsubset}, the unpaired $n-1$s in $f_n^*(z'')$ form a subset of the unpaired $n-1$s in $f_n^*(z')$. 


Thus, we may define an auxiliary shuffle tableau $y$ as follows:

\begin{enumerate}
    \item Start with $f_n^*(z')$, and raise each $n-1\mapsto n$ in $f_n^*(x')$ iff the $n-1$ is unpaired in $f_n^*(z'')$,
    \item We also define a `quasi-lowering operator' $\tilde{f}_{n-1}^*$ that acts on $y$ (as a purely notational tool). This $\tilde{f}_{n-1}^*$ acts as follows: first, make each of the $n$s we added in the step above \textit{inert}, in the sense that when we apply $\tilde{f}_{n-1}^*$ to $y$, we first delete each of the new $n$s created in the last step from the reading word. 
\end{enumerate}
Succinctly, in step (ii), we disallow any of the new $n$s from pairing with any $n-1$s. Intuitively, considering $y$ allows us to only consider the new $n$s that get created by appending the $f_{a-1}^*$ in the expression defining $z'$. We now prove a few properties of this auxiliary $y$. 


\begin{claim} $y$ is a valid shuffle tableau. 
\end{claim}
We must check that the columns of $y$ are strictly increasing and the rows are weakly increasing. First, consider the columns. As previously noted, if $n-1$ is unpaired in $f_n^*(z'')$, it will also be unpaired in $f_n^*(z')$. Thus, any $n-1$ that is unpaired in both $f_n^*(z'')$ and $f_n^*(z')$ cannot be column paired. So, after raising all of these $n-1$s, the columns will still be strictly increasing. 

Now, consider the rows of $y$. Consider an $\underline{\red{n-1}}$ in $f_n^*(z')$ that is also unpaired in $f_n^*(z'')$. Suppose we have some $n-1$ right of it in $f_n^*(z')$. As $\underline{\red{n-1}}$ is present in $f_n^*(z'')$, this $n-1$ is present (rows of $f_n^*(z'')$ are weakly increasing) in $f_n^*(z'')$. Moreover, since it is right of the $\underline{\red{n-1}}$ which is unpaired in $f_n^*(z'')$, it is unpaired in $f_n^*(z'')$ (rows of $f_{n-1}^*f_n^*(z'')$ are weakly increasing). Hence in the construction of $y$, we raise this right $n-1$ to an $n$. Thus the rows will still weakly increase, and so $y$ is a valid shuffle tableau.

\begin{claim} $e_n(y)=0$:
\end{claim}
We seek to show that every $n+1$ in $y$ is paired with an $n$. To see this, note that the set of $n$s and $n+1$s in $y$ equals the set of $n$s and $n+1$s in $f_{n-1}^*f_n^*z''$. Since $e_n(z'')=e_{n-1}(z'')=0$ (to see this, repeatedly apply Lemma \ref{e_i=0}), we by Lemma \ref{e_i=0} have that $e_n(f_{n-1}^*f_n^*z'')=0$. Thus $e_n(y)=0$.

This completes the first part of our proof which was to construct $y$ and $\tilde{f}_{n-1}$ and verify that $y$ is in fact a tableaux with $e_n(y) = 0$. We now move on to the second part: proving Claim \ref{claim:mainclaim_y}.

Consider the rightmost unpaired $n$ in $\tilde{f}_{n-1}^*f_{n+1}^*y$. Note though that because of this modification to $\tilde{f}_{n-1}^*$, $f_n^*\tilde{f}_{n-1}^*f_{n+1}^*y=f_n^*f_{n-1}^*f_{n+1}^*f_n^*z'\not\in X$. We thus seek to prove the following claim:

\begin{claim}\label{claim:mainclaim_y}
If $f_n^*\tilde{f}_{n-1}^*f_{n+1}^*(y)\not\in X$, either $f_n^*f_{n+1}^*(y)\not\in X$ or $f_n^*\tilde{f}_{n-1}^*(y)\not\in X$.
\end{claim}

Suppose for the sake of contradiction that the rightmost unpaired $n$ in $\tilde{f}_{n-1}^*f_{n+1}^*(y)$ is not present in either $f_{n-1}^*(y)$ or $f_{n+1}^*(y)$. We split this into two cases.

\textbf{Case 1: The rightmost unpaired $n$ in $\tilde{f}_{n-1}^*f_{n+1}(y)$ is not present in $y$:}\\
Suppose first that the rightmost unpaired $n$ in $\tilde{f}_{n-1}^*f_{n+1}^*(y)$ is not present in $y$. It is then present but paired in $\tilde{f}_{n-1}^*y$. Since $e_n(y)=0$, this $n-1\mapsto n$ unpairs an $n$ in $y$. The $n$ that gets unpaired must by assumption be to its left. We thus obtain the following sequence of elements,
\[\begin{tikzcd}
	{n+1} & n & {n-1\mapsto n}
\end{tikzcd}\]
where the middle $n$ was paired in $y$ but is unpaired in $\tilde{f}_{n-1}^*y$ and the $n-1\mapsto n$ is in the same column as the $n+1$. We thus in $y$ have
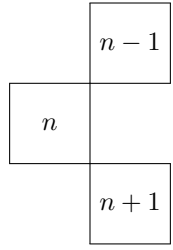
\begin{center}
\ytableausetup{boxsize = 3em}
\begin{ytableau}
    \none &n-1\\
    n&\none\\
    \none &n+1
\end{ytableau}
\end{center}
This though is impossible as the $n-1\mapsto n$ is bad and thus at the end of its row by Lemma \ref{lemma:easyendofrow} and so this would contradict $\lambda$ being strict.

\textbf{Case 2: The rightmost unpaired $n$ is present in $y$:}\\
Suppose now that the rightmost unpaired $n$ in $\tilde{f}_{n-1}^*f_{n+1}^*y$ is present in $y$. Consider the $\underline{\red{n-1\mapsto n}}$ which unpairs the rightmost unpaired $n$ in $\tilde{f}_{n-1}^*f_{n+1}^*y$. The $\underline{\red{n-1\mapsto n}}$ is by assumption paired with an $\boxed{\blue{n+1}}$ in $f_{n-1}^*f_{n+1}^*y$ (since it unpairs an $n$ that was present in $y$). On the other hand, it is also paired with an $\Circled{\green{n+1}}$ in $\tilde{f}_{n-1}^*y$ (since applying $f_{n+1}^*$ will not pair it). Since $e_n(y)=0$, the $\underline{\red{n-1\mapsto n}}$ unpairs an $n$ in $y$. This $n$ must be left of $\boxed{\blue{n+1}}$ as the rightmost unpaired $n$ in $\tilde{f}_{n-1}^*f_{n+1}^*y$ is not yet unpaired in $f_{n-1}^*y$. We thus get the following sequence of elements
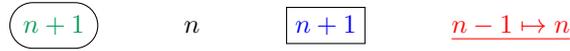
\[\begin{tikzcd}
	{\Circled{\green{n+1}}} & n & {\boxed{\blue{n+1}}} & {\underline{\red{n-1\mapsto n}}}
\end{tikzcd}\]
Placing these in $y$ we then get
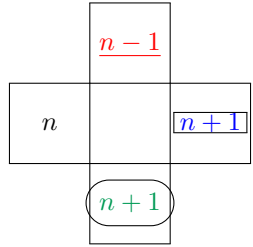
\begin{center}
\ytableausetup{boxsize = 3em}
\begin{ytableau}
    \none &\underline{\red{n-1}}\\
    n&\none&\boxed{\blue{\hspace{0.2em} n+1 } \hspace{0.2em}}\\
    \none &\Circled{\green{n+1}}
\end{ytableau}
\end{center}
Since the $\underline{\red{n-1\mapsto n}}$ and $\Circled{\green{n+1}}$ are not paired in $\tilde{f}_{n-1}^*f_{n+1}^*(y)$, the $\Circled{\green{n+1}}$ is unpaired (and so maps to an $n+2$ in $\tilde{f}_{n-1}^*f_{n+1}^*(y)$). On the other hand, the $\boxed{\blue{n+1}}$ is paired as it is present in $f_{n+1}^*y$. Consider the $n+2$ it is paired with. Since the $\Circled{\green{n+1}}$ is unpaired, the $\boxed{\blue{n+1}}$ is either column paired or paired with an $n+2$ in the same row as the $\Circled{\green{n+1}}$. Suppose first that it is column paired. We thus obtain
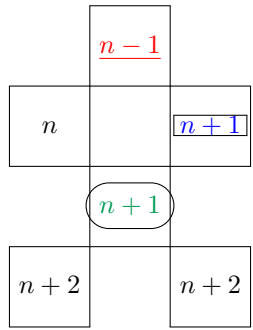
\begin{center}
\ytableausetup{boxsize = 3em}
\begin{ytableau}
    \none &\underline{\red{n-1}}\\
    n&\none&\boxed{\blue{\hspace{0.2em}n+1}\hspace{0.2em}}\\
    \none &\Circled{\green{n+1}}\\
    n+2&\none&n+2
\end{ytableau}
\end{center}
Note that the $n$ is over an $n+2$ as it gets unpaired in $f_{n-1}^*y$. Since the $\Circled{\green{n+1}}$ is unpaired in $y$, the $n+2$ is paired with an $n+1$ in the middle row. Since $e_n(y) = 0$, this $n+1$ is also paired in $y$ with an entry left of the $n$ (otherwise the $\underline{\red{n-1\mapsto n}}$ would unpair an $n$ right of the $\boxed{\blue{n+1}}$). We thus have the following infinite repeating sequence
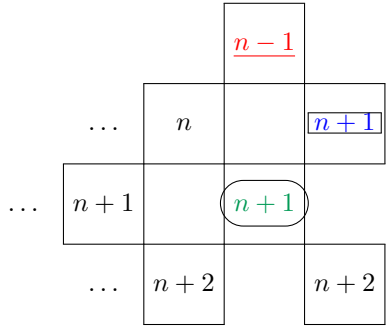
\begin{center}
\ytableausetup{boxsize = 3em}
\begin{ytableau}
    \none&\none&\none &\underline{\red{n-1}}\\
    \none&\none[\dots]&n&\none&\boxed{\hspace{0.2em}\blue{n+1}\hspace{0.2em}}\\
    \none[\dots]&n+1&\none &\Circled{\green{n+1}}\\
    \none&\none[\dots]&n+2&\none&n+2
\end{ytableau}
\end{center}
This though is a contradiction, and so the $n+1$ is not column paired. We thus instead have
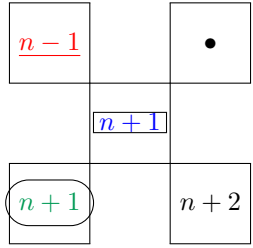
\begin{center}
\ytableausetup{boxsize = 3em}
\begin{ytableau}
    \underline{\red{n-1}}&\none&\bullet\\
    \none &\boxed{\hspace{0.2em}\blue{n+1}\hspace{0.2em}}\\
    \Circled{\green{n+1}}&\none&n+2
\end{ytableau}
\end{center}
We can assume that the $\underline{\red{n-1\mapsto n}}$ is the rightmost unpaired $n-1$ and so $\bullet=n$. We can then without loss of generality assume that the $n$ is the $n$ that gets unpaired, as if it isn't we would have another $n+1$ in the middle row and so we can repeat the construction rightwards. This though is a contradiction, as since $\bullet=n$ is bad (and thus at the end of its row by Lemma \ref{lemma:easyendofrow}) and column paired, $\lambda$ is not strict. This completes case two, and so we have proven our third main claim.

We end by relating $y$ back to $z',z''$ and finally our main problem.
Since $f_n^*\tilde{f}_{n-1}^*f_{n+1}^*(y)=f_n^*f_{n+1}^*f_{n-1}^*f_n^*(z')\not\in X$, we obtain that either $f_n^*\tilde{f}_{n-1}^*(y)\not\in X$ or $f_n^*f_{n+1}^*(y)\not\in X$. We seek to show that $f_n^*\tilde{f}_{n-1}^*(y)\not\in X$. Suppose otherwise; then $f_n^*f_{n+1}^*(y)\not\in X$. 

\begin{claim}
    The set of $n$s and $n+1$s in $f_{n+1}^*(y)$ equals the set of $n$s and $n+1$s in $f_{n+1}^*f_{n-1}^*f_n^*(z'')$.
\end{claim}
We start with the set of $n$s. The set of $n$s in $f_{n+1}^*(y)$ equals the set of $n$s in $y$. Note also that by definition, the set of $n$s in $f_n^*(z')$ equals the set of $n$s in $f_n(z'')$. Finally, the set of $n$s in $y$ equals the set of $n$s in $f_n^*(z')$ along with the $n$s obtained by raising $n-1\mapsto n$ for every $n-1$ in $f_n^*(z')$ that is also unpaired in $f_n(z'')$. Thus, by construction, the set of $n$s in $y$ equals the set of $n$s in $f_{n-1}^*f_n^*(z'')$, which in turn equals the set of $n$s in $f_{n+1}^*f_{n-1}^*f_n^*(z'')$. Now consider the set of $n+1$s. The set of $n+1$s in $f_{n+1}^*(y)$ equals the set of $n+1$s in $f_{n+1}^*f_n^*(z')$, which in turn equals the set of $n+1$s in $f_{n+1}^*f_n^*(z'')$. 

Thus, from the claim, we get that $f_n^*f_{n+1}^*f_{n-1}^*f_n^*(z'')\not\in X$. Recall that
\[
f_n^*f_{n+1}^*f_{n-1}^*f_n^*(z'') = f_n^*f_{n+1}^*f_{n-1}^*f_n^*(f_{n-2}^*\dots f_a^*f_{n-1}^*\dots f_{a+1}^*(z))=f_n^*\dots f_a^*f_{n+1}^*f_n^*f_{n-1}^*\dots f_{a+1}^*(z)=F_{n,a}F_{n+1,a+1}(z).
\]
By commuting the operators $f_{a-2}^*\dots f_{b-1}^*$ in $z$ left, we then get that the set of $n+1$s in $F_{n,a}F_{n+1,a+1}(z)$ equals the set of $n+1$s in $F_{n,a}F_{n+1,a+1}(x)$, and so $F_{n,a}F_{n+1,a+1}(x)\not\in X$. Note though that by Lemma $\ref{lemma:symbolicmanipulation}$, we have
\[
F_{n,a}F_{n+1,b}(x) = f_{n+1}^*(F_{n,b}F_{n+1,a+1})(x) \in X\implies F_{n,b}F_{n+1,a+1}(x)\in X\implies F_{n,a}F_{n+1,a+1}(x)\in X,
\]
which is a contradiction.

Thus $f_n^*\tilde{f}_{n-1}^*(y)\not\in X$. As already noted though, $f_n^*\tilde{f}_{n-1}^*(y)=f_n^*f_{n-1}^*f_n^*(z')\not\in X$. By commuting operators in $f_n^*\dots f_{a+1}^*$ left and applying the second braid relation when necessary, we get
\[
f_n^*f_{n-1}^*f_n^*(z') = f_n^*\dots f_{a-1}^* f_n^*\dots f_{a+1}^*(z) = f_{n-1}^*\dots f_{a}^*(f_n^*\dots f_{a-1}^*(z)) = f_{n-1}^*\dots f_{a}^*F_{n,b-1}(x)\not\in X.
\]
Since $f_{n-1}^*f_n^*(z')$ by assumption contains a bad $n$, this implies that $F_{n,b-1}(x)\not\in X$. This though contradicts our initial assumptions, and so is a final contradiction. We thus obtain that $F_{n,a}F_{n+1,b}(x)\in X$, as hoped.
\end{proof}
\vspace{1em}
\begin{remark}
    Note that in the above proof, steps (i) (construction of $y$) and (iii) (relating back to the original claim) are quite general and do not require much about the specific structure of shuffle tableaux. Because of this, in many cases, the proof of the extension property essentially reduces to proving the simpler statement that if
    \begin{enumerate}
        \item $x\in X$ is extremal with $e_n(x) = 0$
        \item $f_n^*f_{n-1}^*(x)\in X$
        \item $f_n^*f_{n+1}^*(x)\in X$
    \end{enumerate}
    then $f_n^*f_{n+1}^*f_{n-1}^*(x)\in X$. This may be useful for some applications of our axioms.
\end{remark}



\subsection{Proof of the gluing property}\label{section:locprop2}
\begin{proposition}[Gluing Property]
    Let $k,m,a\in \N$ and $x\in X$ be extremal so that $e_i(x)=0$ for all $a-m\leq i\leq a+k$. Suppose that $f_{a+k}^*\dots f_{a}^*(x)\in X$ and $f_{a-m}^*\dots f_{a-1}^*(x)\in X$. Then either $f_{a+k}^*\dots f_{a-1}^*(x)\in X$ or $f_{a-m}^*\dots f_{a-1}^*f_a^*(x)\in X$.
\end{proposition}
\begin{proof}
    Suppose that $f_{a+k}^*\dots f_{a-1}^*(x)\not\in X$. We show that $f_{a-m}^*\dots f_{a-1}^*f_a^*(x)\in X$. Inductively assume that $f_{a-m+1}^* \dots f_{a-1}^*f_a^*(x) \in X$. Suppose for the sake of contradiction that $f_{a-m}^*f_{a-m+1}^* \dots f_{a-1}^*f_a^*(x) \not\in X$. Thus, in $f_{a-m+1}^* \dots f_{a-1}^*f_a^*(x)$, there is a `bad' unpaired $a-m$ in a row of flag $a-m$. Note that since $\lambda$ is strict, the bad $a-m\mapsto a-m+1$ cannot have an element directly underneath it in the same column (following lemma \ref{lemma:easyendofrow}). Our final contradiction will be that this rightmost bad $a-m$ must have an element directly underneath it. 
    
    We start with some initial observations about the structure of the tableau $x$. Using these observations, we seek to show that $x$ exhibits one of the following two substructures shown below. This substructure will serve as the base case of a larger structure, which we will inductively show is contained inside $x$. At the end of this larger structure will be our `bad' $a-m$ directly above another element, our desired contradiction. 
    \begin{center}
    Case $1$: \ytableausetup{boxsize = 3em}
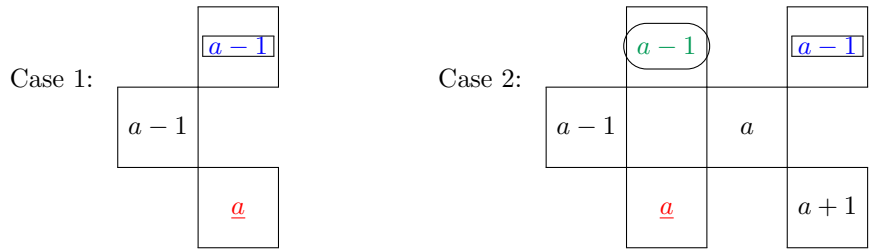
\begin{ytableau}
    \none &\boxed{\blue{\hspace{0.2em}a-1}\hspace{0.2em}}\\
    a-1&\none\\
    \none &\underline{\red{a}}
\end{ytableau}  \hspace{2cm}Case $2$:   \ytableausetup{boxsize = 3em}
\begin{ytableau}
    \none &\Circled{\green{a-1}}&\none&\boxed{\blue{\hspace{0.2em}a-1\hspace{0.2em}}}\\
    a-1&\none&a\\
    \none &\underline{\red{a}}&\none&a+1
\end{ytableau}
\end{center}

    Here, the blue $\boxed{\blue{a-1}}$ (in the top right box) is the rightmost unpaired $a-1$ in $f_a^*x$. The red $\underline{\red{a}}$ (in the bottom left box) is the rightmost unpaired $a$ in $x$. In both cases, the black $a-1$ in the middle row is the rightmost unpaired $a-1$ in $x$. Note that in both of the above diagrams, position is relative but not precise (ie column paired entries are indeed column paired but there could be more entries in $x$ not shown).
    
    \begin{remark}
        Proving the base case will take a bit of work, so the reader might want to skip to the inductive case to see how we will build off of these diagrams. 
    \end{remark}
    
    \begin{proof}[Proof of Base Case]
    We start with three general claims that apply to both cases. 

    \begin{claim} The rightmost unpaired $a-1$ in $f_a(x)$ is right of the rightmost unpaired $a-1$ in $x$.
    \end{claim}
    Suppose that $f_a^*(x),x$ had the same rightmost unpaired $a-1$. Note that $x,f_a^*(x)$ have the same set of $a-1$s and $a-2$s. Thus $f_{a-1}^*f_a^*(x)$ has the same rightmost unpaired $a-2$ as $f_{a-1}^*(x)$. Then by Lemma \ref{lemma:rightunpairs1}, as $f_{a-1}^*f_a^*(x),f_{a-1}^*(x)$ contain the same rightmost unpaired $a-2$ as $f_{a-1}f_a(x),f_{a-1}(x)$ respectively, $f_{a-1}^*f_a^*(x),f_{a-1}^*(x)$ contain the same rightmost unpaired $a-2$. Continuing this argument inductively, we obtain that $f_{a-m+1}^*\dots f_{a-1}^*f_a^*(x),f_{a-m+1}^*\dots f_{a-1}^*(x)$ contain the same rightmost unpaired $a-m$, which contradicts the fact that $f_{a-m}^*f_{a-m+1}^*\dots f_{a-1}^*f_a^*(x)\not\in X$ while $f_{a-m}^*f_{a-m+1}^*\dots f_{a-1}^*(x)\in X$. Thus the rightmost unpaired $a-1$ in $f_a^*(x)$ is strictly right of the rightmost unpaired $a-1$ in $x$.
    
    \begin{claim} 
    The rightmost unpaired $a$ in $f_{a-1}^*(x)$ is right of the rightmost unpaired $\underline{\red{a}}$ in $x$. 
    \end{claim}
    
    Similarly, by Lemma \ref{lemma:rightunpairs2}, since $f_{a+k}^*\dots f_{a-1}^*(x)\not\in X$ while $f_{a+k}^*\dots f_{a}^*(x)\in X$, the rightmost unpaired $a$ in $f_{a-1}^*(x)$ is strictly right of the rightmost unpaired $a$ in $x$. 
    
    \begin{claim}
    The rightmost unpaired $a-1$ in $x$ is right of the rightmost unpaired $\underline{\red{a}}$ in $x$.
    \end{claim}
    To see this, note that the rightmost unpaired $a-1$ in $x$ unpairs the rightmost unpaired $\boxed{\blue{a}}$ in $f_{a-1}^*(x)$. This $\boxed{\blue{a}}$ is right of the rightmost unpaired $\underline{\red{a}}$ in $x$. Thus, in order for the $a-1\mapsto a$ to unpair the $\boxed{\blue{a}}$ in $f_{a-1}^*(x)$, it must be right of the $\underline{\red{a}}$. 

    Putting these three claims together, we thus have the following sequence in the reading word of $x$:
    
    \[\begin{tikzcd}
	{\underline{\red{a}}} & {a-1} & {\boxed{\blue{a-1}}}
    \end{tikzcd}\]
    
    Consider the $a$ that the right $\boxed{\blue{a-1}}$ is paired with. If the $\boxed{\blue{a-1}}$ is column paired, since it gets unpaired by the $\underline{\red{a\mapsto a+1}}$, it must be column paired with the $\underline{\red{a}}$ (otherwise it would not get unpaired). We thus obtain Case $1$:

    \begin{center}
    \ytableausetup{boxsize = 3em}
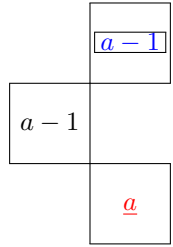
\begin{ytableau}
    \none &\boxed{\blue{\hspace{0.2em}a-1}\hspace{0.2em}}\\
    a-1&\none\\
    \none &\underline{\red{a}}
\end{ytableau}
\end{center}
    
    We now move on to case $2$, which takes a bit more work. If the $\boxed{\blue{a-1}}$ is not column paired, it is paired with an $a$ right of the left $a-1$. We thus get the following sequence in $x$:
    \[\begin{tikzcd}
	{\underline{\red{a}}} & {a-1} &a & {\boxed{\blue{\hspace{0.2em}a-1}\hspace{0.2em}}}
    \end{tikzcd}\]
    Thus, in order for the $\underline{\red{a\mapsto a+1}}$ to unpair the $\boxed{\blue{a-1}}$, the $\underline{\red{a\mapsto a+1}}$ must be paired with an $\Circled{\green{a-1}}$ between the $a,\boxed{\blue{\hspace{0.2em}a-1}\hspace{0.2em}}$. We thus have
    \[\begin{tikzcd}
	{\underline{\red{a}}} & {a-1} &a &\Circled{\green{a-1}}& {\boxed{\blue{\hspace{0.2em}a-1\hspace{0.2em}}}}
    \end{tikzcd}\]
    Since the left $a-1$ is unpaired, the $\underline{\red{a}},\Circled{\green{a-1}}$ are column paired. Note also that the $\boxed{\blue{a-1}}$ is paired; since the left $a-1$ is not paired but it gets unpaired by the $\underline{\red{a\mapsto a+1}}$, it must be paired with an $a$ in the middle row. We thus obtain the following substructure:

    \begin{center}
    \ytableausetup{boxsize = 3em}
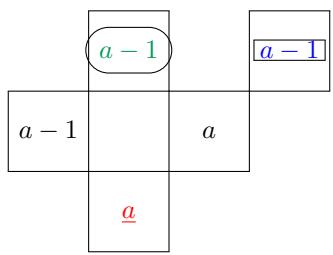
\begin{ytableau}
    \none &\Circled{\green{a-1}}&\none&\boxed{\blue{\hspace{0.2em}a-1}\hspace{0.2em}}\\
    a-1&\none&a\\
    \none &\underline{\red{a}}&\none
\end{ytableau}
\end{center}

     Note here that the $\boxed{\blue{a-1}}$ is just somewhere right of the $\Circled{\green{a-1}}$ and not necessarily in the same row (although we will later show that they are in fact in the same row). 

     We now prove one last useful claim about the structure of $x$. 

    \begin{claim}\label{claim:colpair}
    The rightmost unpaired $a-1$ in $x$ is not in the same column as an $a+1$.
    \end{claim}
    Suppose that the $a-1$ is in the same column as an $a+1$. Then the $a+1$ is paired with an $a$ left of the $\underline{\red{a}}$. We thus have

    \begin{center}
    \ytableausetup{boxsize = 3em}
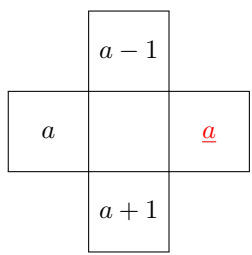
\begin{ytableau}
    \none&a-1&\none\\
    a&\none &\underline{\red{a}}&\none\\
    \none&a+1
\end{ytableau}
\end{center}

    This though is a contradiction as then the $a-1\mapsto a$ would unpair an $a$ left of the rightmost unpaired $a$ in $x$, which contradicts the fact that $f_{a-1}^*(x)$ contains an unpaired $a$ right of the rightmost unpaired $a$ in $x$. This completes the proof of the claim.

    Note that Claim \ref{claim:colpair} applies also in Case $1$; we will use this in the inductive case.

    We can now complete our picture in Case $2$ of the base case. Observe that by assumption, the right $a$ is paired while the $\underline{\red{a}}$ is unpaired. Consider the $a+1$ that the right $a$ is paired with. By Claim 3, the $a$ cannot be column paired. Thus the $a+1$ is right of the $\underline{\red{a}}$. We then have that the $a+1$ has an $a-1$ above it. We can assume without loss of generality that the $a-1$ above the $a+1$ is the rightmost $\boxed{\blue{a-1}}$ that gets unpaired. If it is not we would have another $a$ in the middle row and so we could repeat the above construction. 
    
    Thus in the second case we have the following substructure:

    \begin{center}
    \ytableausetup{boxsize = 3em}
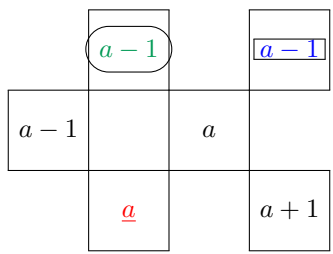
\begin{ytableau}
    \none &\Circled{\green{a-1}}&\none&\boxed{\blue{\hspace{0.2em}a-1\hspace{0.2em}}}\\
    a-1&\none&a\\
    \none &\underline{\red{a}}&\none&a+1
\end{ytableau}
\end{center}

    This completes the proof of what will become our base case.
    \end{proof}

    With this initial rather restrictive substructure in hand, we move on to proving a similar inductive structure for $x$; this will yield our final contradiction. We first fix a little notation. We have shown two possible cases for the structure of $x$; as these two structures will function exactly the same in our induction, we summarize these cases in one diagram as follows:

    
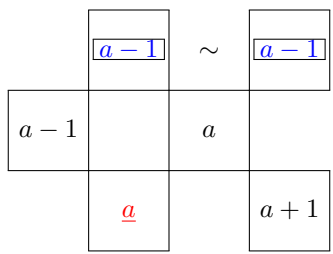
\begin{center}
    \begin{ytableau}
    \none &\boxed{\blue{\hspace{0.2em}a-1}\hspace{0.2em}}&\none[\sim]&\boxed{\blue{\hspace{0.2em}a-1\hspace{0.2em}}}\\
    a-1&\none&a\\
    \none &\underline{\red{a}}&\none&a+1
    \end{ytableau}
    \end{center}

    Here, the $\sim$ indicates that the two $\boxed{\blue{a-1}}$s may in fact be "the same" (ie we only have the left $\boxed{\blue{a-1}}$). Similarly, the $a,a+1$ under the right $\boxed{\blue{a-1}}$ would of course not exist in the first case. Although this notation is a bit unclear, we will see that in either case, the induction proceeds the same, and so we can consider the two cases together. As a final piece of notation, set $x_i:=f_{a-i}^*\dots f_{a-1}^*(x)$ and $y_i:=f_{a-i}^*\dots f_{a-1}^*f_a^*(x)$. 

    Observe, by the definition of the substructure, that in the middle row, we have to have an unpaired $a-1$. On the other hand, in the top row and every row above it, every $a-1$ must be paired (in $x$). Moreover, in the middle row and every row above it, every $a$ must be paired. These will be important assumptions throughout the inductive step.

    We now sketch the larger inductive structure we seek to find in $x$: 

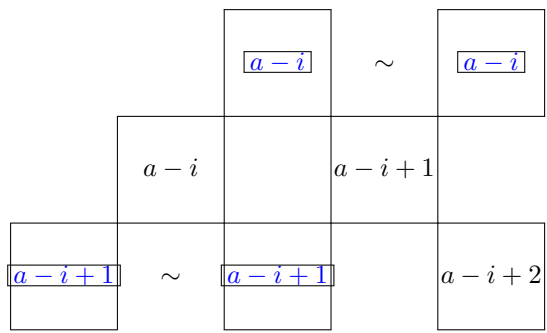
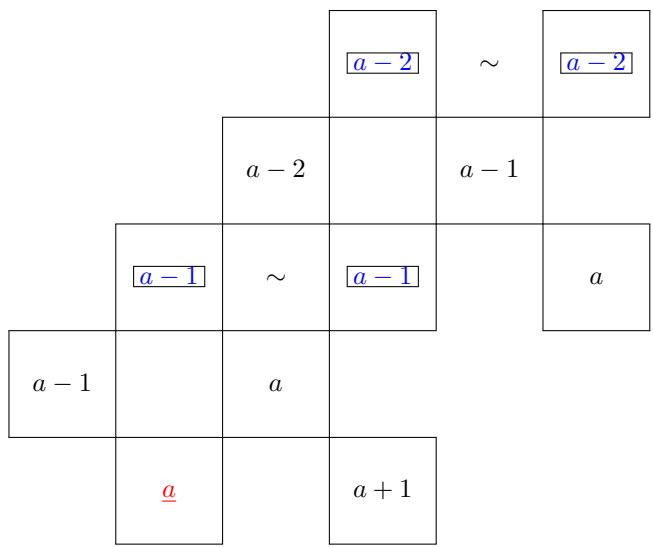
    \ytableausetup{boxsize = 4em}
\begin{ytableau}
    \none&\none&\none&\none&\none&\none&\none&\none&\none&\boxed{\blue{\hspace{0.2em}a-i}\hspace{0.2em}}&\none[\sim]&\boxed{\blue{\hspace{0.2em}a-i\hspace{0.2em}}}\\
    \none&\none&\none&\none&\none&\none&\none&\none&a-i&\none&a-i+1\\
    \none&\none&\none&\none&\none&\none&\none&\boxed{\blue{\hspace{0.2em}a-i+1}\hspace{0.2em}}&\none[\sim]&\boxed{\blue{\hspace{0.2em}a-i+1}\hspace{0.2em}}&\none&a-i+2\\
    \none&\none&\none&\none&\none&\none&\none[\iddots]\\
    \none&\none&\none&\boxed{\blue{\hspace{0.2em}a-2\hspace{0.2em}}}&\none[\sim]&\boxed{\blue{\hspace{0.2em}a-2\hspace{0.2em}}}\\
    \none&\none&a-2&\none&a-1\\
    \none &\boxed{\blue{\hspace{0.2em}a-1}\hspace{0.2em}}&\none[\sim]&\boxed{\blue{\hspace{0.2em}a-1}\hspace{0.2em}}&\none&a\\
    a-1&\none&a\\
    \none &\underline{\red{a}}&\none&a+1
\end{ytableau}

    As before, the blue entries $\boxed{\blue{a-j}}$ indicate the rightmost unpaired $a-j$ in $y_{j-1}$. The middle $a-j$s (in black) indicate the rightmost unpaired $a-j$ in $x_{j-1}$. Moreover, although we show the column pairings over the right $\boxed{\blue{a-j}}$, they can be over either $\boxed{\blue{a-j}}$ connected by a $\sim$.
    
    As already argued, all of these hypotheses hold in the base case. Now that we have presented the inductive structure which we seek to show, we prove the inductive case. 
    \begin{proof}[Inductive Step]
    We inductively build the structure upwards. We can assume that the rightmost unpaired $\boxed{\blue{a-i}}$ in $y_{i-1}$ is strictly right of the rightmost unpaired $a-i$ in $x_{i-1}$.
    
    Consider first the rightmost unpaired $a-i$ in $x_{i-1}$. By induction we have 

    \begin{center}
    \ytableausetup{boxsize = 4.3em}
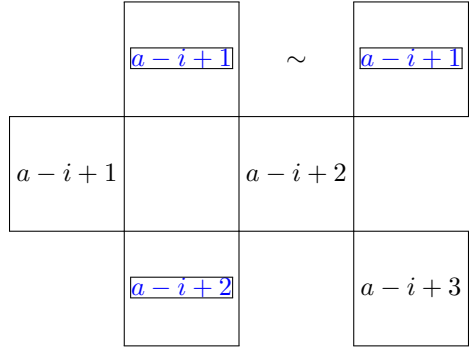
\begin{ytableau}
    \none &\boxed{\blue{a-i+1}}&\none[\sim]&\boxed{\blue{a-i+1}}\\
    a-i+1&\none&a-i+2\\
    \none &\boxed{\blue{a-i+2}}&\none&a-i+3
\end{ytableau}
\end{center}

   We inductively have that the middle $a-i+1$ is unpaired in $x_{i-2}$. Consider the $a-i$ unpaired by this middle $a-i+1\mapsto a-i+2$ when we apply $f_{a-i+1}^*$. 

   \begin{claim} The $a-i$ unpaired by the middle $a-i+1$ is right (in the reading word of $x_{i-2}$) of the rightmost unpaired $\boxed{\blue{a-i+1}}$ in $y_{i-2}$.
   \end{claim}
   
   Suppose otherwise. Then the $a-i$ which gets unpaired is in the same row as the $\boxed{\blue{a-i+1}}$. It is thus in the same column as an $a-i+2$ (if it were above an $a-i+1$, it would be column-paired with that element instead). Since we have by induction that the black $a-i+1$ is unpaired in $x_{i-2}$ while the $a-i+2$ is still present, the $a-i+2$ is paired with an $a-i+1$ left of the middle $a-i+1$. This new $a-i+1$ (red, in the diagram below) is then also paired with an $a-i$ left of the $\boxed{\blue{a-i+1}}$. Inductively, we thus get the following infinitely repeating structure indicated in red:

   
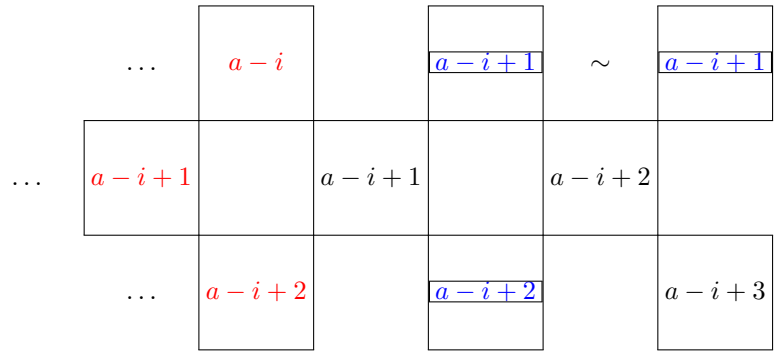
\begin{center}
   \begin{ytableau}
    \none &\none[\dots]&\red{a-i}&\none &\boxed{\blue{\hspace{0.2em}a-i+1}\hspace{0.2em}}&\none[\sim]&\boxed{\blue{\hspace{0.2em}a-i+1\hspace{0.2em}}}\\
    \none[\dots]&\red{a-i+1}&\none&a-i+1&\none&a-i+2\\
    \none&\none[\dots]&\red{a-i+2}&\none &\boxed{\blue{\hspace{0.2em}a-i+2}\hspace{0.2em}}&\none&a-i+3
\end{ytableau}
\end{center}

    This is a contradiction, and so the $a-i$ unpaired by the $a-i+1$ must be right of the blue $\boxed{\blue{a-i+1}}$. 

    Thus, since the $\boxed{\blue{a-i}}$ which gets unpaired by the $\boxed{\blue{a-i+1\mapsto a-i+2}}$ is right of the $a-i$ which gets unpaired by the middle $a-i+1\mapsto a-i+2$, we must have that the $\boxed{\blue{a-i+1}}$ is column paired with an $\Circled{\green{a-i}}$ (if it were not column paired, the middle $a-i+1\mapsto a-i+2$ would have unpaired the same $a-i$ which gets unpaired by the $\boxed{\blue{a-i+1\mapsto a-i+2}}$). Moreover, the middle $a-i+1$ unpairs an $a-i$ in the middle row. 
    
    If the $\boxed{\blue{a-i}}$ is column paired, then we have that the $\boxed{\blue{a-i}}=\Circled{\green{a-i}}$, and so we obtain one of the two cases specified by our inductive hypotheses:

    \begin{center}
     \ytableausetup{boxsize = 4.3em}
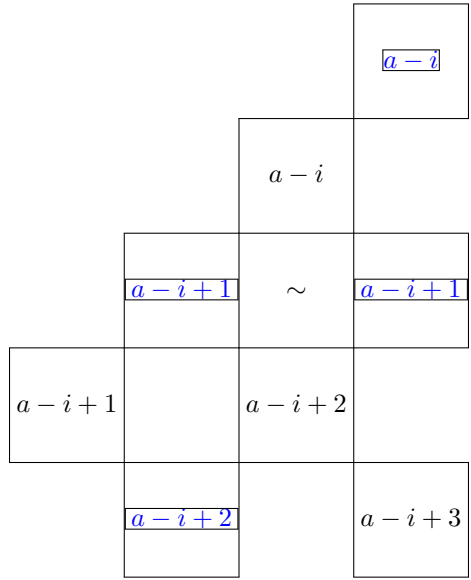
\begin{ytableau}
    \none&\none&\none&\boxed{\blue{a-i}}\\
    \none&\none&a-i\\
    \none &\boxed{\blue{\hspace{0.2em}a-i+1}\hspace{0.2em}}&\none[\sim]&\boxed{\blue{\hspace{0.2em}a-i+1}\hspace{0.2em}}\\
    a-i+1&\none&a-i+2\\
    \none &\boxed{\blue{\hspace{0.2em}a-i+2}\hspace{0.2em}}&\none&a-i+3
\end{ytableau}
\end{center}

    Now suppose that the $\boxed{\blue{a-i}}$ is not column paired. It is then somewhere right of the $\Circled{\green{a-i}}$. It is paired with an $\underline{\red{a-i+1}}$ in the middle row. This $\underline{\red{a-i+1}}$ is paired with an $a-i+2$. By induction, the $\boxed{\blue{a-i+1}}$ is the rightmost unpaired $a-i+1$ in $y_{i-2}$. Thus the $\underline{\red{a-i+1}}$ is either column paired or paired with an $a-i+2$ right of the $\boxed{\blue{a-i+1}}$. This brings us to the final difficult claim in this proof.

   \begin{claim}The $\underline{\red{a-i+1}}$ is not column paired.
   \end{claim}

    Suppose for the sake of contradiction that the $\underline{\red{a-i+1}}$ is column paired with a $\underline{\red{a-i+2}}$. This $\underline{\red{a-i+2}}$ is then paired. It is also either column paired or paired with an $\underline{\red{a-i+3}}$ right of the $\boxed{\blue{a-i+2}}$. Continue this process downwards. Recall that by Claim 3 from our base case, since the bottom $a-1$ cannot be in the same column as a $\underline{\red{a+1}}$, not every $\underline{\red{a-i+1}},\underline{\red{a-i+2}},\dots,\underline{\red{a+1}}$ can be column paired. We thus have the following substructure:

    \begin{center}
    \ytableausetup{boxsize = 4.3em}
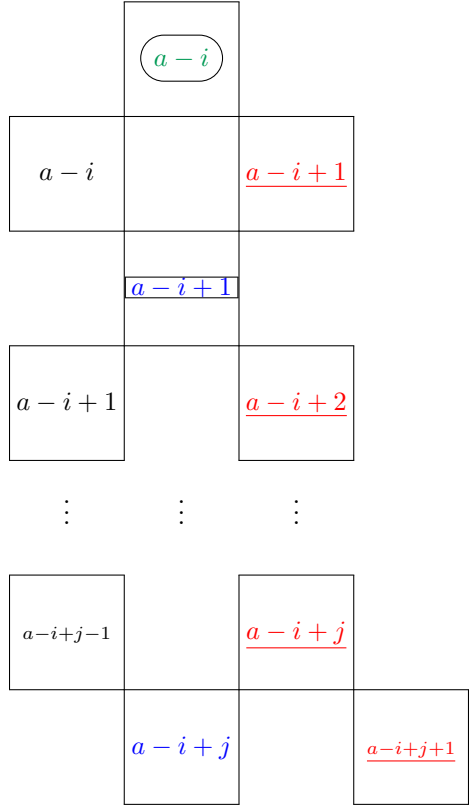
\begin{ytableau}
    \none&\Circled{\green{a-i}}&\none\\
    a-i&\none&\underline{\red{a-i+1}}\\
    \none&\boxed{\blue{\hspace{0.2em}a-i+1}\hspace{0.2em}}&\none\\
    a-i+1&\none&\underline{\red{a-i+2}}\\
    \none[\vdots]&\none[\vdots]&\none[\vdots]\\
    \scriptstyle{a-i+j-1}&\none&\underline{\red{a-i+j}}\\
    \none&\blue{\hspace{0.2em}a-i+j\hspace{0.2em}}&\none&\underline{\red{\scriptstyle{a-i+j+1}}}
\end{ytableau}
\end{center}

    Here, the red elements are all paired. Above the $\underline{\red{a-i+j+1}}$, we have an $a-i+j-1$. As this $a-i+j-1$ is paired in $x_{i-j}$, it is paired with an $a-i+j$ in the row below. Like the $\underline{\red{a-i+j}}$, this $a-i+j$ is either column paired or paired right of the $\boxed{\blue{a-i+j}}$. If it is column paired, we can just continue our process by replacing the $\underline{\red{a-i+j}}$ with this $a-i+j$. On the other hand, if it is not column paired, it is paired with an $a-i+j+1$ in the row below. We can then again just repeat our current construction and obtain an infinite sequence going right. This is hence another contradiction, which completes the proof of the claim.

    Thus the $\underline{\red{a-i+1}}$ is not column paired. Thus the $\underline{\red{a-i+2}}$ is in the row below. Above the $\underline{\red{a-i+2}}$, we must have an $a-i$. We can again without loss of generality assume that this $a-i$ is the $\boxed{\blue{a-i}}$ which gets unpaired in $y_{i-1}$. We thus obtain the second case in our inductive hypothesis:

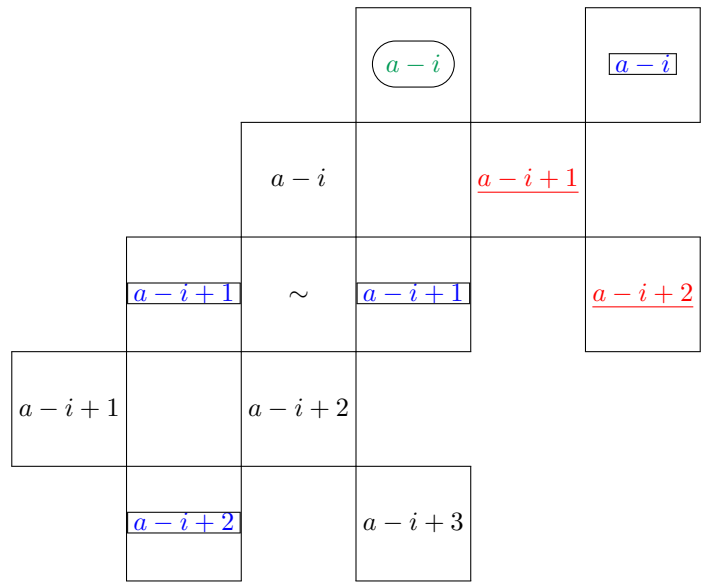
\begin{center}
    \ytableausetup{boxsize = 4.3em}
\begin{ytableau}
    \none&\none&\none&\Circled{\green{a-i}}&\none&\boxed{\blue{\hspace{0.2em}a-i\hspace{0.2em}}}\\
    \none&\none&a-i&\none&\underline{\red{a-i+1}}\\
    \none &\boxed{\blue{\hspace{0.2em}a-i+1}\hspace{0.2em}}&\none[\sim]&\boxed{\blue{\hspace{0.2em}a-i+1}\hspace{0.2em}}&\none&\underline{\red{a-i+2}}\\
    a-i+1&\none&a-i+2\\
    \none &\boxed{\blue{\hspace{0.2em}a-i+2\hspace{0.2em}}}&\none&a-i+3
\end{ytableau}
\end{center}

    This thus completes the proof of our inductive case.
    \end{proof}
    Combining the constructions in the base case and inductive case, we obtain that the rightmost unpaired $a-m$ in $y_{m-1}$ has an element under it in its same column. Since the $a-m$ is bad by assumption, it is at the end of its row. This thus contradicts $\lambda$ being strict, which finally completes the proof of the gluing property for shuffle tableaux
\end{proof}

\section{Some Applications}\label{sec:applications}

In this section, we will state some immediate corollaries to Proposition \ref{prop:applyingcharacterization}. These corollaries parallel important Schur positivity results, with the same twist as Theorem \ref{thm:TL}: they only hold if the relevant partitions are strict. 

\subsection{A flagged Littlewood-Richardson rule for products}

Reiner-Shimozono \cite{reiner1995key} showed that flagged \textit{skew} Schur polynomials are key-positive, and explicitly described a generalization of the Littlewood-Richardson rule. But for Schur polynomials, more can be said: products $s_{\lambda/\mu}s_{\nu/\rho}$ are also Schur positive, since products of straight Schur polynomials are Schur positive. 

It is thus natural to try to generalize Reiner-Shimozono's theorem to products of skew flagged Schur polynomials:
\[s^{\vec{b}}_{\lambda/\mu}s^{\vec{b}}_{\nu/\rho}\]

\begin{remark}
We believe that shuffle tableaux are the `right' perspective from which to consider products of flagged Schur polynomials, for the following general reasoning. For unflagged Schur polynomials, there is a very simple way to see why the Littlewood-Richardson rule expresses both product and skew expansions simultaneously. Namely, given two straight shapes, there is an easy way to put them together into a skew shape:

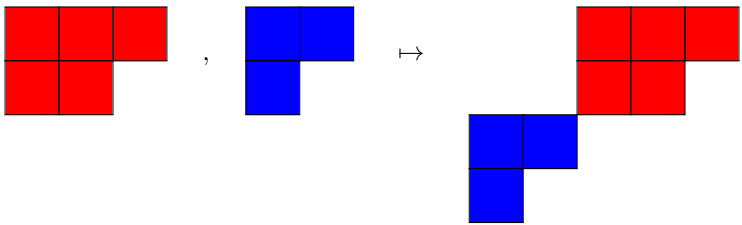
\begin{center}
    \ytableausetup{boxsize = 2em}
   \begin{ytableau}
      *(red)  & *(red) & *(red)\\
    *(red) & *(red)
    \end{ytableau} \hspace{1em}, \hspace{1em}\begin{ytableau}
      *(blue)  & *(blue)\\
    *(blue)
    \end{ytableau} \hspace{1em} $\mapsto$ \hspace{1em} \begin{ytableau}
      \none & \none & *(red)  & *(red) & *(red)\\
    \none & \none & *(red) & *(red)\\
    *(blue)  & *(blue)\\
    *(blue)
    \end{ytableau}
    \end{center}

    In this way, pairs of semistandard Young tableaux correspond to skew Young tableaux. There is a problem from the perspective of flagged tableaux, though: this way of gluing together two shapes into a skew shape only preserves the nondecreasing-ness of a flag $\vec{b}$ if that flag is constant. However, interlacing the two shapes with the same flag $\vec{b}$, as a shuffle tableau does, always preserves the nondecreasing-ness of the flag.  
    
\end{remark}

With this vague philosophizing out of the way, we have: 
\vspace{1em}

\begin{theorem}\label{thm:flaggedLR}
    Let $\lambda/\mu$ and $\nu/\rho$ be skew shapes satisfying the same two properties:
    \begin{enumerate}
        \item For all $i$, $\lambda_i \neq \nu_i$
        \item $\lambda$ and $\nu$ are \textit{interlacing}: that is, $\lambda_i \geq \nu_{i+1}$ and $\nu_i \geq \lambda_{i+1}$. 
    \end{enumerate}

    Then, $s_{\lambda/\mu}^{\vec{b}} s_{\nu/\rho}^{\vec{b}}$ is key positive. Furthermore, there is a relatively simple combinatorial algorithm for its key expansion. 
\end{theorem}

\begin{proof} 
    This follows immediately from Theorem \ref{thm:TL}. Suppose that $\lambda,\mu$ each have $n$ parts. Construct the flagged Jacobi-Trudi matrix $A_{\lambda'/\mu'}(\vec{b'})$ with 
    \begin{enumerate}
        \item flag given by $\vec{b'}=(b_1,b_1,b_2,b_2,\dots)$ (the `doubled version' of $\vec{b}$)
        \item partition $\lambda'$ of length $2n$ given by 
        \[\lambda'=(\max\{\lambda_1,\nu_1\}+n,\min\{\lambda_1,\nu_1\}+n,\dots,\max\{\lambda_n,\nu_n\}+1,\,\min\{\lambda_n,\nu_n\}+1).\]
        \item partition $\mu'$ of length $2n$ given by 
        \[\mu'=(\max\{\mu_1,\rho_1\}+n,\min\{\mu_1,\rho_1\}+n,\dots,\max\{\mu_n,\rho_n\}+1,\,\min\{\mu_n,\rho_n\}+1).\]
    \end{enumerate}
    Then $\lambda$ is strict by hypotheses (i) an (ii), and so by Theorem \ref{thm:TL}, the Temperley-Lieb immanants of $A_{\lambda',\mu'}(\vec{b'})$ are key positive. Take $I$ equal to the indices of rows which correspond to $\lambda$ and $J$ equal to the indices of rows which correspond to $\mu$. Denote $\overline{I} = [2n]\setminus I$ and $\overline{J} = [2n]\setminus J$.
    We then have, if $\Theta(I,J)$ is the set of TL basis elements which are compatible with $I,J$ (see \cite{nguyen2025temperley} for details),
    \[
    s_{\lambda/\mu}^{\vec{b}}s_{\nu/\rho}^{\vec{b}} = \Delta_{I,J}(A_{\lambda'/\mu'}(b'))\Delta_{\overline{I},\overline{J}}(A_{\lambda'/\mu'}(\vec{b'})) = \sum_{w\in \Theta(I,J)} \T{Imm}^{TL}_w(A_{\lambda'/\mu'}(\vec{b'})).
    \]
    Thus, by expanding each $\T{Imm}^{TL}_w(A_{\lambda'/\mu'}(b'))$ in the basis of key polynomials, we get a key positive expansion for $s_{\lambda/\mu}^{\vec{b}}s_{\nu/\rho}^{\vec{b}}$.
\end{proof}

Perhaps surprisingly, if the hypotheses on the shapes $\lambda/\mu, \nu/\rho$ are not true, the product $s_{\lambda/\mu}^{\vec{b}} s_{\nu/\rho}^{\vec{b}}$ is \textit{not} key positive in general. 
\vspace{1em}

\begin{remark}
    Theorem \ref{thm:flaggedLR} can be generalized slightly. The flags $\vec{b}, \vec{b}'$ on $s_{\lambda/\mu}$ and $s_{\nu/\rho}$ need not be identical, as long as they are \textit{interlacing} (that is, $b_i \leq b_{i+1}'$ and $b_i' \leq b_{i+1}$). 
\end{remark}
\vspace{1em}

\begin{remark}
    By expanding each TL immanant into key polynomials using shuffle tableaux, we have a combinatorial formula for the product $s_{\lambda/\mu}^{\vec{b}}s_{\nu/\rho}^{\vec{b}}$ in Theorem \ref{thm:flaggedLR}. If we assume slightly stronger hypotheses on our partitions though, we get a much simpler algorithm for the key-positive expansion of a product of flagged Schur polynomials. Namely, suppose that $\lambda/\mu,\nu/\rho$ satisfy the following: \begin{enumerate}
        \item For all $i$, $\lambda_i\geq \nu_i$ and $\mu_i\geq \rho_i$,
        \item $\lambda$ and $\nu$ are strict partitions.
    \end{enumerate}
    We describe this algorithm below.
\end{remark}

Assume the above two hypotheses. Given Theorem \ref{thm:flaggedLR} and Remark \ref{rmk:algorithm}, we describe a relatively straightforward algorithm for this key expansion. To express $s^b_{\lambda/\mu}s^b_{\nu/\rho}$ as a sum of key polynomials, perform the following steps:

\begin{enumerate}
    \item Make the shuffle diagram $(\lambda/\mu)\circledast(\nu/\rho)$. 
    \item Find all highest-weight shuffle tableaux $T$ of this shape. Each one will contribute a key polynomial to the sum $k_{\lambda, w}$, where $\lambda$ is given by the content of $T$. 
    \item For each highest weight $T$, find the corresponding $w$ following the algorithm outlined in Remark \ref{rmk:algorithm}. 
\end{enumerate}

The most computationally inefficient step of this algorithm is step $2$. In principle, one would have to run through all shuffle tableaux $T$, and for each one, check whether every $e_i$ gives us $e_i(T) = 0$. But, thanks to the bijection described in \cite{nguyen2025shuffletableauxlittlewoodrichardsoncoefficients}, this step is actually much easier than that (we must only find all `$D$-peelable (straight) tableaux.') 

\vspace{1em}

\begin{example}
    Let us expand the product $s_{\lambda}^bs_{\nu}^b$, with $\lambda = (3, 2)$, $\nu = (2, 1)$, and flag $\vec{b} = (1, 3)$, according to this rule. 

    The relevant shuffle diagram is 
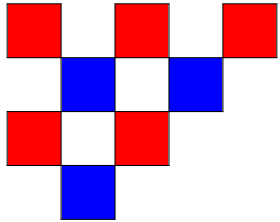
\begin{center}
\ytableausetup{boxsize = 2em}
   \begin{ytableau}
       *(red) & \none & *(red) & \none & *(red) \\
  \none  & *(blue) & \none & *(blue)\\
  *(red) & \none & *(red)\\
  \none & *(blue)
  
\end{ytableau}
\end{center}
The flagged shuffle tableaux we seek have entries at most $1$ in the first two rows and at most $3$ in the last two rows. Using \cite{nguyen2025shuffletableauxlittlewoodrichardsoncoefficients}, we can quickly compute the highest weight elements supported on this diagram:

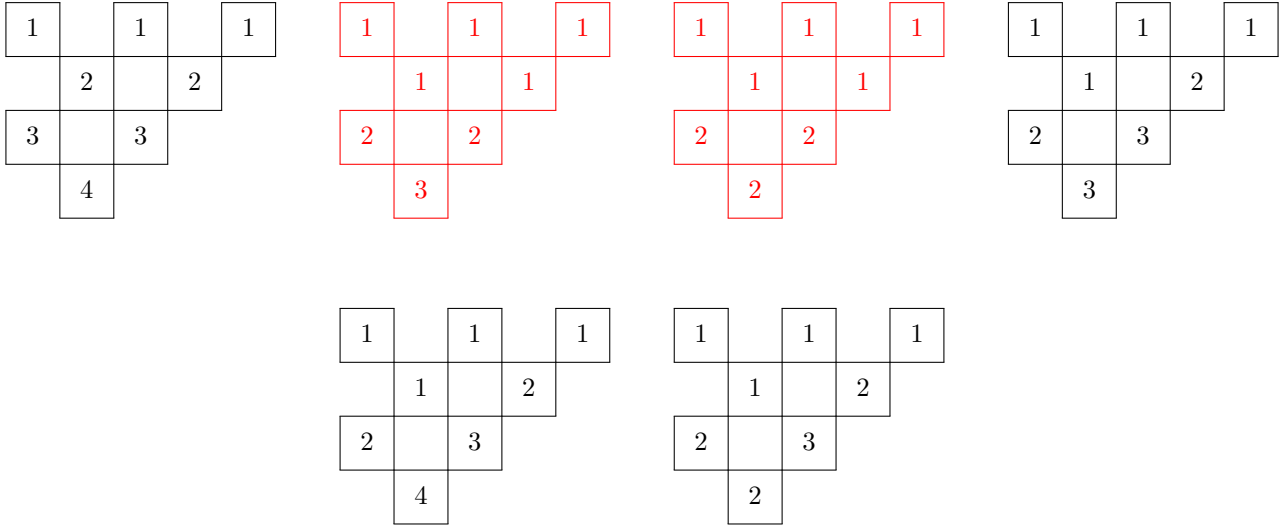
\begin{center}
   \begin{ytableau}
       $1$ & \none & $1$ & \none & $1$ \\
  \none  & $2$ & \none & $2$\\
  $3$ & \none & $3$\\
  \none & $4$
\end{ytableau} \hspace{0.7em}  \hspace{0.7em}
\textcolor{red}{\begin{ytableau}
       $1$ & \none & $1$ & \none & $1$ \\
  \none  & $1$ & \none & $1$\\
  $2$ & \none & $2$\\
  \none & $3$
\end{ytableau}  \hspace{0.7em}  \hspace{0.7em}
\begin{ytableau}
       $1$ & \none & $1$ & \none & $1$ \\
  \none  & $1$ & \none & $1$\\
  $2$ & \none & $2$\\
  \none & $2$
\end{ytableau}}  \hspace{0.7em}  \hspace{0.7em}
\begin{ytableau}
       $1$ & \none & $1$ & \none & $1$ \\
  \none  & $1$ & \none & $2$\\
  $2$ & \none & $3$\\
  \none & $3$
\end{ytableau} 

\vspace{8 mm}

\begin{ytableau}
       $1$ & \none & $1$ & \none & $1$ \\
  \none  & $1$ & \none & $2$\\
  $2$ & \none & $3$\\
  \none & $4$
\end{ytableau} 
\hspace{0.7em} \hspace{0.7em}
\begin{ytableau}
       $1$ & \none & $1$ & \none & $1$ \\
  \none  & $1$ & \none & $2$\\
  $2$ & \none & $3$\\
  \none & $2$
\end{ytableau}
\end{center}

Only the red highlighted shuffle tableaux are flagged by $(1,1,3,3)$, so there will be two key polynomials in the expansion. Furthermore, using Remark \ref{rmk:algorithm}, the permutation $s_2$ is associated to both tableaux. Therefore, we have 

\[s^{(1,3)}_{(3,2)}s^{(1,3)}_{(2,1)} = k_{(5,2,1), s_2} + k_{(5,3), s_2} = k_{(5,1,2)} + k_{(5,0,3)}\]
\end{example}

\subsection{A log concavity result}

We recall our log concavity inequality for flagged Schur polynomials, which was the original motivating question for this project:

\vspace{1em}

\begin{corollary}\label{cor:logconcave}
    Let $\lambda/\mu$ and $\nu/\rho$ be skew shapes satisfying the same two properties as in Theorem \ref{thm:flaggedLR}

    Then, the difference 

    \[s^b_{\lambda/\mu \vee \nu/\rho} s^b_{\lambda/\mu \wedge \nu/\rho} - s^b_{\lambda/\mu}s^b_{\nu/\rho}\]

    is key-positive. 

\end{corollary}

Here, the operations $\vee$ and $\wedge$ are defined on skew diagrams as follows. For partitions $\lambda, \mu$, we let $\lambda \vee \mu$ be the partition $(\max(\lambda_1, \mu_1), \max(\lambda_2, \mu_2) \dots)$. We define $\lambda \wedge \mu$ similarly, replacing maxima with minima. Then, for skew shapes, we write 
\[(\lambda/\mu) \wedge (\nu/\rho) =\lambda \wedge \nu/\mu \wedge \rho\]

\begin{example}
    Let $\lambda = (3,1)$, $\nu = (2,2)$, and let $\mu, \rho$ be empty. Let the flag $\vec{b}$ be $(2,4)$. Then, $\lambda \vee \mu = (3,2)$, $\lambda \wedge \mu = (2, 1)$. We can compute that 

    \[s_{(3,2)}^{(2,4)}s_{(2,1)}^{(2,4)} -s_{(3,1)}^{(2,4)}s_{(2,2)}^{(2,4)}= k_{3, 3, 1, 1} + k_{3, 4, 0, 1} + k_{4, 4}\]

    which is indeed key positive. 
\end{example}

The proof is almost immediate from Theorem \ref{thm:TL} using the same construction as in Theorem \ref{thm:flaggedLR} and following a similar strategy to the proof for the unflagged case of Lam-Postnikov-Pylyavskyy \cite{lam2007schur}. 

\section{Appendix}\label{sec:appendix}

Informally, Theorem \ref{thm:principalmoreusable}, as well as Proposition \ref{prop:idealdecomp}, hold whenever the objects of the underlying crystal $\mathcal{B}$ are `tableau-like,' and our subset $X \subseteq \mathcal{B}$ is defined by some kind of `row-flagging conditions.' Examples include flagged skew tableaux and flagged shuffle tableaux. To motivate the utility of the results in this appendix, we state an open problem that should now be solvable with our techniques. 

Furthermore, in light of the fact that every (finite, normal, type $A$) crystal is isomorphic to a crystal $\mathcal{B}_{\lambda}$ of semistandard Young tableaux, the following fact seems intuitive to us (although we have not yet found a reference for it) 
\vspace{1em}

\begin{conjecture}
    Every Demazure crystal $\mathcal{B}_{w, \lambda}$ (in the type $A$, normal, finite case) is isomorphic to a crystal of flagged skew tableaux for some shape $\lambda/\mu$ and nondecreasing flag $\vec{b}$. This would imply that all results in this appendix hold for every Demazure crystal. 
\end{conjecture}


\subsection{Open problem: flagged $P$-Schur polynomials}

It was a problem of Stanley to show that the \textit{$P$-Schur functions} are Schur positive. $P$-Schur functions are generating functions for \textit{shifted tableaux;} we do not define shifted tableaux here, and instead refer the reader to \cite{shiftedtableaux} for a thorough explanation. 
The Schur positivity of $P$-Schur functions follows from Sagan's insertion algorithm for shifted tableaux \cite{shiftedinsertion}, independently developed by Worley \cite{shiftedtableaux}. Another proof was recently given by Assaf and Oguz \cite{shiftedcrystals}, who showed that shifted tableaux organize themselves into crystals. Motivated by their result, we define \textit{flagged $P$-Schur functions} as generating functions for flagged shifted tableaux (where flagging a shifted tableau just means that each row is given an upper bound, as usual). 
\vspace{1em}

\begin{problem}
    Do flagged $P$-Schur functions expand positively in the key polynomial basis? Do flagged shifted tableaux assemble into Demazure crystals? 
\end{problem}

Based on computational examples, the answer seems to be positive. Furthermore, it is easy to see that flagged shifted tableaux form \textit{extremal} subsets. Therefore, it only remains to check the conditions of Theorem \ref{thm:principalmoreusable} and Proposition \ref{prop:idealdecomp} for flagged shifted tableaux. 

\subsection{Proof of Theorem \ref{thm:principalmoreusable}}
In this section of the appendix, we seek to prove that a union of Demazure crystals $X$ is a Demazure crystal if and only if $X$ satisfies the two local properties in Theorem \ref{thm:principalmoreusable}. Throughout this section, assume that $X$ be a connected component of a union of Demazure crystals. The symbol $b_{\lambda}$ will generally denote the unique highest weight element of $X$. 

By Theorem \ref{thm:characterization}, it suffices to show that $X$ has a unique lowest weight element (in each of its connected components). Using the lowest-weight-finding algorithm in Remark \ref{rmk:algorithm}, we already have a candidate for this lowest weight element. Our approach is then to show that for any other extremal $x\in X$, we can pass to a lower weight element using a combination of our two local properties in Theorem \ref{thm:principalmoreusable} and braid relation manipulations (as described in Lemma \ref{lemma:braidextend}). 

Recall the two properties from Theorem \ref{thm:principalmoreusable}: 

Let $n,a,b\in \N$.
\begin{enumerate}
    \item Let $x\in X$ be extremal such that $e_i(x)=0$ for all $i\leq n+1$. Assume $b\leq a<n$. Suppose that $F_{n,a}F_{n+1,b}(x) \in X$ and $F_{n+1,b-1}(x)\in X$. Then $F_{n,a}F_{n+1,b-1}(x)\in X$. 
    \item Let $k,m\in \N$ and $x\in X$ be extremal so that $e_i(x)=0$ for all $a-m\leq i\leq a+k$. Let $k,m,a$ be such that $f_{a+k}^*\dots f_{a}^*(x)\in X$ and $f_{a-m}^*\dots f_{a-1}^*(x)\in X$. Then either $f_{a+k}^*\dots f_{a-1}^*(x)\in X$ or $f_{a-m}^*\dots f_{a-1}^*f_a^*(x)\in X$. 
\end{enumerate}

We start with a few technical lemmas. To understand the purpose of these lemmas and see a warm up to our general approach of extending expressions, the reader may want to skip to example \ref{ex:2operators} (the case of two $F_{a,b}$ operators) before coming back to these proofs. 

We now prove that the gluing property as stated in Theorem \ref{thm:principalmoreusable} implies a clunkier but more usable result described in Lemma \ref{lemma:prop2moreusable}. This allows us to take a short operator $F_{n,b}$ (ie for $b$ close to $n$) and 'glue' it to a longer chain of operators. We refer to this equivalent result also as the gluing property.

To do this, we first need the following lemma.
\vspace{1em}

\begin{lemma}\label{lemma:e_i=0cor}
    Let $b>a+m$. Let $x\in X$ be extremal so that $e_i(x) = 0$ for all $a\leq i\leq k+m+1$.
    Then
    \[
    e_i(F_{k,a}\dots F_{k+m,a+m}F_{k+m+1,b}(x)) = 0
    \]
    for all $a\leq i\leq k+m+1$.
\end{lemma}
\begin{proof}
    We induct on $m$. First suppose that $m=0$. The expression we must consider is $F_{k,a}F_{k+1,b}(x)$. This immediately follows for $i<k$ by Lemma \ref{e_i=0}. Thus, consider $i = k+1$ and induct on $k$. First, suppose that $k=a$. For $b>a+1$, we have $e_{a+1}(F_{a,a}(x))=0$. For $b=a+1$ we have $e_{a+1}(F_{a,a}F_{a+1,a+1})=e_{a+1}(f_af_{a+1}(x))=0$ by Lemma \ref{e_i=0}. Now consider the inductive case for $k$. By induction, we have that $e_k(F_{k-1,a}F_{k,b}(x))=0$. We can thus apply Lemma \ref{e_i=0} to get that
    \[
    e_{k+1}(f_k^*f_{k+1}^*(F_{k-1,a}F_{k,b}(x)))=e_{k+1}(F_{k,a}F_{k+1,b}(x))=0.
    \]
    With the base case $m=1$ proved, consider the inductive case $F_{k,a}\dots F_{k+m,a+m}F_{k+m+1,b}(x)$. Suppose that $e_i(F_{k+1,a+1}\dots F_{k+m,a+m}F_{k+m+1,b}x)=0$ for all $k+2\leq i\leq k+m+1$. As before, $e_i(F_{k,a}\dots F_{k+m,a+m}F_{k+m+1,b}x)=0$ for all $i<k$ by Lemma \ref{e_i=0}. For $i>k+1$, we also get that $e_i(F_{k,a}\dots F_{k+m,a+m}F_{k+m+1,b}(x))=0$ by induction. It thus again suffices only to check $e_i(F_{k,a}\dots F_{k+m,a+m}F_{k+m+1,b}(x))$ for $i = k$. We use the same argument as above and induct on $k$ again. First consider the base case $k=a$. For $b>a+m+1$ we have
    \[
    e_{a+1}(F_{k,a}F_{k+1,a+1}\dots F_{k+m,a+m}(x))=0
    \]
    by Lemma \ref{e_i=0}. On the other hand, for $b=a+m+1$ we have
    \[
     e_{a+1}(F_{k,a}F_{k+1,a+1}\dots F_{k+m,a+m}F_{k+m+1,a+m+1}(x))=0,
    \]
    again by Lemma \ref{e_i=0}. Now, consider the inductive case. By induction we have that 
    \[e_k(F_{k-1,a}F_{k,a+1}\dots F_{k+m-1,a+m}F_{k+m,b})=e_k(F_{k-1,a}F_{k,a+1}F_{k+2,a+2}\dots F_{k+m,a+m}F_{k+m+1,b}(x))=0.\]
    Thus, by Lemma \ref{e_i=0},
    \[
    e_{k+1}(f_k^*f_{k+1}^*(F_{k-1,a}F_{k,a+1}F_{k+2,a+2}\dots F_{k+m,a+m}F_{k+m+1,b}))=e_{k+1}(F_{k,a}\dots F_{k+m,a+m}F_{k+m+1,b}(x))=0,
    \]
    which completes the proof.
\end{proof}

We can now prove our more usable extension property:
\vspace{1em}

\begin{lemma}\label{lemma:prop2moreusable}
    Let $X$ be such that the gluing property holds. Then $X$ satisfies the following property. Let $n,k,a,b\in \N$ be so that $a<b$ and $x\in X$ is extremal with $e_i(x)=0$ for all $i\geq a$. Suppose that: 
    \begin{enumerate}
        \item $F_{n,a}\dots F_{n+k,a+k}(x)\in X$
        \item $F_{n+k+1,b}(x)\in X$ and 
        \item $F_{n+k+1,a+k}(x)\not\in X$
    \end{enumerate} 
    Then, we can extend our expression to find: 
    \[
    F_{n,a}\dots F_{n+k,a+k}F_{n+k+1,b}(x)\in X.
    \]
\end{lemma}
\begin{proof}
    First, observe that by commuting some lowering operators rightwards using the first braid relation. we have
    \[
     F_{n,a} F_{n+1, a+1}\dots F_{n+k-1, a+k-1}F_{n+k,a+k}F_{n+k+1,b}(x)=
     \]
     \[F_{n,b-k-1}F_{n+1, b-k}\dots F_{n+k-1, b-2}F_{n+k,b-1}F_{n+k+1,b}(F_{b-k-2,a}F_{b-k-1, a+1}\dots F_{b-2,a+k}(x))
    \]
    To simplify notation, set $y_0:=F_{b-k-2,a}\dots F_{b-2,a+k}(x)$. We then seek to show that
    \[
    F_{n,b-k-1}\dots F_{n+k,b-1}F_{n+k+1,b}(y_0)\in X
    \]
    We first show, as a base case, that $F_{n,b-k-1}\dots F_{n+k,b-1}f_b^*(y_0)\in X$. Observe that since $F_{n+k+1,a+k}(x) \not\in X$, we get that 
    
    \begin{equation}\label{eq:y_0one}f_{n+k+1}^*\dots f_{b-1}^*(y_0)\not\in X\end{equation} (This follows again by braid relating the operators $f_{n+k+1}^* \dots f_{b-1}^*$ to the right, into the defining expression for $y_0$. We can do this manipulation because of the assumption that $a < b$).  
    
    On other other hand, since
    \[
    f_{n+k+1}^*\dots f_b^*(y_0) = F_{b-k-2,a}\dots F_{b-2,a+k}(f_{n+k+1}^*\dots f_b^*(x))\in X,
    \]
    we have \begin{equation}\label{eq:y_0two}f_{n+k+1}^*\dots f_b^*(y_0)\in X\end{equation} Moreover, \begin{equation}\label{eq:y_0three}f_{b-k-1}\dots f_{b-1}(y_0)\in X\end{equation} Finally, note that $e_i(y_0)=0$ for all $b-k-1\leq i\leq n+k+1$ by Lemma \ref{lemma:e_i=0cor}. 
    
    Combining Equations \ref{eq:y_0one}, \ref{eq:y_0two} and \ref{eq:y_0three}, by the gluing property we obtain
    \[
    f_{b-k-1}^*\dots f_{b-1}^*f_b^*(y_0)\in X.
    \]
    Now return to considering $F_{n,b-k-1}\dots F_{n+k,b-1}f_b^*(y_0)$. Braid relating $f_b^*$ leftwards (repeatedly using the first braid relation), we get
    \[
    F_{n,b-k-1}\dots F_{n+k,b-1}f_b^*(y_0)=f_{b-k-1}^*(F_{n,b-k-1}\dots F_{n+k,b-1}(y_0))
    \]
    Since $F_{n,b-k-1}\dots F_{n+k,b-1}y_0\in X$ by assumption, if $F_{n,b-k-1}\dots F_{n+k,b-1}f_by_0\not\in X$, $F_{n,b-k-1}\dots F_{n+k,b-1}f_by_0$ would have a bad entry $b-k-1\mapsto b-k-2$. This though is impossible as we have shown that $f_{b-k-1}^*\dots f_{b-1}^*f_b^*(y_0)\in X$. We thus obtain that
    \[
    f_{b-k-1}^*(F_{n,b-k-1}\dots w_{n+k,b-1}(y_0))=F_{n,b-k-1}\dots F_{n+k,b-1}f_b^*(y_0)\in X.
    \]

    This completes our base case. 
    
    We now induct to show that $F_{n,a}\dots F_{n+k,a+k}F_{n+k+1,b}(x)\in X$. Set 
    \[
    y_i:=(f_{b+i-k-2}^*\dots f_a^*)\dots (f_{b+i-2}^*\dots f_{a+k}^*)(f_{b+i-1}^*\dots f_{b}^*)(x)
    \]
    Inductively assume that for all $j<i$, 
    \[
    (f_n^*\dots f_{b+j-k-1}^*)\dots (f_{n+k}^*\dots f_{b+j-1}^*)f_{b+j}^*(y_j)\in X.
    \]
    We seek to show that 
    \[
    (f_n^*\dots f_{b+i-k-1}^*)\dots (f_{n+k}^*\dots f_{b+i-1}^*)f_{b+i}^*(y_i)\in X.
    \]
    As before, observe that $f_{n+k+1}^*\dots f_{b+i}^*f_{b+i-1}^*y_i\not\in X$ while $f_{n+k+1}^*\dots f_{b+i}^*(y_i)\in X$ by induction. 
    
    Moreover, $f_{b+i-k-1}^*\dots f_{b+i-1}^*(y_i)\in X$ by induction. Finally, note that $e_j(y_i)=0$ for all $b+i-k-1\leq j\leq n+k+1$ by Lemma \ref{lemma:e_i=0cor}.
    Thus again by the gluing property, 
    \[
    f_{b+i-k-1}^*\dots f_{b+i-1}^*f_{b+i}^*(y_i)\in X.
    \]
    Using the same braid relation argument as before, we obtain the desired result.
\end{proof}

The following technical lemma provides some additional valid and extremely useful extension operations on operators that follow entirely by braid relations.
\vspace{1em}
\begin{lemma}\label{lemma:braidextend}
    Let $\alpha = F_{1,a_1}\dots F_{k-1,a_{k-1}}$ be an arbitrary sequence of operators and let $x\in X$ be extremal. Suppose $\alpha F_{k,a}F_{k+1,b}(x)\in X$.
    \begin{enumerate}
        \item Suppose that $a<b$ and $F_{k+1,a}(x)\in X$. Then, $\alpha F_{k,b-1}F_{k+1,a}(x)\in X$.
        \item If $b\leq a$ then $\alpha F_{k,b}F_{k+1,a+1}(x)\in X$.
    \end{enumerate}
\end{lemma}
\begin{proof}
    We start with part (i). By Lemma \ref{lemma:symbolicmanipulation}, 
    \[
    \alpha F_{k,b-1}F_{k+1,a}x = f_{k+1}^*(\alpha F_{k,b-1}f_{k+1}^*\dots f_{b-2}^*f_b^*\dots f_a^*)(x)=f_{k+1}^*(\alpha F_{k,a}F_{k+1,b}(x)).
    \]
    Since $F_{k+1,a}(x)\in X$, $\alpha F_{k,b-1}F_{k+1,a}(x)$ will not have a bad entry $k+1\mapsto k+2$ taking us outside of $X$. Thus, since $\alpha F_{k,a}F_{k+1,b}(x)\in X$, $f_{k+1}^*(\alpha F_{k,a}F_{k+1,b}(x)) = \alpha F_{k,b-1}F_{k+1,a}(x)\in X$.

    We now do (ii). We again have, by Lemma \ref{lemma:symbolicmanipulation},
    \[
    \alpha F_{k,a}F_{k+1,b}(x) = f_{k+1}^*(\alpha F_{k,b}F_{k+1,a+1}(x))\implies \alpha F_{k,b}F_{k+1,a+1}(x)\in X.
    \]
\end{proof}

We now move on to characterizing when we can extend a sequence of operators. To motivate the following definitions, we start by considering the special case of a sequence of two operators.

\vspace{1em}

\begin{example}\label{ex:2operators}
   Let $y=F_{n,a}F_{n+1,b}(b_\lambda)$. Let $c$ be an integer with $c<b$ and $F_{n+1,c}(b_\lambda) \in X$. The algorithm described in Remark \ref{rmk:algorithm} is \textit{greedy} in the sense that it 'prioritizes' the length of $F_{n+1,a_{n+1}}$ before the operator $F_{n,a_n}$. We thus seek to find a lower weight element $z\in X$. To this end, $z = F_{n,a_n}F_{n+1,a_{n+1}}(b_\lambda)$ with $a_{n+1}=c$ is a good candidate; it just may or may not actually be in $X$. 

   Using our two local properties, this problem naturally splits into three cases.
   \begin{enumerate}
       \item The simplest case: first, suppose that $b\leq a$. Then we can use the extension property to show that $F_{n,a}F_{n+1,c}(b_\lambda)\in X$, and so we can directly take $z=F_{n,a}F_{n+1,c}(b_\lambda)$.
       \item Now, suppose that $a<b$ and $F_{n+1,a}(b_\lambda)\not\in X$. Then we can apply our gluing property in the form of Lemma \ref{lemma:prop2moreusable} to show that $F_{n,a}F_{n+1,c})b_\lambda)\in X$. We can thus again simply take $z=F_{n,a}F_{n+1,c}b_\lambda$.
       \item Finally, suppose that $a<b$ and $F_{n+1,a}b_\lambda\in X$. Then by Lemma \ref{lemma:braidextend}, observe that $F_{n,b-1}F_{n+1,a}(b_\lambda) \in X$. As $F_{n,a}F_{n+1,b}$ is a subword of the expression $F_{n,b-1}F_{n+1,a}$, $F_{n,b-1}F_{n+1,a}(b_\lambda)$ is lower weight than $y$ in the dominance order. We can thus pass to $z=F_{n,b-1}F_{n+1,a}(b_\lambda)$. We can then again apply our extension property, if necessary, to pass to $z:=F_{n,b-1}F_{n+1,c}(b_\lambda)$. Note that if $b=a+1$, we still have $z=F_{n,a}F_{n+1,c}(b_\lambda) \in X$.
   \end{enumerate}
    In all three of these cases, we are able to find a lower weight element in $X$, but the first two cases involve a simpler `extension:' we were able to naively extend $y$ to $z=F_{n,a}F_{n+1,c}(b_\lambda)$. In these two cases, we would call $y$ \textit{extendable} according to definition \ref{def:extendable}. In the third case though, if $b>a+1$, we had to make the $F_{n,a_n}$ operator shorter in order to pass to our lower weight element $z=F_{n,b-1}F_{n+1,c}(b_\lambda)$. Thus in this case, $y$ is an example of an element that is \textit{not} extendable. 
\end{example}

With this motivation in mind, we formalize our notion of extendible elements below.
\vspace{1em}

\begin{definition}\label{def:extendable}
    Let $y = F_{1,a_1}\dots F_{n,a_n}(x)$ so that $e_i(x)=0$ for al $i\leq n$. Let $1 \leq i \leq j \leq n$. Define $c_{i,j}$ as the maximal integer so that there exists a chain of operators $F_{m,k_m}$ as a subword of $F_{1,a_1}\dots F_{n,a_n}$ such that 
    \[i=k_0<k_1<k_2<\dots < k_{c_{i,j}}\leq j\]
    with $a_{k_m}\leq a_i+m$ for all $m$. Call $y$ \textit{extendable} from $x$ if for all such $i,j$, either $a_i+c_{i,j}\geq a_j$ or, if $a_i+c_{i,j}<a_j$, $F_{j,a_i+c_{i,j}}F_{j+1,a_{j+1}}\dots F_{n,a_n}x\not\in X$.
\end{definition}
\vspace{1em}

\begin{remark}
    The definition for extendable may look a bit convoluted. Its purpose is to pick out helpful elements $y \in X$ that can be `extended' further down in the simplest way possible using our two local properties, in order to find lower weight elements of $X$. The following two propositions help illustrate this idea. 

    As already noted, the key example which motivates our definition of extendable is the case of just two operators which we explored in Example \ref{ex:2operators}. In the case of two operators, the $c_{i,j}$ values arose in the third case to differentiate between when $b=a+1$ (extendable: $c_{n,n+1}=1$) and when $b>a+1$ (not extendable: $c_{n,n+1}=0$).
    
    For longer expressions, these values also have a fairly concrete origin. For $b\geq a$, by applying the first braid relation repeatedly and then the second, we get that
    \[
    f_b^*F_{n,a}(x) = f_n^*\dots f_b^*f_{b+1}^*f_b^*f_{b+1}^*\dots f_a^*(x)=F_{n,a}f_{b+1}^*(x).
    \]
    Thus, if we commute an operator $f_b^*$ through a chain of `long' operators (in the hopes of extending some operator to the right), it will increase by $1$ for each operator. The $c_{i,j}$ values just record how many times we increase $f_{b_i}^*$ if we were to 'move' $f_{b_i}^*$ right to $F_{j,b_j}$ via braid relating. To further understand the necessity of the $c_{i,j}$ values, see the proof of Lemma \ref{lemma:extendableexistance}. We also give another example below. 
\end{remark}
\vspace{1em}

\begin{example}
    Consider $y = F_{3,1}F_{4,2}F_{5,5}(b_\lambda)\in X$. Suppose first that $F_{5,2}(b_\lambda)\not\in X$. Then $y$ is extendable. To check this, note that $c_{3,4}=1$, $c_{3,5}=1$ ($F_{3,1}F_{4,2}F_{5,5}$ does not give us $c_{3,5} = 2$, because we have $5>1+c_{3,4}=2$). Finally, $c_{4,5}=0$ for similar reasons.

    Checking the definition of extendibility, we have:
    \begin{enumerate}
        \item For $i=3, j=4$: $a_i + c_{i,j} = 1+1 \geq 2 = a_j$
        \item For $i = 3$, $j = 5$: $a_i+c_{i,j} = 1+1 < a_j$, but $F_{5,2}(b_{\lambda}) \notin X$. 
        \item For $i = 4$, $j = 5$, $a_i+c_{i,j} = 2+0 < a_j$, but again, $F_{5,2}(b_{\lambda}) \notin X$ satisfies extendibility
    \end{enumerate}

    So, we have confirmed that $y$ is extendable. This is motivated by the fact that if, for example, $F_{5,3}(b_\lambda)\in X$, we can apply our gluing property to \textit{extend} $x$, finding a lower weight element in $X$:
    \[
    F_{3,1}F_{4,2}F_{5,3}(b_\lambda)\in X.
    \]
    On the other hand, if $F_{5,2}(b_\lambda)\in X$, $y$ is \textit{not} extendable, as the computations above show that $a_i + c_{i,j} < a_j$ for $i = 3, j = 5$ or $i = 4, j = 5$. This makes sense because we would not necessarily have that $F_{3,1}F_{4,2}F_{5,2}(b_\lambda)\in X$ following from the gluing property. Instead though, we can commute operators from the $F_{4,2}$ rightwards using Lemma \ref{lemma:braidextend} (i) to get that
    \[
    F_{3,1}F_{4,4}F_{5,2}(b_\lambda)\in X
    \]
    This gives us a `less naive' lower weight element of $X$. 
    
\end{example}

    The following proposition justifies the use of the term `extendible.'
    \vspace{1em}

\begin{proposition}
    Let $y = F_{1,a_1}\dots F_{n,a_n}(x)\in X$ be extendable from $x$ and $x$ extremal with $e_i(x)=0$ for all $i\leq n$. Suppose that $F_{n,b}(x)\in X$ where $b<a_n$. Then $y':=F_{1,a_1}\dots F_{n-2, a_{n-2}}F_{n-1,a_{n-1}}F_{n,b}x\in X$.
\end{proposition}

\begin{proof}
    We inductively show that $F_{k,a_k}\dots F_{n-1,a_{n-1}}F_{n,b}(x)\in X$. The base case follows immediately as $F_{n,b}(x)\in X$ by assumption. Now, suppose that $F_{k+1,a_{k+1}}\dots F_{n-1,a_{n-1}}F_{n,b}(x)\in X$; we seek to show that $F_{k,a_k}\dots F_{n-1,a_{n-1}}F_{n,b}(x)\in X$. 
    
    For ease of notation, set $a:=a_k$. We introduce the following auxiliary element $z = F_{k,b_k}F_{k+1,b_{k+1}}\dots F_{n-1,b_{n-1}}F_{n,b_n}(x)$, where we define $b_i$ as follows. \begin{enumerate}
        \item If $a_i\leq a + c_{k,i}$, set $b_i:=a+c_{k,i}$.
        \item If $a_i>a+c_{k,i}$, set $b_i:=0$.
    \end{enumerate}
     (Here, $c_{k,i}$ is the integer we defined in Definition \ref{def:extendable}). We refer to this process as \textit{reducing} $y$ and provide several examples of constructing $z$ after the proof in Example \ref{example:proof3.2}. Note that $z\not = 0$. Observe also that the defining expression of $z$ is a subexpression for the defining expression of $y$ (we are always setting $b_i$ to be at most $a_i$, by the extendability of $y$). 
     Then, since $\T{wt}(y)\leq \T{wt}(z)$, by the ideal property from Theorem \ref{thm:characterization}, $z\in X$.
     
     From here, we will use our two local properties to reconstruct $y'$ from $z$ and show that $y'\in X$. More formally, we prove the following claim.

     \begin{claim}
         Let $y = F_{k,a_k}\dots F_{n,a_n}(x)$ be extendable and let $z$ be the reduced element determined by the previous process from $y$. Suppose that $y' = F_{k+1,a_{k+1}}\dots F_{n-1,a_{n-1}}F_{n,c}(x)\in X$ and $z\in X$. Then $F_{k,a_k}(y')\in X$. 
     \end{claim}
    
    Set $y_i':=F_{i,a_i}\dots F_{n,b}(x)$. We now induct on $i$ to show that $F_{k,b_k}\dots F_{i-1,b_{i-1}}(y_i')\in X$. Suppose that we have $F_{k,b_k}\dots F_{i,b_{i}}(y_{i+1}')\in X$. We show that $F_{k,b_k}\dots F_{i-1,b_{i-1}}(y_{i}')\in X$. 

    Define $c:=a_i$ for $i<n$ and $c:=b$ for $i = n$. With this notation, we then have
    \[
    F_{k,b_k}\dots F_{i-1,b_{i-1}}(y_{i}') = F_{k,b_k}\dots F_{i-1,b_{i-1}}F_{i,c}(y_{i+1}').
    \]
    We work through a few cases on $b_i$.\\
    \textbf{Case 1: $b_i \not = 0$.}\\
    Suppose that $b_i = a+m$ for some $m\in \N$. By induction we have
    \[
    F_{k,b_k}\dots F_{i,b_{i}}(y_{i+1}')\in X\implies F_{k,a}F_{k+1,a+1}\dots F_{k+m,a+m}(y_{i+1}')\in X
    \]
    First suppose that $c\geq a$. Let $c = a+j$ for $j<m$. We then can use Lemma \ref{lemma:braidextend} to commute $F_{k+j,a+j} = F_{k+j,c}$ leftwards and get:
    \[
    F_{k,a}\dots F_{k+j,a+j}F_{k+j+1,a+j+1}\dots F_{k+m,a+m}(y_{i+1}')\in X\implies F_{k,a}\dots F_{k+j,a+j}F_{k+j+1,c}\dots F_{k+m,a+m}(y_{i+1}')\in X.
    \]
    Repeating this argument, we get that
    \[
    F_{k,a}\dots F_{k+j,a+j}F_{k+j+1,a+j+1}\dots F_{k+m,c}(y_{i+1}')\in X.
    \]
    Now let $c\leq a$. Note that,
    \[
    F_{k+1,a_{k+1}} \dots F_{i,a_{i}}(y_{i+1}')\in X\implies F_{k+1,a+1}F_{k+2,a+2} \dots F_{k+m,c}(y_{i+1}')\in X
    \]
    by the ideal property.
    Then, using Lemma \ref{lemma:braidextend} (ii) repeatedly, we obtain that
    \[
    F_{k+1,a+1}F_{k+2,a+2} \dots F_{k+m,c}(y_{i+1}')\in X\implies F_{k+1,c}F_{k+2,a+2} \dots F_{k+m,a+m}(y_{i+1}')\in X
    \]
    Thus by extension property,
    \[
    F_{k,a}F_{k+1,c}F_{k+2,a+2} \dots F_{k+m,a+m}(y_{i+1}')\in X.
    \]
    Commuting the operators in $F_{k+1,c}$ right again using Lemma \ref{lemma:braidextend}, we get that
    \[
    F_{k,a}F_{k+1,a+1}\dots F_{k+m,c}(y_{i+1}')\in X.
    \]
    Since $F_{k+1,a_{k+1}} \dots F_{i,a_{i}}(y_{i+1}')\in X$, we then have by Lemma \ref{lemma:braidextend} that
    \[
     F_{k,b_k}\dots F_{i-1,b_{i-1}}F_{i,c}(y_{i+1}')\in X,
    \]
    as hoped.\\
    \textbf{Case 2: $b_i = 0$.}\\
    Let $k+m+1$ be minimal so that $b_{k+m+1}=0$. We show that
    \[
    F_{k,a}F_{k+1,a+1}\dots F_{k+m,a+m}F_{k+m+1,c-d}(F_{k+m+2,b_m}\dots F_{i-2,b_{i-1}})(y_{i+1}')\in X,
    \]
    where $d = c_{a+m,i}$ is the number of operators $F_{j,b_j}$ with $b_j\not = 0$ between $F_{i,b_i}$ and $F_{k+m,a+m}$. 
    We seek to use our gluing property. To apply it though, we first need two claims.
    \begin{claim}
        $F_{k+m+1,c-d}(F_{k+m+2,b_{k+m+2}}\dots F_{i-1,b_{i-1}})(y_{i+1}')\in X$.
    \end{claim}
    \begin{proof}
        Note that
        \[
        F_{k+m+2,b_{k+m+2}}\dots F_{i-1,b_{i-1}}F_{i,c}(y_{i+1}') \in X.
        \]
        We then get that 
        \[
        F_{k+m+1,b_{k+m+2}}\dots F_{i-2,b_{i-1}}F_{i-1,c}(y_{i+1}') \in X.
        \]
        We can then shift operators rightwards using Lemma \ref{lemma:braidextend} to obtain
        \[
        F_{k+m+1,c-d}(F_{k+m+2,b_{k+m+2}}\dots F_{i-1,b_{i-1}})(y_{i+1}')\in X,
        \]
        as claimed.
    \end{proof}
    
    \begin{claim}\label{claim:notinX_induction}
        $F_{k+m+1,a+m}(F_{k+m+2,b_{k+m+2}}\dots F_{i-1,b_{i-1}})(y_{i+1}')\not\in X$.
    \end{claim}
    \begin{proof}
        Suppose for the sake of contradiction that $F_{k+m+1,a+m}(F_{k+m+2,b_{k+m+2}}\dots F_{i-1,b_{i-1}})(y_{i+1}')\in X$. As this is a strictly shorter expression we can then inductively reconstruct this to show that
        \[
        F_{k+m+1,a+m}(y_{a+m+1}')\in X.
        \]
        This, however, contradicts the extendability of $y$. 
    \end{proof}

    Thus, by our gluing property, we obtain 
    \[
    F_{k,a}F_{k+1,a+1}\dots F_{k+m,a+m}F_{k+m+1,c-d}(F_{k+m+2,b_m}\dots F_{i-2,b_{i-1}})(y_{i+1}')\in X.
    \]
    Finally, we can undo our first claim and commute $F_{k+m+1,c-d}$ back rightwards. This completes our induction, and so $y'\in X$. This inductive step also completes our proof.
\end{proof}

\begin{example}\label{example:proof3.2}
As the above proof is a bit complex, we provide several small examples to illustrate the process by which we extended $y$. In each case, we trace through the inductive argument. The reader might find these individual examples more illuminating than the proof itself. 
    \begin{itemize}
        \item Let $y = F_{3,1}F_{4,3}F_{5,3}(b_\lambda)$ and suppose we have that $F_{5,2}(b_\lambda)\in X$ while $F_{3,1},F_{4,1}F_{5,3}(b_\lambda)\not\in X$ (so that $y$ is extendable). We seek to show that $y'=F_{3,1}F_{4,3}F_{5,2}(b_\lambda)\in X$. In this case, the auxiliary element is given simply by $z = F_{3,1}(b_\lambda)$. Reconstructing $y'$ then proceeds in the following steps:
        \begin{enumerate}
            \item $F_{3,1}F_{5,2}(b_\lambda)\in X$. It suffices to show that $F_{3,1}F_{4,2}(b_\lambda)\in X$. If $F_{4,1}(b_\lambda)\in X$ then we would have that $F_{4,1}F_{5,2}\in X$ (by our gluing property), which is a contradiction. Thus by our gluing property, $F_{3,1}F_{5,3}(b_\lambda)\in X$.
            \item We now show that $F_{3,1}F_{4,3}F_{5,2}(b_\lambda)$. This follows immediately by our gluing property since $F_{4,3}F_{5,2}(b_\lambda)\in X$ while $F_{4,1}F_{5,2}(b_\lambda)\not\in X$. 
        \end{enumerate}
        \item Let $y= F_{3,2}F_{4,2}F_{5,1}F_{6,6}(b_\lambda)$. Suppose that $F_{6,5}(b_\lambda)\in X$ while $F_{6,4}(b_\lambda)\not \in X$. Then $z = F_{3,2}F_{4,3}F_{5,4}(b_\lambda)$. Then
        \begin{enumerate}
            \item First note that $F_{3,2}F_{4,3}F_{5,4}F_{6,5}(b_\lambda)\in X$ follows immediately by our gluing property.
            \item Now note that $F_{4,1}F_{5,4}F_{6,5}(b_\lambda)\in X$. Thus by our extension property, $F_{3,2}F_{4,1}F_{5,4}F_{6,5}(b_\lambda)\in X$ and so $F_{3,2}F_{4,3}F_{5,1}F_{6,5}(b_\lambda)$.
            \item Again by our gluing property, $F_{3,2}F_{4,2}F_{5,1}F_{6,5}(b_\lambda)\in X$.
        \end{enumerate}
        \item We end with a longer expression. Let $y = F_{3,1}F_{4,1}F_{5,5}F_{6,1}F_{7,7}F_{8,8}F_{9,5}(b_\lambda)$ be extendable. Suppose that $F_{9,4}(b_\lambda)\in X$ while $F_{9,3}(b_\lambda)\not\in X$. We have that $z = F_{3,1}F_{4,2}F_{6,3}(b_\lambda)$.
        \begin{enumerate}
            \item We first show that $F_{3,1}F_{4,2}F_{6,3}F_{9,4}(b_\lambda)\in X$. To do this, we show first that $F_{3,1}F_{4,2}F_{5,3}F_{6,3}(b_\lambda)\in X$. Suppose for the sake of contradiction that $F_{5,2}F_{6,3}(b_\lambda)\in X$. We show that this would imply that $F_{5,2}F_{6,1}F_{7,7}F_{8,8}F_{9,4}(b_\lambda)\in X$. We start by showing that $F_{5,2}F_{6,3}F_{9,4}(b_\lambda)\in X$. To do this, we show that $F_{5,2}F_{6,3}F_{7,4}(b_\lambda)\in X$. Suppose first for the sake of contradiction that $F_{7,3}(b_\lambda)in X$. Suppose again for the sake of contradiction that $F_{8,3}(b_\lambda)\in X$. Note that by assumption, $F_{9,3}(b_\lambda)\not\in X$. Thus by our gluing property, $F_{8,3}F_{9,4}(b_\lambda)\in X$. This though is a contradiction, as $y$ is extendable. 

            Thus $F_{8,3}(b_\lambda)\not\in X$, and so $F_{7,3}F_{8,4}(b_\lambda)\in X$. As $F_{9,4}(b_\lambda)\in X$, this implies that $F_{7,3}F_{9,4}(b_\lambda)\in X$. By assumption, $F_{8,3}F_{9,4}(b_\lambda)\not\in X$ as $y$ is extendable. Thus $F_{7,3}F_{8,8}F_{9,4}(b_\lambda)\in X$. This is again a contradiction to $y$ being extendable. Thus $F_{7,3}(b_\lambda)\not\in X$. We thus obtain that $F_{5,2}F_{6,3}F_{7,4}(b_\lambda)\in X$. This in turn implies that $F_{5,2}F_{6,3}F_{9,4}(b_\lambda)\in X$. 

            By commutativity, we now immediately get that $F_{5,2}F_{6,3}F_{8,8}F_{9,4}(b_\lambda)\in X$. Now by the gluing property, since by extendability $F_{7,3}F_{8,8}F_{9,4}(b_\lambda)\not\in X$, $F_{5,2}F_{6,3}F_{7,7}F_{8,8}F_{9,4}(b_\lambda)\in X$. We now finally show that $F_{5,2}F_{6,1}F_{7,7}F_{8,8}F_{9,4}(b_\lambda)\in X$. By Lemma \ref{lemma:braidextend}, $F_{5,2}F_{6,2}F_{7,7}F_{8,8}F_{9,4}(b_\lambda)\in X$. Then by our extension property, $F_{5,2}F_{6,1}F_{7,7}F_{8,8}F_{9,4}(b_\lambda)\in X$. This though is a final contradiction to the extendability of $y$.

            This rather complicated process was an example of the inductive reconstruction step of Claim \ref{claim:notinX_induction}. 

            Thus $F_{5,2}F_{6,3}(b_\lambda)\not\in X$. By the gluing property, we thus have $F_{3,1}F_{4,2}F_{5,3}F_{6,3}(b_\lambda)\in X$ as hoped. We then finally have
            \[
            F_{3,1}F_{4,2}F_{5,3}F_{6,3}(b_\lambda) = F_{3,1}F_{4,2}F_{6,3}F_{6,4}(b_\lambda)\in X\implies F_{3,1}F_{4,2}F_{6,3}F_{9,4}(b_\lambda)\in X.
            \]
            \item We now have $F_{3,1}F_{4,2}F_{6,3}F_{9,4}(b_\lambda)\in X$. By commutativity, we then get that $F_{3,1}F_{4,2}F_{6,3}F_{8,8}F_{9,4}(b_\lambda)\in X$.
            \item We again by commutativity get that $F_{3,1}F_{4,2}F_{6,3}F_{7,7}F_{8,8}F_{9,4}(b_\lambda)\in X$.
            \item We now show that $F_{3,1}F_{4,2}F_{6,1}F_{7,7}F_{8,8}F_{9,4}(b_\lambda)\in X$. Note that \[F_{3,1}F_{4,2}F_{6,3}F_{7,7}F_{8,8}F_{9,4}(b_\lambda)\in X\implies F_{3,1}F_{4,2}F_{5,3}F_{7,7}F_{8,8}F_{9,4}(b_\lambda)\in X.\]
            Then using Lemma \ref{lemma:braidextend}, we have
            \[
            F_{3,1}F_{4,2}F_{5,3}F_{7,7}F_{8,8}F_{9,4}(b_\lambda)\in X\implies F_{3,1}F_{4,1}F_{5,3}F_{7,7}F_{8,8}F_{9,4}(b_\lambda)\in X\implies F_{3,1}F_{4,2}F_{5,1}F_{7,7}F_{8,8}F_{9,4}(b_\lambda)\in X.
            \]
            Then by commutativity, $F_{3,1}F_{4,2}F_{6,1}F_{7,7}F_{8,8}F_{9,4}(b_\lambda)\in X$.
            \item The next step is to show that $F_{3,1}F_{4,2}F_{5,5}F_{6,1}F_{7,7}F_{8,8}F_{9,4}(b_\lambda)\in X$. Have immediately by extendability that $F_{5,2}F_{6,1}F_{7,7}F_{8,8}F_{9,4}(b_\lambda)\not\in X$. Thus by the gluing property, 
            \[
            F_{3,1}F_{4,2}F_{5,5}F_{6,1}F_{7,7}F_{8,8}F_{9,4}(b_\lambda)\in X.
            \]
            \item Finally, by Lemma \ref{lemma:braidextend}, $F_{3,1}F_{4,1}F_{5,5}F_{6,1}F_{7,7}F_{8,8}F_{9,4}(b_\lambda)\in X$. This finally completes our last example.
        \end{enumerate}
    \end{itemize}
\end{example}

The following lemma shows that given any extremal element in $X$, we can pass to a lower weight extendable element.
\vspace{1em}

\begin{lemma}\label{lemma:extendableexistance}
    Let $z\in X$ be an extremal lowest weight element (that is, if $z'\in X$ satisfies $\T{wt}(z')\leq \T{wt}(z)$ in the dominance order, then $\T{wt}(z')=\T{wt}(z)$). Then, $z$ is extendable.
\end{lemma}
\begin{proof}
Let $z$ be an extremal lowest weight element and suppose for the sake of contradiction that $z$ is not extendable. Let $k$ be maximal so that $F_{k,a_k}\dots F_{n,a_n}(b_\lambda)$ is not extendable. Let $m$ be minimal so that $a_k+c_{k,m}<a_m$ and $F_{m,a_k+c_{k,m}}(F_{m-1,a_{m-1}}\dots F_{n,a_n}(b_\lambda))\in X$. 

For ease of notation, set $c:=c_{k,m}$, $a_m:=b$, and $a_k:=a$.

Also define $x:=F_{m+1,a_{m+1}}\dots F_{n,a_n}(b_\lambda)$ and $\alpha:=F_{1,a_1}\dots F_{k-1,a_{k-1}}$. Then $F_{k+1,a_{k+1}}\dots F_{m,a_m}(x)$ is extendable, and so $F_{k+1,a_{k+1}}\dots F_{m,a+c}(x)\in X$. Moreover, since $z$ is lowest weight by assumption, 
\[
\alpha F_{k,b-c-1}F_{k+1,a_{k+1}}\dots F_{m-1,a_{m-1}}F_{m,a+c}(x)\not\in X
\]
We show that this leads to a contradiction. Let $\beta:=F_{k+1,a_{k+1}}\dots F_{m-1,a_{m-1}}$ 

Since $F_{k+1,a_{k+1}}\dots F_{m,a+c}(x)\in X$, the bad element in $\alpha F_{k,b-c-1}\beta F_{m,a+c}(x)$ is less than $k+1\mapsto k+2$. 

We now define a bit of convenient notation to make the following symbolic manipulations more clear. Denote the operator $F_{1,a_1}\dots F_{k,a_k}(x)$ by the length $k$ tuple $(a_1,a_2,\dots,a_k)$ (here $k$ and $a_k$ are arbitrary). Abusing notation a little, we then have that $\alpha F_{k,b-c-1}F_{k+1,a_{k+1}}\dots F_{m-1,a_{m-1}}F_{m,a+c}x$ corresponds to the tuple $(\alpha, b-c-1,a_{k+1},\dots,a_{m-1},a+c)$ (here $\alpha$ corresponds to a length $k-1$ expression). If we let $\beta:=F_{k+1,a_{k+1}}\dots F_{m-1,a_{m-1}}$, we then get the even simpler expression $(\alpha,b-c-1,\beta,a+c)$. 

Using the above observations and notation, we then get that we can do the following two manipulations on $(a,a_{k+1},\dots,a_{m-1},b)$ and still get an element of $X$.
\begin{enumerate}
    \item Consider a tuple $(b_1,\dots,b_k)$. If $b_{i+1}\leq b_i$ then we can transform
    \[
    (b_1,\dots,b_i,b_{i+1},\dots,b_k)\mapsto (b_1,\dots,b_{i+1},b_i+1,\dots,b_k).
    \]
    This follows by (ii) from Lemma \ref{lemma:braidextend}.
    \item We can also invert this as follows. Suppose that $b_{i+1}>b_i$. If $F_{i+1,b_i}F_{i+2,b_{i+2}}\dots F_{k,b_k}x\in X$ then we can transform
    \[
    (b_1,\dots,b_i,b_{i+1},\dots,b_k)\mapsto (b_1,\dots,b_{i+1}-1,b_i,\dots,b_k).
    \]
    This follows from part (i) of Lemma \ref{lemma:braidextend}.
\end{enumerate}

It thus suffices to show that using the above two manipulations, we can map
\[
(\alpha,a,\beta,b)\mapsto (\alpha,b-c-1,\beta,a+c).
\]
We prove this by inducting on the length of $\beta$. For $\beta$ with length $0$, we immediately get this identity by applying (ii).

Now consider $(\alpha,a,\beta,b)$ where $\beta=(a_{k+1},\dots,a_{m-1})$ has length $m-k-1>0$. We consider two cases. For the first case, suppose that every entry $a_i$ in $\beta$ satisfies $a_i>a$ (equivalently, $c=0$). Then in particular, $a_{m-1}>a$. Note also though that $a_{m-1}>b$, as we assumed that $F_{k+1,a_{k+1}}\dots F_{n,a_n}b_\lambda$ was extendable. Let $\beta'=(a_{k+1},\dots,a_{m-2})$. We then get that by (i)
\[
(\alpha,a,\beta,b)\mapsto (\alpha,a,\beta',b,a_{m-1}+1).
\]
Note that $(\beta',b,a_{m-1}+1)$ is still extendable. Thus by induction, since $c=0$ and the length of $\beta'$ is strictly shorter than the length of $\beta$, we then have
\[
(\alpha,a,\beta',b,a_{m-1}+1)\mapsto (\alpha,b-1,\beta',a,a_{m-1}+1).
\]
We can finally use (ii) to invert the above 
\[
(\alpha,b-1,\beta',a,a_{m-1}+1)\mapsto (\alpha,b-1,\beta,a),
\]
and so obtain the desired expression. Now suppose that $c>0$. Then there exists $a_i\leq a$. Take the minimal such $i$. Set $\beta':=(a_{k+1}+1,a_{k+2}+1,\dots,a_{i-1}+1,a_{i+1},\dots,a_{m-1})$.Repeatedly applying (i) we get
\[
(\alpha,a,\beta,b)\mapsto (\alpha,a,a_i,a_{k+1}+1,a_{k+2}+1,\dots,a_{i-1}+1,a_{i+1},\dots,a_{m-1},b) = (\alpha,a,a_i\beta',b).
\]
We then in turn get
\[
(\alpha,a,a_i,\beta',b)\mapsto (\alpha,a_i,a+1,\beta',b).
\]
Note that $(\beta',b)$ is still extendable (although we increase each $a_j$ for $j<i$ by $1$, we also decrease $c$ by $1$), and so we can again induct to get
\[
(\alpha,a_i,a+1,\beta',b)\mapsto (\alpha,a_i,b-c,\beta',a+c).
\]
We can then use (ii) to invert our previous work and finally obtain
\[
(\alpha,a_i,b-c,\beta',a+c)\mapsto (\alpha,b-c-1,a_i,\beta',a+c)\mapsto (\alpha,b-c-1,\beta,a+c).
\]
Here the final mapping is obtained by repeatedly applying $(ii)$ to move $a_i$ back up to the $i$th index. This completes our induction, and so translating all of this back to our original notation, we obtain by braid relation manipulations that
\[
F_{k,a}F_{k+1,a_{k+1}}\dots F_{m-1,a_{m-1}}F_{m,b}x\in X\implies F_{k,b-c-1}F_{k+1,a_{k+1}}\dots F_{m-1,a_{m-1}}F_{m,a+c}x\in X.
\]
This though is a contradiction, as $z$ was by assumption lowest weight. Thus $z$ is extendable.
\end{proof}

With these two results in hand, we can give a quick proof of Theorem \ref{thm:principalmoreusable}.
\vspace{1em}

\begin{theorem}
    Let $X$ be a connected union of Demazure crystals. Then if $X$ satisfies the assumptions in Theorem \ref{thm:principalmoreusable}, $X$ is a Demazure crystal.
\end{theorem}
\begin{proof}
    By Theorem \ref{thm:characterization}, it suffices to show that $X$ is principal, or, equivalently, that $X$ contains a \textit{unique} lowest weight element. Let $y\in X$ be an extremal element. Let $z$ be an extremal lowest weight element so that $\T{wt}(y)\leq \T{wt}(z)$ (such an element always exists as $X$ is finite). Let $z'\in X$ be the element obtained via our lowest-weight-finding algorithm \ref{rmk:algorithm}. Write $z'=F_{1,k_1}\dots F_{n,k_n}(b_\lambda)$. We show that $z = z'$, and will in turn get that $\T{wt}(z')<\T{wt}(y)$.
    
    First, by Lemma \ref{lemma:extendableexistance}, $z$ is extendable. Write $z = F_{1,a_1}\dots F_{n,a_n}(b_\lambda)$. We inductively show that $a_i = k_i$. The base case $b_\lambda$ is immediate. Suppose that we have shown that $a_i = k_i$ for all $i>j$. Consider $a_j$. Set $x:=F_{j+1,k_{j+1}}\dots F_{n,k_n}b_\lambda$. Then $z = F_{1,a_1}\dots F_{j,a_j}x$ and $z'=F_{1,k_1}\dots F_{j,k_j}x$. 
    
    By assumption, $k_j\leq a_j$. Thus, since $z = F_{1,a_1}\dots F_{j,a_j}x$ is extendable and $F_{j,k_j}x\in X$, we get that
    \[
    F_{1,a_1}\dots F_{j-1,a_{j-1}}F_{j,k_j}x\in X.
    \]
    Since $z$ is lowest weight, this implies that $k_j = a_j$. Thus in fact $z = z'$, and so $z'$ is the unique lowest weight element. 
\end{proof}

\subsection{Further Characterization of the Ideal Property}
The following proposition provides a more usable characterization of the ideal property for Demazure crystals. Note that it is far from necessary though, as many ideal subsets do not satisfy our gluing property (see Theorem \ref{thm:principalmoreusable}). Since we have to prove our gluing property for shuffle tableaux anyways, this characterization is convenient for our work.
\vspace{1em}

\begin{proposition}\label{prop:idealmoreusable}
    Let $X$ be a subset of a (finite, normal $\mathfrak{gl}_n$) crystal. Suppose that $X$ satisfies the following list of properties. 
    \begin{enumerate}
        \item Our gluing property from Theorem \ref{thm:principalmoreusable} holds for $X$.
        \item Let $x\in X$. Suppose that $f_n^*f_{n-1}^*\dots f_{n-k}^*(x)\in X$. Then $f_n^*\dots f_{n-k+1}^*(x)\in X$.
        \item Let $x\in X$. Suppose that $f_n^*f_{n+1}^*\dots f_{n+k}^*(x)\in X$. Then $f_n^*\dots f_{n+k-1}^*(x)\in X$.
    \end{enumerate}
    Then $X$ is an ideal subset.
\end{proposition}

To motivate this proposition, let's compare these conditions with the original definition of an ideal subset. In the diagram below, the ideal condition states that if elements $x$ and $y$ are both in $X$, then $z$ must also be in $X$ (and be nonzero). 

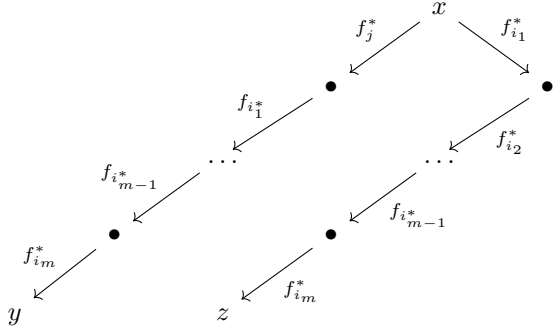
\begin{center}
\begin{tikzcd}
	&&&& x \\
	&&& \bullet && \bullet \\
	&& \dots && \dots \\
	& \bullet && \bullet \\
	y && z
	\arrow["{f_j^*}"', from=1-5, to=2-4]
	\arrow["{f_{i_1}^*}", from=1-5, to=2-6]
	\arrow["{f_{i_{1}^*}}"', from=2-4, to=3-3]
	\arrow["{f_{i_2}^*}", from=2-6, to=3-5]
	\arrow["{f_{i_{m-1}^*}}"', from=3-3, to=4-2]
	\arrow["{f_{i_{m-1}^*}}", from=3-5, to=4-4]
	\arrow["{f_{i_m}^*}"', from=4-2, to=5-1]
	\arrow["{f_{i_m}^*}", from=4-4, to=5-3]
\end{tikzcd}
\end{center}

The ideal property is thus essentially saying that Demazure crystals are `closed under taking subexpressions' of raising operators. Conditions (ii) and (iii) of Proposition \ref{prop:idealmoreusable} also ask for closure under taking subexpressions, but greatly limit the expressions we must consider. The nontrivial claim is that, with extra help from our gluing property, it is enough to consider these simple expressions only. 
\vspace{1em}

\begin{example}
We start by working carefully through the case of two operators.

Suppose that $z:= F_{k,a}F_{k+1,b}(x)$ where $e_i(x) = 0$ for all $i\leq k+1$. Suppose that we want to remove $f_c^*$ from this expression, and show that $y:=F_{k,a}F_{k+1,c+1}f_{c-1}^*\dots f_b^*(x)\in X$.
As in Example \ref{ex:2operators}, this problem splits naturally into three cases.
        \begin{enumerate}
            \item $c\leq a$: By Lemma \ref{lemma:braidextend} part (ii),
            \[
            F_{k,a}F_{k+1,b}(x) = f_{k+1}^*(F_{k,b}F_{k+1,a+1}(x)) \in X.
            \]
            \[
            \implies F_{k,b}F_{k+1,a+1}(x)\in X.
            \]
            Now, we may apply condition (ii) of Proposition \ref{prop:idealmoreusable} to the element $f_{c-1}^* f_{c-2}^*\dots f_b^*F_{k+1, a+1}(x) \in X$, getting that $F_{k,c+1}f_{c-1}^*\dots f_b^*F_{k+1,a+1}(x)\in X$. We can also peel off the $F_{k+1, a+1}$ using condition (ii) again, getting that $F_{k+1,c+1} f^*_{c-1} \dots f_b^*(x) \in X$.
            
            Thus by Lemma \ref{lemma:braidextend} part (i) (letting $\alpha = F_{k,a}$), we get
            \[
            F_{k,a}F_{k+1,c+1}f_{c-1}^*\dots f_b^*(x)\in X,
            \]
            as hoped.
            \item $c>a$ and $F_{k,a}f_{c-1}^*\dots f_b^*(x)\in X$: By Lemma \ref{lemma:techlemidealprop} proven below, we can also assume that $b>a$ in this case. By Lemma \ref{lemma:braidextend} (i),
            \[
            F_{k,a}F_{k+1,b}(x)\in X \implies F_{k,c-1}F_{k+1,a}f_{c-1}^*\dots f_b^*(x)\in X.
            \]
            Then since $b>a$, we can commute each operator in $f_{c-1}^*\dots f_b^*$ left through $F_{k+1,a}$ (each operator decreases by $1$). We thus have that $F_{k,b-1}F_{k+1,a}(x)\in X$. Then by (ii), $F_{k,c}f_{c-2}^*\dots f_{b-1}^*F_{k+1,a}(x)\in X$. Commuting the operators in $f_{c-2}^*\dots f_{b-1}^*$ back right (each operator increases by $1$) we get
            \[
            F_{k,c}F_{k+1,a}(f_{c-1}^*\dots f_b^*)(x)\in X.
            \]
            Then by Lemma \ref{lemma:braidextend} (i), this implies that
            \[
            F_{k,a}F_{k+1,c+1}f_{c-1}^*\dots f_b^*(x)\in X.
            \]
            \item $c>a$ and $F_{k,a}f_{c-1}^*\dots f_b^*x\not\in X$: Set $u:=f_{c-2}^*\dots f_a^*f_{c-1}^*\dots f_b^*x$. By Lemma \ref{lemma:e_i=0cor}, $e_i(u)=0$ for all $c-1\leq i\leq k+1$. Note that we have,
            \[
            F_{k,c-1}F_{k+1,c}(u) = f_k^*f_{k+1}^*(F_{k-1,c-1}F_{k,c}(u)).
            \]
            Thus by (iii), $f_k^*(F_{k-1,c-1}F_{k,c}(u)) = F_{k,c-1}F_{k,c}(u)\in X$. We then have, commuting $F_{k,c}$ out left,
            \[
            F_{k,c-1}F_{k,c}(u) = f_{k-1}^*\dots f_{c-1}^*F_{k,c-1}(u)\in X\implies F_{k,c-1}(u)\in X.
            \]
            Then note that $F_{k+1,c-1}(u)\not\in X$ while $F_{k+1,c+1}(u)\in X$ by (ii). Thus, by (i), 
            \[F_{k,c-1}F_{k+1,c+1}(u) = F_{k,a}F_{k+1,c+1}f_{c-1}^*\dots f_b^*(x)\in X.\]
            \end{enumerate}
        \end{example}

Before proving this proposition, we verify a short technical lemma.
\vspace{1em}

\begin{lemma}\label{lemma:techlemidealprop}
    Let $a,b,c\in \N$ be such that $b\leq a<c$. Then, 
    \[(s_n s_{n-1} \dots s_a)(s_{n+1} s_n \dots s_{c+1}) (s_{c-1} s_{c-2} \dots s_b)\]
    is not a reduced word.
\end{lemma}
\begin{proof}

    We will apply braid relations until two copies of $s_a$ are next to each other in the word, so that the whole word collapses to $0$. Repeatedly applying the first braid relation to move $(s_{n+1} s_n \dots s_{c+1})$ rightwards, we have
    \[
    (s_n \dots s_a)(s_{n+1} \dots s_{c+1}) s_{c-1} \dots s_b = (s_n s_{n-1} \dots s_a)(s_{c-1} \dots s_b)(s_{n+1} \dots s_{c+1})
    \]
    If $a = c-1$, then we are already done. Otherwise, we can then push $(s_{c-1} \dots s_{a+1})$ leftwards, by repeatedly applying the second braid relation: 
    \[(s_n s_{n-1} \dots s_a)(s_{c-1} \dots s_b)(s_{n+1} \dots s_{c+1})
     = (s_{c-2} s_{c-3}\dots s_a)(s_n s_{n-1} \dots s_a)(s_a s_{a-1} \dots s_b)(s_{n+1}s_n \dots s_{c+1})
    \]

    But now the last $s_a$ in the second parenthesis is next to the first $s_a$ in the third parenthesis, so the entire expression is $0$, as desired. 
    
\end{proof}

We can now prove Proposition \ref{prop:idealmoreusable}.

\begin{proof}[Proof of Proposition \ref{prop:idealmoreusable}]
    Set $z:=f_{i_1}^*\dots f_{i_n}^*f_j^*(x)\in X$ where $s_{i_1}\dots s_{i_n}s_j$ is a reduced word for some permutation $w$ and $z, x \in X$ are extremal. 
    
    We seek to show that $y:=f_{i_1}^*\dots f_{i_n}^*(x)\in X\sqcup \{0\}$. Note that $y$ is also extremal, so we can let $w'$ be the permutation such that $w'(\lambda) = \text{wt}(y)$. 
    
    Write $z:=F_{1,a_1}\dots F_{n,a_n}(b_\lambda)$. Then, since $w'$ is a subexpression of $w$ and $\text{len}(w') + 1 = \text{len}(w)$, we can write $y=F_{1,a_1}F_{2, a_2}\dots (F_{k,a+1}f_{a-1}^* f_{a-2}^*\dots f_{a_k}^*)F_{k+1,a_{k+1}}\dots F_{n,a_n}(b_\lambda)$ for some $a$ and $k$ (we just remove some $f_a^*$ from somewhere within our expression for $z$). 

    Set $u:=F_{k+1,a_{k+1}}\dots F_{n,a_n}(b_\lambda)$. We then have 
    $y=F_{1,a_1}\dots F_{k,a+1}(f_{a-1}^*\dots f_{a_k+1}^*f_{a_k}^*(u))$
    and 
    \begin{equation}\label{eq:zequation}
    z= F_{1, a_1}F_{2, a_2} \dots F_{k, a_k} = F_{1,a_1}\dots F_{k,a+1}\mathbf{f_a}^*(f_{a-1}^*\dots f_{a_k+1}^*f_{a_k}^*(u))
    \end{equation}
    We induct on the number of operators $F_{i,a_i}$ on the right side of Equation \ref{eq:zequation}. 
    
    \textbf{Base case: } suppose there is only one $F_{i, a_i}$ in the expression above, so that 
    \[z = F_{k, a+1}(\mathbf{f_a^*}f_{a-1}^* \dots f_{a_k+1}^*f_{a_k}^*)(u)\]
    
    We wish to extract the bolded $f_a^*$ and show that the resulting expression lies in $X$. But this follows directly from our second assumption: applying condition (ii)  to the element $(f_{a-1}^* \dots f_{a_k+1}^*f_{a_k}^*)(u)$, we have 
    \[F_{k, a+1}(f_{a-1}^* \dots f_{a_k+1}^*f_{a_k}^*)(u) \in X\]
    as desired. 
    
    To prove the inductive step, we do some casework on the integers $a_{k-1},a$. 

    \textbf{Case 1: $a\leq a_{k-1}$.}\\
    By Lemma \ref{lemma:symbolicmanipulation}, 
    \begin{equation}\label{eq:commutingf_k}
    F_{k-1,a_{k-1}}F_{k,a_k}(u)=f_k^*
    (F_{k-1,a_{k-1}})(F_{k,a_{k-1}+1}f_{a_{k-1}-1}^*\dots f_{a_k}^*)(u)
    \end{equation}
    Suppose first that $a=a_{k-1}$. Then, commuting the $f_k^*$ leading the right side of Equation \ref{eq:commutingf_k} leftwards using the first braid relation, we get that \[z=f_k^*(y)\] 
    Thus since $z\in X$, $y\in X$ (by either (ii) or (iii)).

    Now, suppose that $a<a_{k-1}$. Note that we can further simplify the above expression to get
    \[
    f_k^*(F_{k-1,a_{k-1}}F_{k,a_{k-1}+1}f_{a_{k-1}-1}^*\dots f_{a_k}^*(u)) = f_k^*(F_{k-1,a_{k}}F_{k,a_{k-1}+1}(u))
    \]
    (Since $a_k \leq a < a_{k-1}$, we just commute $(f_{a_{k-1}-1}^* \dots f_{a_k}^*)$ leftwards using the first braid relation). We then have, by commuting $f_k^*$ left via the first braid relation: 
    \[
    z = f_k^*(F_{1,a_1}\dots F_{k-1,a_{k}}F_{k,a_{k-1}+1}(u))\T{ and }y=f_k^*(F_{1,a_1}\dots F_{k-1,a+1}f_{a-1}^*\dots f_{a_k}^*F_{k,a_{k-1}+1}(u)).
    \]
    Set $v:=F_{k,a_{k-1}+1}(u)$. We then apply condition (ii) to  
    \[
    F_{1,a_1}F_{2,a_2}\dots F_{k-2,a_{k-2}}F_{k-1,a_k}(v)
    \]
    getting that $F_{1,a_1}\dots F_{k-1,a+1}f_{a-1}^*\dots f_{a_k}^*(v)\in X$ by induction. By condition (ii), $F_{k,a+1}f_{a-1}^*\dots f_{a_k}^*(u)\in X$, and so $y$ cannot have a bad $k\mapsto k+1$. Thus $y\in X$.

    \textbf{Case 2: $a>a_{k-1}$}\\
    First, by Lemma \ref{lemma:techlemidealprop}, we can assume also that $a_k>a_{k-1}$ (otherwise $y=0$, because our word for $w'$ was not reduced). We further decompose this case into two subcases.
    
    Suppose first that $F_{k,a_{k-1}}(f_{a-1}^*\dots f_{a_k+1}^*f_{a_k}^*)(u) \in X$. Note that
    \[
    F_{1,a_1}\dots F_{k-1,a_{k-2}}F_{k-1,a_{k-1}}F_{k,a_k}(u) = F_{1,a_1}\dots F_{k-1,a_{k-2}}F_{k-1,a_{k-1}}F_{k,a}(f_{a-1}^*\dots f_{a_k+1}^*f_{a_k}^*(u))
    \]
    Thus, by Lemma \ref{lemma:braidextend} (i), we get
    \[
    F_{1,a_1}\dots F_{k-2,a_{k-2}}F_{k-1,a-1}F_{k,a_{k-1}}(f_{a-1}^*\dots f_{a_k}^*(u))\in X.
    \]
    Braid relating $f_{a-1}\dots f_{a_k}$ left by repeatedly using the \textit{second} braid relation, we obtain that the above expression equals 
    \[
    F_{1,a_1}\dots F_{k-2,a_{k-2}}F_{k-1,a_k-1}F_{k,a_{k-1}}(u).
    \]
    Setting $v:=F_{k,a_{k-1}}(u)$, we then apply condition (ii) to remove $f_{a-1}^*$ from $F_{k-1,a_k-1}$:
    \[
    F_{1,a_1}\dots F_{k-2,a_{k-2}}F_{k-1,a}(f_{a-2}^*\dots f_{a_{k-1}}^*)F_{k,a_{k-1}}(u)\in X.
    \]
    Braid relating the $(f_{a-2}^* \dots f_{a_{k-1}}^*)$ back through using the second braid relation gives us
    \[
    F_{1,a_1}\dots F_{k-2,a_{k-2}}F_{k-1,a}F_{k,a_{k-1}}f_{a-1}^*\dots f_{a_k}^*(u)\in X.
    \]
    By Lemma \ref{lemma:braidextend}, we finally obtain that
    \[
     F_{1,a_1}\dots F_{k-2,a_{k-2}}F_{k-1,a_{k-1}}F_{k,a+1}f_{a-1}^*\dots f_{a_k}^*(u)\in X,
    \]
    as hoped.

    Now suppose that $F_{k,a_{k-1}}f_{a-1}^*\dots f_{a_k}^*(u)\not\in X$. Moreover, inductively let $m$ be such $a_{k-m+1}<a_{k-m+2}<\dots <a_{k-1}<a_k<a$. We start by showing that the first part of this expression is contained in $X$.
    \begin{claim}\label{claim:speccase}
    \[F_{k-m+1,a_{k-m+1}}\dots F_{k-1,a_{k-1}}F_{k,a+1}f_{a-1}\dots f_{a_k}(u)\in X\]
    \end{claim}
    \begin{proof}
    Set 
    \[
    b:= (f_{a-m}^*\dots f_{a_{k-m+1}}^*) \dots (f_{a-3}^*\dots f_{a_{k-3}}^*)(f_{a-2}^*\dots f_{a_{k-1}}^*)(f_{a-1}^*\dots f_{a_k}^*)(u).
    \]
    Then since $a_{k-m+1}<a_{k-m+2}<\dots <a_{k-1}<a_k<a$, by Lemma \ref{lemma:e_i=0cor}, $e_i(b)=0$ for all $a-m<i\leq k$. Moreover, by commuting operators right,
    \[
    F_{k-m+1,a_{k-m+1}}\dots F_{k-1,a_{k-1}}F_{k,a+1}\mathbf{f_{a}}^*f_{a-1}^*\dots f_{a_k}^*(u) = F_{k-m+1,a-m+1}\dots F_{k-1,a-1}F_{k,a}(b).
    \]
    We thus seek to show that $F_{k-m+1,a-m+1}\dots F_{k-1,a-1}F_{k,a+1}(b)\in X$. We first show that $F_{k-m+1,a-m+1}\dots F_{k-1,a-1}(b)\in X$ (ie we remove the entire operator $F_{k,a}$). Note that
    \[
    F_{k-m+1,a-m+1}\dots F_{k-1,a-1}F_{k,a}(b) = f_{k-m+1}^*\dots f_k^*(F_{k-m,a-m+1}\dots F_{k-2,a-1}F_{k-1,a}(b))
    \]
    Thus by hypothesis (iii), we can remove $f_k^*$ to get
    \[
    f_{k-m+1}^*\dots f_{k-1}^*(F_{k-m,a-m+1}\dots F_{k-2,a-1}F_{k-1,a}(b)) = F_{k-m+1,a-m+1}\dots F_{k-1,a-1}F_{k-1,a}(b)\in X.
    \]
    Then by commuting operators in $F_{k-1,a}$ left out to the outside of our expression (and repeating the above argument if necessary), we obtain that $F_{k-m+1,a-m+1}\dots F_{k-1,a-1}(b)\in X$. Thus by (i), since $F_k(a-1)(b)\not\in X$ by assumption and $F_{k,a+1}(b)\in X$ by (ii), we obtain that
    \[
    F_{k-m+1,a-m+1}\dots F_{k-1,a-1}F_{k,a+1}(b)\in X,
    \]
    as hoped.
    \end{proof}
    Now consider adding the operator $F_{k-m,a_{k-m}}$. We break this again into a couple of cases. 

    If $a_{k-m}<a_{k-m+1}$, we can immediately proceed by induction to the next case $a_{k-m}<a_{k-m+1}<\dots <a_{k-1}<a$. 

    Now suppose that $a_{k-m}\geq a_{k-m+1}$. We then have, by the same argument as in Case 1,
    \[
    z = f_{k-m+1}^*(F_{1,a_1}\dots F_{k-m,a_{k-m+1}}F_{k-m+1,a_{k-m}+1}(F_{k-m+2,a_{k-m+2}}\dots F_{k-1,a_{k-1}}F_{k,a}(f_{a-1}^*\dots f_{a_k}^*(u)))).
    \]
    By Claim \ref{claim:speccase}, since $y$ cannot contain a bad $k-m+1$, it is sufficient to show that 
    \[
    F_{1,a_1}\dots F_{k-m,a_{k-m+1}}F_{k-m+1,a_{k-m}+1}(F_{k-m+2,a_{k-m+2}}\dots F_{k-1,a_{k-1}}F_{k,a+1}(f_{a-1}^*\dots f_{a_k}^*(u)))\in X.
    \]
    If now $a_{k-m}+1<a_{k-m+2}$, we reduce to the case argued in Claim \ref{claim:speccase}. If though $a_{k-m}+1\geq a_{k-m+2}$, we can repeat the above argument and then compare $a_{k-m}+2$ with $a_{k-m+3}$. 
    
    Repeating this, we obtain that we either reduce to Case 1 or can replace $a_{k-m},a_{k-m+1},\dots ,a_{k-1},a$ by another sequence $b_{k-m}<b_{k-m+1}<\dots <b_k$ (as described above). We can then again proceed by induction out to the operator $F_{k-m-1,a_{k-m-1}}$.
\end{proof}

\bibliographystyle{plain}
\bibliography{ref}

\end{document}